\documentclass[11pt, oneside]{article}

\usepackage{amsmath, amsthm, amscd, amsxtra}
\usepackage{mathtools}

\usepackage[T1]{fontenc}
\usepackage{newpxmath}      
\usepackage{bm} 
\usepackage{microtype}      
\usepackage[dvipsnames]{xcolor} 
\usepackage{tikz-cd}
\usepackage[shortlabels]{enumitem}
\usepackage{adjustbox}
\usepackage{tabularx}
\usepackage{marginnote}
\usepackage{pdfpages}
\usepackage[
  top=1in,
  textheight=8.5in,
  textwidth=6.3in
]{geometry}

\usepackage{hyperref}
\hypersetup{
  colorlinks=true,
  linkcolor=red,     
  citecolor=blue, 
  urlcolor=blue
}
\usepackage{cleveref}     
\usepackage{bookmark}

\newtheorem{theorem}{Theorem}[section]
\newtheorem{lemma}[theorem]{Lemma}
\newtheorem{proposition}[theorem]{Proposition}
\newtheorem{corollary}[theorem]{Corollary}

\theoremstyle{definition}

\newtheorem{definition}[theorem]{Definition}
\newtheorem{remark}[theorem]{Remark}

\newtheorem*{claim}{Claim}

\renewcommand{\leq}{\leqslant}
\renewcommand{\geq}{\geqslant}
\renewcommand{\emptyset}{\varnothing}
\NewDocumentCommand{\varprojlimone}{ e_ }{%
  \IfValueTF{#1}{%
    \varprojlim_{#1}\!{}^{\smash{\scriptstyle 1}}%
  }{%
    \varprojlim\!{}^{\smash{\scriptstyle 1}}%
  }%
}
\DeclareFontFamily{U}{mathx}{}
\DeclareFontShape{U}{mathx}{m}{n}{<-> mathx10}{}
\DeclareSymbolFont{mathx}{U}{mathx}{m}{n}
\DeclareMathAccent{\widehat}{0}{mathx}{"70}
\DeclareMathAccent{\widecheck}{0}{mathx}{"71}
\begin{document}

\title{Brown's Asymptotic Limit Functor and Proper Homology}

\author{Sumanta Das and Rekha Santhanam}
\date{} 

\maketitle

\begin{abstract}
    In 1974, E. M. Brown introduced the $\wp$-functor to study the proper homotopy groups of an end, suggesting a parallel proper homology theory that has since remained undeveloped. In this paper, we construct this missing proper homology and relate it to proper homotopy by developing a proper Hurewicz theorem. Bypassing abstract pro-categorical machinery, we show that $\wp$ has major advantages over the classical limits $\varprojlim$ and $\varprojlimone$: it is exact, detects pro-triviality, satisfies a cardinality dichotomy, and connects to these classical limits through a 4-term exact sequence. As an application, we prove that the Brown--Grossman proper fundamental group of any open contractible $3$-manifold other than $\mathbb{R}^3$ is uncountable and perfect.
\end{abstract}

\section{Introduction}
In \cite[pp. 489--491]{MR287542}, L. Siebenmann established convenient criteria for a proper map between finite-dimensional, locally finite simplicial complexes to be a proper homotopy equivalence. These criteria were subsequently formalized by F. T. Farrell, L. R. Taylor, and J. B. Wagoner \cite[Theorem 4.1]{MR334226}, and independently by E. M. Brown \cite[p. 43]{MR356041}. Brown's formalization relies on his development of the proper homotopy groups of an end. He showed that these groups can be extracted from an inverse system of classical homotopy groups---associated with a sequence of neighborhoods of the end---by applying a functor he introduced, denoted by $\wp$ \cite[p. 45]{MR356041}. In this paper, we refer to it as the \emph{asymptotic limit functor}.

These proper homotopy groups (often called \emph{Brown--Grossman proper homotopy groups} \cite{MR1361888}) are useful well beyond generalizing Whitehead's theorem to the proper category. For example, in \cite{MR334225}, Brown and Tucker used the $\wp$-functor to define \emph{end-irreducible} $3$-manifolds (the natural analog of boundary-irreducible manifolds), allowing them to prove non-compact version of Waldhausen's rigidity theorem \cite{MR224099} for compact Haken $3$-manifolds. In the same paper, they gave several criteria for deciding when a non-compact $3$-manifold is homeomorphic to a surface times an interval. One of these criteria---called the \emph{Open Collar Theorem} in Tucker's thesis \cite[p. 26]{MR2621405}---gives a clean characterization of $\mathbb{R}^3$: an open, irreducible, contractible $3$-manifold with one end has a trivial proper fundamental group if and only if it is homeomorphic to $\mathbb{R}^3$ \cite[Corollary 3.3]{MR334225}.

A theorem of McMillan \cite{MR137105}, generalizing a construction by Whitehead \cite{zbMATH02532945}, shows that there are uncountably many pairwise non-homeomorphic open, contractible $3$-manifolds. Any such exotic manifold (i.e., one not homeomorphic to $\mathbb{R}^3$) is wild at infinity, meaning in particular that it cannot be end-irreducible \cite[p. 35]{MR2621405}. Consequently, working with these exotic manifolds is inherently difficult unless one imposes additional conditions, such as Brown's notion of ``finite genus at infinity'' \cite{MR511426}. Our first theorem algebraically quantifies this wildness by showing a stark dichotomy: while the proper fundamental group of $\mathbb{R}^3$ is trivial, for any such exotic manifold it is uncountable and perfect. This also provides a sharp contrast to the ordinary fundamental group of a manifold, which is \emph{always countable}.

\begin{theorem}\label{introthm:wildness}
    Every open, contractible $3$-manifold not homeomorphic to $\mathbb{R}^3$ has exactly one end, and its Brown--Grossman proper fundamental group is both uncountable and perfect.
\end{theorem}

The proof of uncountability relies on our algebraic study of the asymptotic limit functor $\wp$, specifically its comparison with the classical limit functors $\varprojlim$ and $\varprojlimone$. To show that this group is perfect, we develop a notion of \emph{proper homology}. While Brown suggested such a theory was possible \cite[p. 44]{MR356041}, it has remained undeveloped in the literature. We establish two equivalent formulations of this proper homology and prove a proper Hurewicz theorem, which provides the crucial step in deducing perfectness. In fact, we prove something stronger in Theorem~\ref{introthm:homologysphere} below.

To achieve these results, the remainder of the introduction pursues three main objectives: first, a systematic algebraic study of the $\wp$-functor; second, the development of proper homology; and finally, the computation of proper homology groups for various spaces, identifying when they reduce to familiar abelian groups such as those of the form $\prod_{n=0}^\infty G_n \big/ \bigoplus_{n=0}^\infty G_n$.

We first demonstrate the structural advantages of $\wp$ over $\varprojlim$ and $\varprojlimone$. It is well known that $\varprojlim$ fails to be right exact, and $\varprojlimone$ measures this failure (see Section~\ref{fe}). Moreover, these two classical limits do not detect whether an inverse system is isomorphic to the trivial system (see Section~\ref{protriv}). In contrast, we establish the following properties for $\wp$.

\begin{theorem}\label{introthm:exactanduncountable}
    The asymptotic limit functor $\wp$, which assigns a group to every countably indexed inverse system of groups, is exact. Moreover, for an inverse system $\underline{G}$, the group $\wp(\underline{G})$ is trivial if and only if $\underline{G}$ is (pro-)isomorphic to the inverse system of trivial groups; otherwise, $\wp(\underline{G})$ is uncountable. Furthermore, when $\underline{G}$ is an inverse system of abelian groups, there is an exact sequence
    \[0 \longrightarrow \varprojlim \underline{G} \longrightarrow \wp(\underline{G}) \xrightarrow{\Phi^\infty} \wp(\underline{G}) \longrightarrow \varprojlimone \underline{G} \longrightarrow 0,\]
    where $\Phi^\infty$ is an induced shift operator.
\end{theorem}

We call $\Phi^\infty$ an ``induced shift operator'' because it mirrors the classical shift operator, whose kernel and cokernel are precisely $\varprojlim \underline{G}$ and $\varprojlimone \underline{G}$. Furthermore, as emphasized above, $\wp(\underline{G})$ always admits a natural group structure. This stands in stark contrast to the first derived limit: if $\underline{G}$ is an inverse system of non-abelian groups, $\varprojlimone\underline{G}$ is generally just a pointed set, whereas $\wp(\underline{G})$ remains a group. To further illustrate how classical limits fall short geometrically---particularly in dimension three---we refer the reader to the discussion following Theorem~\ref{introthm:stringofspheres}.

Next, we develop proper homology groups and relate them to the Brown--Grossman proper homotopy groups. We restrict our attention---for simplicity---to a connected, non-compact, smooth manifold $M$. We also fix an end $e$ of $M$. The end $e$ can be described by a decreasing sequence $\mathcal{U} \coloneqq U_0 \supseteq U_1 \supseteq U_2 \supseteq \cdots$ of open connected subsets of $M$ satisfying $\bigcap_{i\geq 0} U_i = \varnothing$. Using $\mathcal{U}$, it is straightforward to define when a sequence of subsets of $M$ escapes to $e$. Conceptually, a \emph{proper $k$-cycle} for $(M,e)$ is a sequence of $k$-chains escaping to the end $e$ whose boundaries eventually vanish. Similarly, a \emph{proper $k$-boundary} for $(M,e)$ is a sequence of $k$-chains that eventually bounds a sequence of $(k+1)$-chains escaping to the end $e$. The \emph{proper homology group} of $(M,e)$ in degree $k$, denoted $\underline{H}_k(M,e)$, is then defined as the quotient of proper cycles by proper boundaries. We show that this construction has a connection to the Brown--Grossman proper homotopy groups associated to the pair $(M,e)$ via a proper Hurewicz theorem.

\begin{theorem}\label{introthm:properHurewicz}
    The abelianization of the proper fundamental group is isomorphic to the first proper homology group. Moreover, for any $k\geq 1$, if the first $k$ proper homotopy groups vanish, then the $(k+1)$-th proper homotopy group is isomorphic to the $(k+1)$-th proper homology group.
\end{theorem}
It is natural to ask why we do not simply define proper homology as $\wp(\underline{H}_k(\mathcal{U}))$, where $\underline{H}_k(\mathcal{U})$ is the inverse system obtained by applying the $k$-th singular homology to each term of $\mathcal{U}$. Indeed, we prove that the algebraic construction $\wp(\underline{H}_k(\mathcal{U}))$ is isomorphic to our geometric definition of proper homology $\underline{H}_k(M,e)$ given above. To clarify their distinct roles, we will later refer to them as \emph{algebraic} and \emph{geometric} proper homology, respectively. We emphasize that despite their isomorphism, both perspectives are strictly necessary. While the algebraic formulation is powerful for computations, it is not well-suited for establishing properties such as functoriality, proper homotopy invariance, and long exact sequences. Furthermore, proving the degree-one case of Theorem~\ref{introthm:properHurewicz} fundamentally relies on the geometric formulation. Conversely, while the geometric formulation is better equipped to handle these theoretical aspects, we use the algebraic formulation to prove the higher-degree cases of Theorem~\ref{introthm:properHurewicz}.

Finally, we utilize $\wp$ and its various algebraic properties to compute the proper homology groups of several specific spaces. The next theorem proves that open, contractible $3$-manifolds serve as a natural non-compact counterpart to integral homology $3$-spheres.

\begin{theorem}\label{introthm:homologysphere}
    Open, contractible $3$-manifolds have the exact same proper homology groups as $\mathbb{R}^3$ in every degree---namely, $\left.\prod_{m\geq 0} \mathbb Z\right/\bigoplus_{m\geq 0} \mathbb Z$ in degrees $k=0$ and $k=2$, and the trivial group otherwise.
\end{theorem}

Unlike finitely generated abelian groups, uncountable abelian groups lack a general classification theory. Consequently, a major challenge in these computations is showing that the proper homology computed via $\wp$ reduces to a recognizable abelian group, such as one of the form $\prod_{n=0}^\infty G_n/\bigoplus_{n=0}^\infty G_n$. We address this challenge in our next computation, which determines the proper homology of Brown's \emph{string of $k$-spheres} $\underline{\mathbb{S}^k}$---the space $[0,\infty)$ together with a distinct $k$-sphere attached at each integer point. This space is required to define the $k$-th Brown--Grossman proper homotopy group, playing a role analogous to that of a pointed $k$-sphere in defining the $k$-th ordinary homotopy group. A straightforward calculation shows that the $0$-th proper homology of $\underline{\mathbb{S}^k}$ is $\prod_{n \ge 0} \mathbb{Z} \big/ \bigoplus_{n \ge 0} \mathbb{Z}$, and that all other proper homology groups vanish except in degree $k$. By analogy with the singular homology of a single $k$-sphere, where $H_0(\mathbb{S}^k) \cong H_k(\mathbb{S}^k)$, one might naturally expect the $k$-th proper homology of $\underline{\mathbb{S}^k}$ to be isomorphic to its $0$-th proper homology. However, establishing this expected isomorphism is far from direct; in particular, a potential approach via a proper analogue of the Mayer--Vietoris sequence fails due to strict connectivity obstructions (as detailed in Section~\ref{section:stringofspheres}). We overcome this challenge by appealing to the classification of algebraically compact groups.

\begin{theorem}\label{introthm:stringofspheres}
    The $k$-th proper homology group $\underline{H}_k(\underline{\mathbb{S}^k})$ is a torsion-free, algebraically compact group. Its maximal divisible subgroup is isomorphic to $\bigoplus^{\mathfrak{c}} \mathbb{Q}$, and its $p$-th torsion-free number is $\mathfrak{c}$ for every prime $p$. Consequently, $\underline{H}_k(\underline{\mathbb{S}^k})$ is isomorphic to $\prod_{n \ge 0} \mathbb{Z} \big/ \bigoplus_{n \ge 0} \mathbb{Z}$.
\end{theorem}
We conclude by illustrating how the inverse limit fails to capture the geometric wildness that the asymptotic limit successfully detects, at least in dimension three. To make this precise, we first establish some notation. Consider a nice space $X$ (see Definition~\ref{def:nice}; for example, any non-compact, connected topological manifold) with exactly one end $e$. This end can be represented by a nested sequence $\mathcal{U} = U_0 \supseteq U_1 \supseteq U_2 \supseteq \cdots$ of open, path-connected subsets of $X$ with compact complements such that $\bigcap_{n\geq 0} U_n = \varnothing$. Let $r\colon [0,\infty)\to X$ be a proper ray. After omitting an initial segment and reparametrizing, we may assume that $r([n,\infty))\subseteq U_n$ for every $n \ge 0$ \cite[Exercise 3.3.18]{MR3598162}. Consider the inverse system $\underline{\pi}_1(\mathcal{U},r)=\{\pi_1(U_n,r(n))\}$, where the bonding homomorphisms $\pi_1(U_{n+1}, r(n+1))\to \pi_1(U_n, r(n))$ are induced by inclusion followed by a change of basepoint along the path $r\vert{}_{[n,n+1]}$. We aim to compare how the inverse and asymptotic limits of this system capture the wildness at infinity.

The first group is called the \emph{fundamental group at infinity} of $X$ based at $r$, defined as $\varprojlim\underline \pi_1(\mathcal{U},r)$. If $\mathcal{U}'$ is another nested sequence describing the same end, then $\varprojlim\underline \pi_1(\mathcal{U}',r)$ (possibly after modifying $r$ as above) is isomorphic to $\varprojlim\underline \pi_1(\mathcal{U},r)$ \cite[\S 1]{MR1396771}. However, the dependence on $r$ is more delicate: if $r'\colon [0,\infty)\to X$ is another proper ray satisfying similar containment properties with respect to $\mathcal{U}$, then in general $\varprojlim\underline \pi_1(\mathcal{U},r)$ is not isomorphic to $\varprojlim\underline \pi_1(\mathcal{U},r')$ (see, e.g., \cite[Example 16.2.4]{MR2365352} or \cite[p. 132]{MR676973}). On the other hand, if $r$ and $r'$ are properly homotopic---which is guaranteed if we additionally assume that $X$ is \emph{strongly connected at infinity}---then these two groups are isomorphic \cite[Corollary 3.4.22]{MR3598162}. The famous Whitehead manifold provides an example of a manifold that is not strongly connected at infinity; see \cite[Proposition 16.1.2]{MR2365352} or \cite[Proposition 3.4.37]{MR3598162}.

In contrast, the Brown--Grossman proper fundamental group of $X$ based at $r$ (or more precisely, in Brown's terminology, based at the \emph{germ} of $r$), which is isomorphic to $\wp\left(\underline\pi_1(\mathcal{U},r)\right)$ \cite[p. 45]{MR356041}, does not depend on the choice of the ray $r$. That is, for any two proper rays in $X$---regardless of whether they are properly homotopic---their corresponding Brown--Grossman proper fundamental groups are isomorphic (see \cite[p. 42]{MR356041} or \cite[p. 8]{MR2621405}). We note that when $X=\mathbb R^n$ for $n\geq 3$, both the proper fundamental group and the fundamental group at infinity are trivial. This is because, in this setting, the inverse system $\underline \pi_1(\mathcal{U},r)$ is (pro-)isomorphic to the inverse system of trivial groups $\underline{\bm{1}}$ (see Theorem~\ref{proinv} and \cite[Example 3.2.13 and Exercise 3.4.3]{MR3598162}).

A closely related notion in the literature is that of simple connectivity at infinity: we say $X$ is \emph{simply connected at infinity} if for any compact subset $C\subseteq X$, there exists a compact subset $D\subseteq X$ containing $C$ such that every loop in $X\setminus D$ is contractible in $X\setminus C$. According to \cite[Proposition 3.4.36]{MR3598162}, $X$ is simply connected at infinity if and only if $\underline \pi_1(\mathcal{U},r)$ is (pro-)isomorphic to $\underline{\bm{1}}$. Furthermore, by Theorem~\ref{introthm:exactanduncountable}, it follows that $X$ is simply connected at infinity if and only if the proper fundamental group of $X$ is trivial. By a celebrated result---often referred to as \emph{Stallings' characterization of $\mathbb R^n$}---for any $n \geq 3$, an open, contractible $n$-manifold $M$ is homeomorphic to $\mathbb{R}^n$ if and only if $M$ is simply connected at infinity. This characterization was established by Stallings \cite{MR149457} and Luft \cite{MR221486} for $n \geq 5$, Freedman \cite{MR679066} for $n = 4$, and Edwards \cite{MR150745}, Wall \cite{MR181993}, and Brown--Tucker \cite{MR334225} for $n = 3$ (modulo the $3$-dimensional Poincaré conjecture, which was subsequently proved by Perelman).

While both of these groups are trivial for $\mathbb{R}^n$, their behavior differs significantly on exotic manifolds. The famous Whitehead manifold \cite{zbMATH02532945} is not simply connected at infinity, and hence is not homeomorphic to $\mathbb R^3$. In particular, its proper fundamental group is non-trivial. However, the fundamental group at infinity of the Whitehead manifold is trivial \cite[\S 1]{MR1396771}. More generally, any contractible genus-one $3$-manifold $\mathcal W$ (a notion introduced in \cite{MR137105}) is not simply connected at infinity, yet its fundamental group at infinity is trivial \cite[Remark 2.5]{MR4705657}. Any such $\mathcal W$ can be expressed as an ascending union of solid tori $\{T_i\}$ such that $T_i$ is contractible in $T_{i+1}$, and the minimal number of intersection points between the core of $T_i$ and a meridian disc of $T_{i+1}$ is at least two \cite[Theorem 2.8]{MR3739201}. In contrast, there are infinitely many open, contractible $3$-manifolds that do have a non-trivial fundamental group at infinity \cite[Appendix C]{MR4705657}. In conclusion, within the class of open, contractible $3$-manifolds, the triviality of the fundamental group at infinity does not guarantee that the manifold is homeomorphic to $\mathbb{R}^3$. On the other hand, the triviality of the proper fundamental group definitively characterizes the manifold as homeomorphic to $\mathbb{R}^3$.

\subsection{Relation to Previous Work}
While the majority of the results in this paper are, to the best of our knowledge, self-contained developments of the $\wp$-functor, Theorem~\ref{introthm:exactanduncountable} intersects directly with several known results. We detail these connections to the existing literature below.

The algebraic properties of the $\wp$-functor were first investigated in unpublished work of J. Chipman \cite{Chipman1976Isomorphism,Chipman1976Presentations}. Subsequently, Hernández-Paricio \cite{MR1324568} developed a categorical treatment of the functor in the setting of pro-categories and towers, using tools such as Gabriel--Zisman calculus of fractions, representable functors, and algebraic structures arising from endomorphism monoids, near-rings, and rings. However, the abstract categorical machinery used in \cite{MR1324568} makes it difficult to extract computable invariants for applications.

Our paper diverges from the categorical abstraction of previous literature by providing a concrete, element-wise algebraic toolkit for the $\wp$-functor, making these invariants directly computable. While some of the results in Sections~\ref{fe} and~\ref{protriv} parallel those established in \cite{MR1324568}, we reconstruct them here using explicit maps without invoking pro-categories. Specifically, our Theorem~\ref{fun}, Theorem~\ref{exactness}, Theorem~\ref{proinv}, Corollary~\ref{protrivdet}, and Theorem~\ref{colimrep} serve as explicit counterparts to the abstract categorical properties established in \cite[pp. 11--12, Propositions 4.2, 4.3, 4.6, Theorem 4.4, and Corollary 7.16]{MR1324568}, translating abstract existence theorems into tangible algebraic formulas. Notably, the quotient groups utilized in Theorem~\ref{colimrep} coincide exactly with Grossman's reduced power construction \cite[pp. 623--624]{MR413097}. However, while Grossman evaluates the asymptotic limit $\wp(\underline G)$ via an inverse limit of these sequence-based equivalence classes, we capture $\wp(\underline G)$ via a direct limit over diverging non-decreasing index sequences.

A central algebraic objective of this paper is to demonstrate the structural advantages of the asymptotic limit functor $\wp$ over the classical inverse limit $\varprojlim$ and its first derived functor $\varprojlimone$. For this reason, we systematically contrast the behavior of $\wp$ against the known limitations of these classical functors under identical hypotheses.

A fundamental algebraic property of $\varprojlimone$ is its cardinality dichotomy: for a tower of countable abelian groups, $\varprojlimone$ either vanishes (precisely when the Mittag-Leffler condition is satisfied) or is uncountably infinite. This uncountability phenomenon was initially established by Gray \cite{MR196743} and subsequently by Jensen \cite{MR407091}, while Geoghegan \cite{MR560787} later provided an elegant proof of this dichotomy by applying the Baire category theorem. We establish that the $\wp$-functor exhibits an analogous cardinality dichotomy (Section~\ref{dichotomy}). However, a direct application of the Baire category theorem fails in this setting. To secure the dichotomy for the $\wp$-functor, our proof fundamentally relies on first establishing the result for a suitably chosen subgroup.

To construct a direct homological bridge between the classical limits and the asymptotic limit, we introduce a $4$-term exact sequence (Theorem~\ref{4ses}). Within this sequence, the kernel representation ($\varprojlim\underline G\cong \ker \Phi^\infty$) was proven using an abstract categorical framework in \cite[Theorem 4.15]{MR1324568}. However, the cokernel representation ($\varprojlimone\underline G\cong \operatorname{coker}(\Phi^\infty)$)---while briefly mentioned in \cite[Remark 4.17]{MR1324568}---was provided without proof or explicit formulation. Our treatment bypasses this abstract categorical machinery entirely. By relying instead on direct coordinate algebra, we employ a fundamentally different, and fully explicit, proof strategy.

\subsection{Organization of the Paper}
The paper is organized as follows. In Section~\ref{section2}, we recall the necessary background on the $\wp$-functor and prove the first half of Theorem~\ref{introthm:exactanduncountable}. Specifically, we show that $\wp$ is exact, detects pro-triviality, and satisfies a cardinality dichotomy. In Section~\ref{connection}, we prove the remainder of Theorem~\ref{introthm:exactanduncountable} by constructing the 4-term exact sequence. Section~\ref{homology} introduces the geometric version of proper homology and proves that it is isomorphic to our algebraic construction. We then use this machinery in Section~\ref{section:Hur} to prove the proper Hurewicz theorem (Theorem~\ref{introthm:properHurewicz}). Finally, Section~\ref{section:computation} is dedicated to explicit computations, where we calculate the proper homology of open contractible $3$-manifolds (Theorems~\ref{introthm:wildness} and~\ref{introthm:homologysphere}) and Brown's string of $k$-spheres (Theorem~\ref{introthm:stringofspheres}).

\section{Definition and Properties}\label{section2}
Throughout this paper, we restrict our attention to inverse systems indexed by the non-negative integers. Such an \emph{inverse system of groups} (or \emph{tower}) is a collection $\underline{G} = \{ G_n, p^m_n \}_{m \ge n \ge 0}$, where each $G_n$ is a group and the maps $p^m_n \colon G_m \to G_n$ are homomorphisms such that $p^n_n = \mathrm{id}_{G_n}$ for all $n$, and $p^m_n \circ p^k_m = p^k_n$ for all $k \ge m \ge n$.

 Let $\mathsf{Grp}$ denote the category of groups, and let $\mathsf{InvGrp}$ denote the category of inverse systems of groups. The objects of $\mathsf{InvGrp}$ are inverse systems $\underline{G} = \{G_n, p_n^m\}_{m \ge n \ge 0}$, and the morphisms $\Phi \colon \underline{G} \to \underline{H}$ (where $\underline{H} = \{H_n, q_n^m\}$) are sequences of group homomorphisms $\Phi = \{\phi_n \colon G_n \to H_n\}_{n \ge 0}$ such that $q_n^m \circ \phi_m = \phi_n \circ p_n^m$ for all $m \ge n \ge 0$.

\begin{definition} \cite[p. 44]{MR356041}
    Let $\underline{G} = \{ G_n, p^m_n \}_{m \ge n \ge 0}$ be an inverse system of groups, and let $\mathscr{S}_{\underline{G}}$ denote the set of all sequences $\{ g_{k(n)} \}_{n \ge 0}$ such that $g_{k(n)} \in G_{k(n)}$ and $k(n) \to \infty$ as $n \to \infty$. Two sequences $\{ g_{k(n)} \}$ and $\{ g'_{l(n)} \}$ in $\mathscr{S}_{\underline{G}}$ are \emph{asymptotically equivalent}, denoted $\{ g_{k(n)} \} \sim \{ g'_{l(n)} \}$, if there exists a sequence of integers $m(n) \to \infty$ with $m(n) \le \min\{k(n), l(n)\}$ such that
    \[
    p^{k(n)}_{m(n)}(g_{k(n)}) = p^{l(n)}_{m(n)}(g'_{l(n)})
    \]
    for all sufficiently large $n$. Let $\wp(\underline{G}) \coloneqq \mathscr{S}_{\underline{G}} / {\sim}$ denote the set of equivalence classes of these sequences.
\end{definition}
Let $[\{g_{k(n)}\}]$ and $[\{h_{l(n)}\}]$ be two elements in $\wp(\underline{G})$. To define their product, we first set $m(n) = \min\{k(n), l(n)\}$ for each $n$. We project both sequences down to $G_{m(n)}$ and multiply them inside $G_{m(n)}$, defining the product as: \begin{align*}
    [\{g_{k(n)}\}] \cdot [\{h_{l(n)}\}] &\coloneqq \left[ \left\{ p^{k(n)}_{m(n)}(g_{k(n)}) \cdot p^{l(n)}_{m(n)}(h_{l(n)}) \right\}_{n \ge 0} \right].
\end{align*}

To see that this product is well-defined, we first show that it does not depend on the choice of the bounding sequence. Suppose we choose another sequence $m'(n) \to \infty$ such that $m'(n) \le \min\{k(n), l(n)\}$. Since $m(n) = \min\{k(n), l(n)\}$, we clearly have $m'(n) \le m(n)$. Because $p^{m(n)}_{m'(n)}$ is a group homomorphism, projecting the terms of the $m(n)$-sequence down to $m'(n)$ yields
\begin{align*}
    p^{m(n)}_{m'(n)}\left( p^{k(n)}_{m(n)}(g_{k(n)}) \cdot p^{l(n)}_{m(n)}(h_{l(n)}) \right) 
    &= p^{m(n)}_{m'(n)}\left( p^{k(n)}_{m(n)}(g_{k(n)}) \right) \cdot p^{m(n)}_{m'(n)}\left( p^{l(n)}_{m(n)}(h_{l(n)}) \right) \\
    &= p^{k(n)}_{m'(n)}(g_{k(n)}) \cdot p^{l(n)}_{m'(n)}(h_{l(n)}).
\end{align*}
This is exactly the sequence we would obtain if we had used $m'(n)$ to define the product initially. By the definition of $\sim$, the two resulting sequences are equivalent, meaning the product is well-defined with respect to the choice of index sequence.

Next, we show that the product is independent of the choice of representatives. Suppose $\{g_{k(n)}\} \sim \{g'_{k'(n)}\}$ and $\{h_{l(n)}\} \sim \{h'_{l'(n)}\}$. By definition, there exist integer sequences $m_1(n) \to \infty$ and $m_2(n) \to \infty$ such that 
\[
p^{k(n)}_{m_1(n)}(g_{k(n)}) = p^{k'(n)}_{m_1(n)}(g'_{k'(n)}) \quad \text{and} \quad p^{l(n)}_{m_2(n)}(h_{l(n)}) = p^{l'(n)}_{m_2(n)}(h'_{l'(n)})
\]
for all sufficiently large $n$. Set $r(n) \coloneqq \min\{ m_1(n), m_2(n) \} \to \infty$. Projecting the product sequences down to $G_{r(n)}$, we have for all sufficiently large $n$:
\begin{align*}
p^{k(n)}_{r(n)}(g_{k(n)}) \cdot p^{l(n)}_{r(n)}(h_{l(n)})
&= p^{m_1(n)}_{r(n)}\bigl(p^{k(n)}_{m_1(n)}(g_{k(n)})\bigr) \cdot p^{m_2(n)}_{r(n)}\bigl(p^{l(n)}_{m_2(n)}(h_{l(n)})\bigr) \\
&= p^{m_1(n)}_{r(n)}\bigl(p^{k'(n)}_{m_1(n)}(g'_{k'(n)})\bigr) \cdot p^{m_2(n)}_{r(n)}\bigl(p^{l'(n)}_{m_2(n)}(h'_{l'(n)})\bigr) \\
&= p^{k'(n)}_{r(n)}(g'_{k'(n)}) \cdot p^{l'(n)}_{r(n)}(h'_{l'(n)}).
\end{align*}
Hence, the resulting product sequences are equivalent. This establishes that the operation is well-defined, leading to the following theorem.

\begin{theorem}\textup{\cite[p. 44]{MR356041}}
    If $\underline{G}$ is an inverse system of groups, then $\wp(\underline{G})$ admits a group structure where the product of two elements $[\{g_{k(n)}\}]$ and $[\{h_{l(n)}\}]$ is given by
    \[
    [\{g_{k(n)}\}] \cdot [\{h_{l(n)}\}] \coloneqq \left[ \left\{ p^{k(n)}_{m(n)}(g_{k(n)}) \cdot p^{l(n)}_{m(n)}(h_{l(n)}) \right\}_{n \ge 0} \right],
    \]
    for any sequence of integers $m(n) \to \infty$ satisfying $m(n) \le \min\{k(n), l(n)\}$.
\end{theorem}
\begin{corollary}
    If $\underline{G}$ is an inverse system of abelian groups, then $\wp(\underline{G})$ is an abelian group.
\end{corollary}

The following theorem appears in Brown's paper \cite[p. 46]{MR356041} without proof. We include a proof here for completeness. The theorem also illustrates that, unlike the inverse limit and the first derived limit, the asymptotic limit may be uncountable even in very simple situations.
\begin{theorem}[Asymptotic Limit of Constant System]\label{constant}
    Let $G$ be a group, and consider the constant inverse system $\underline{G}$ defined by 
    \[
    G_n = G \quad \text{and} \quad p^m_n = \mathrm{id}_G \quad \text{for all } m \ge n \ge 0.
    \]Then the map $\Phi \colon \left.\prod_{n \ge 0} G \right/ \bigoplus_{n \ge 0} G \longrightarrow \wp(\underline{G})$ defined by sending the equivalence class of a sequence $\{ g_n \}_{n \ge 0}$ to its asymptotic equivalence class $[\{ g_{k(n)=n} \}]$ in $\wp(\underline{G})$ is an isomorphism.
\end{theorem}

    \begin{proof}To see that $\Phi$ is well-defined, suppose $\{ g_n \}$ and $\{ g'_n \}$ represent the same class in the quotient. This means $g_n = g'_n$ for all $n \ge N$ for some integer $N$. By choosing the bounding sequence $m(n) = n$, we see that $p^n_n(g_n) = g_n = g'_n = p^n_n(g'_n).$ for all $n \ge N$. Thus, $\{g_n\} \sim \{g'_n\}$, meaning $\Phi([\{ g_n \}]) = \Phi([\{ g'_n \}])$, so the map is independent of the choice of representative.

        Now, we check that $\Phi$ is a group homomorphism. Let $\{ g_n \}$ and $\{ h_n \}$ be sequences in $\prod G$. Applying $\Phi$ to their product yields:
        \[
        \Phi([\{ g_n \}] \cdot [\{ h_n \}]) = \Phi([\{ g_n h_n \}]) = [\{ g_n h_n \}].
        \]
        In $\wp(\underline{G})$, the product of $[\{ g_n \}]$ and $[\{ h_n \}]$ using the diagonal bounding sequence $m(n) = n$ gives exactly $[\{ p^n_n(g_n) p^n_n(h_n) \}] = [\{ g_n h_n \}]$. Thus, $\Phi$ preserves the group operation.

        To prove $\Phi$ is injective, suppose $\Phi([\{ g_n \}]) = [\{ 1 \}]$, the identity element of $\wp(\underline{G})$. By definition, there exists an index sequence $m(n) \to \infty$ with $m(n) \le n$ such that for all sufficiently large $n$:
        \[
        p^n_{m(n)}(g_n) = p^n_{m(n)}(1) = 1.
        \]
        Because every bonding map in the trivial system is the identity, $p^n_{m(n)}(g_n) = g_n$. Therefore, $g_n = 1$ for all sufficiently large $n$. This implies $\{ g_n \} \in \bigoplus_{n \ge 0} G$, meaning $[\{ g_n \}]$ is the identity element of the quotient. Hence, $\ker(\Phi)$ is trivial and $\Phi$ is injective.

        To prove $\Phi$ is surjective, let $[\{g_{k(n)}\}] \in \wp(\underline{G})$ be an arbitrary asymptotic class, where $k(n) \to \infty$. We construct a diagonal sequence $\{g'_n\}_{n \ge 0} \in \prod G$ by defining:
        \[
        g'_n \coloneqq g_{k(n)} \quad \text{for all } n \ge 0.
        \]
        To show $\{g'_n\} \sim \{g_{k(n)}\}$, we define the bounding index sequence $m(n) \coloneqq \min\{ n, k(n) \}$. Since both $n \to \infty$ and $k(n) \to \infty$, it follows that $m(n) \to \infty$. Projecting both sequences to level $m(n)$ using the identity bonding maps gives $p^n_{m(n)}(g'_n) = g'_n = g_{k(n)}$ and $ p^{k(n)}_{m(n)}(g_{k(n)}) = g_{k(n)}.$ The projections are identically equal for all $n$. Therefore, $\{g'_n\} \sim \{g_{k(n)}\}$, which means $\Phi([\{g'_n\}]) = [\{g_{k(n)}\}]$. This proves $\Phi$ is surjective. 
        
        Since $\Phi$ is a bijective homomorphism, it is an isomorphism.
    \end{proof}

\begin{remark}
    We note a typographical error in the survey by Porter \cite[p. 137]{MR1361888} regarding the definition of Brown's $\wp$ functor. The original text defines equivalence as follows:
    \begin{quote}
        ``Given two such sequences, $\{g_{k(n)}\}$ and $\{g'_{l(n)}\}$, we say that they are equivalent if there is a third sequence $m(n)$, $m(n)\to\infty$ as $n\to\infty$, $m(n)\leq \min(k(n),l(n))$, and $p^{k(n)}_{m(n)} g_{k(n)}=p^{l(n)}_{m(n)} g'_{l(n)}$ for all $n\in\mathbb{N}$.''
    \end{quote}
    The final equality should only be required to hold for all but finitely many $n$. If it were strictly enforced for all $n\in\mathbb{N}$, a calculation analogous to Theorem~\ref{constant} would yield $\prod_{n=0}^\infty G$ rather than the correct quotient $\prod_{n=0}^\infty G \big/ \bigoplus_{n=0}^\infty G$ established by Brown.
\end{remark}

\subsection{Asymptotic Limit as an Exact Functor}\label{fe}
We now establish some categorical properties of $\wp$. While the inverse limit fail to preserve right exactness, the asymptotic limit exhibits no such failure. In what follows, we show that $\wp$ is not only a well-defined covariant functor, but also preserves short exact sequences.
\begin{theorem}[Functoriality]\label{fun}
    The assignment $\wp$ defines a covariant functor from the category $\mathsf{InvGrp}$ to the category $\mathsf{Grp}$. Specifically, $\wp$ maps an inverse system $\underline{G}$ to the group $\wp(\underline{G})$, and it maps a morphism of inverse systems $\Phi = \{\phi_n\}_{n \ge 0} \colon \underline{G} \to \underline{H}$ to the induced group homomorphism $\wp(\Phi) \colon \wp(\underline{G}) \to \wp(\underline{H})$ given by
    \[
    \wp(\Phi)\left([\{g_{k(n)}\}]\right) \coloneqq \left[ \{ \phi_{k(n)}(g_{k(n)}) \}_{n \ge 0} \right].
    \]
\end{theorem}

\begin{proof}
    To verify that $\wp(\Phi)$ is well-defined, suppose $\{g_{k(n)}\} \sim \{g'_{l(n)}\}$. By definition, there exists a sequence $m(n) \to \infty$ with $m(n) \le \min\{k(n), l(n)\}$ such that $ p^{k(n)}_{m(n)}(g_{k(n)}) = p^{l(n)}_{m(n)}(g'_{l(n)})$ for all sufficiently large $n$. Applying the homomorphism $\phi_{m(n)}$ to both sides and utilizing the commutativity condition $q^b_a \circ \phi_b = \phi_a \circ p^b_a$ of the morphism $\Phi$, we obtain:
    $$q^{k(n)}_{m(n)} \left( \phi_{k(n)}(g_{k(n)}) \right) 
        = \phi_{m(n)}\left(p^{k(n)}_{m(n)}(g_{k(n)})\right) 
        = \phi_{m(n)}\left(p^{l(n)}_{m(n)}(g'_{l(n)})\right) 
        = q^{l(n)}_{m(n)} \left( \phi_{l(n)}(g'_{l(n)}) \right)$$for all sufficiently large $n$. This establishes that $\{ \phi_{k(n)}(g_{k(n)}) \} \sim \{ \phi_{l(n)}(g'_{l(n)}) \}$, proving that the map is independent of the choice of representative sequence.

    A routine verification, using the coordinate-wise definition of the product in $\wp$ and the fact that each $\phi_n$ is a homomorphism, confirms that $\wp(\Phi)$ is indeed a group homomorphism. Finally, it is evident directly from the definition that $\wp(\operatorname{id}_{\underline{G}}) = \operatorname{id}_{\wp(\underline{G})}$, and for morphisms $\Phi \colon \underline{G} \to \underline{H}$ and $\Psi \colon \underline{H} \to \underline{K}$, we have $\wp(\Psi \circ \Phi) = \wp(\Psi) \circ \wp(\Phi)$. Thus, $\wp$ is a well-defined functor.
\end{proof} 

 Let $\underline{A} = \{A_n, p^m_n\}$, $\underline{B} = \{B_n, q^m_n\}$, and $\underline{C} = \{C_n, r^m_n\}$ be three inverse systems of groups, and let $\Phi \colon \underline{A} \to \underline{B}$ and $\Psi \colon \underline{B} \to \underline{C}$ be morphisms of inverse systems given by level-wise maps $\phi_n$ and $\psi_n$. Suppose for every integer $n \ge 0$, the following sequence of groups is exact:
    \[
    1 \to A_n \xrightarrow{\phi_n} B_n \xrightarrow{\psi_n} C_n \to 1.
    \]
If $A_n, B_n,$ and $C_n$ are abelian groups for all $n \ge 0$, then the failure of right exactness of the inverse limit functor $\varprojlim$ is encoded by the first right derived functor, $\varprojlimone$. More precisely, we have an exact sequence of groups \cite[Proposition 2.12 (ii)]{MR1410261} \cite[Appendix]{MR488016}: 
\begin{equation}\label{rigtexactfail}
    1\to \varprojlim \underline A\to \varprojlim \underline B\to \varprojlim \underline C\to \varprojlimone \underline A\to \varprojlimone \underline B\to \varprojlimone\underline C\to 1.
\end{equation}In the non-abelian setting, we can only expect this to be an exact sequence of pointed sets \cite[Proposition 2.3, Chapter 9]{MR365573} \cite[p. 168]{MR676973}.  In contrast, the asymptotic limit functor yields an exact sequence of groups even if $\underline A, \underline B,$ and $\underline C$ are systems of non-abelian groups.
\begin{theorem}[Exactness]\label{exactness}The induced sequence of asymptotic groups is also exact:
    \[
    1 \to \wp(\underline{A}) \xrightarrow{\wp(\Phi)} \wp(\underline{B}) \xrightarrow{\wp(\Psi)} \wp(\underline{C}) \to 1.
    \]
\end{theorem}

\begin{proof}
    The proof proceeds in three standard steps. The trivial element in any asymptotic group will be represented by the diagonal trivial sequence $[\{1_n\}]$.
\medskip

    \emph{Left Exactness:} Let $[\{a_{k(n)}\}] \in \wp(\underline{A})$ be an arbitrary class such that $\wp(\Phi)([\{a_{k(n)}\}]) = [\{1_n\}]$. By definition of the induced map, this means $[\{\phi_{k(n)}(a_{k(n)})\}] = [\{1_n\}]$ in $\wp(\underline{B})$. Thus, there exists a bounding sequence $m(n) \to \infty$ with $m(n) \le \min\{k(n), n\}$ such that for all sufficiently large $n$, we have $ q^{k(n)}_{m(n)}(\phi_{k(n)}(a_{k(n)})) = q^n_{m(n)}(1_n) = 1_{m(n)}.$ Because $\Phi$ is a morphism of inverse systems, it strictly commutes with the bonding maps. Swapping the order of operations yields $\phi_{m(n)}(p^{k(n)}_{m(n)}(a_{k(n)})) = 1_{m(n)}.$ By the level-wise exactness of the original sequence, $\phi_{m(n)} \colon A_{m(n)} \to B_{m(n)}$ is injective. This forces the argument to be the identity: $ p^{k(n)}_{m(n)}(a_{k(n)}) = 1_{m(n)}$ for all sufficiently large $n$. Because $p^n_{m(n)}(1_n) = 1_{m(n)}$, this is precisely the condition that $\{a_{k(n)}\} \sim \{1_n\}$ in $\underline{A}$. Therefore, $[\{a_{k(n)}\}] = [\{1_n\}]$, and $\wp(\Phi)$ is injective.

\medskip

    \emph{Middle Exactness:}
    Because $\psi_n \circ \phi_n = 1$ for every $n$, functoriality immediately implies $\wp(\Psi) \circ \wp(\Phi) = 1$, yielding $\operatorname{im}(\wp(\Phi)) \subseteq \ker(\wp(\Psi))$.
    
    Conversely, let $[\{b_{k(n)}\}] \in \ker(\wp(\Psi))$. By definition, $[\{\psi_{k(n)}(b_{k(n)})\}] = [\{1_n\}]$ in $\wp(\underline{C})$. There exists a bounding sequence $m(n) \to \infty$ with $m(n) \le \min\{k(n), n\}$ such that for all sufficiently large $n$, we have $r^{k(n)}_{m(n)}(\psi_{k(n)}(b_{k(n)})) = 1_{m(n)}.$
    Using the commuting property of $\Psi$, this becomes $\psi_{m(n)}(q^{k(n)}_{m(n)}(b_{k(n)})) = 1_{m(n)}$. Let $b'_{m(n)} \coloneqq q^{k(n)}_{m(n)}(b_{k(n)}) \in B_{m(n)}$. Since $\psi_{m(n)}(b'_{m(n)}) = 1_{m(n)}$, the exactness at level $m(n)$ guarantees that $b'_{m(n)} \in \operatorname{im}(\phi_{m(n)})$. Thus, for all sufficiently large $n$, there exists an element $a_{m(n)} \in A_{m(n)}$ such that $\phi_{m(n)}(a_{m(n)}) = b'_{m(n)}$. (For any initial terms where this fails, we simply set $a_{m(n)} = 1_{m(n)}$). 
    
    Because $m(n) \to \infty$, the sequence $\{a_{m(n)}\}$ defines a valid class in $\wp(\underline{A})$. Applying the forward map gives $\wp(\Phi)([\{a_{m(n)}\}]) = [\{\phi_{m(n)}(a_{m(n)})\}] = [\{b'_{m(n)}\}].$ However, $\{b'_{m(n)}\}$ is explicitly defined as the projection of the original sequence $\{b_{k(n)}\}$ down to level $m(n)$. By the property of asymptotic equivalence, $[\{b'_{m(n)}\}] = [\{b_{k(n)}\}]$. Substituting this equality yields $\wp(\Phi)([\{a_{m(n)}\}]) = [\{b_{k(n)}\}]$, proving that $\ker(\wp(\Psi)) \subseteq \operatorname{im}(\wp(\Phi))$.

\medskip

    \emph{Right Exactness:} Let $[\{c_{k(n)}\}] \in \wp(\underline{C})$ be an arbitrary class, represented by a sequence with $c_{k(n)} \in C_{k(n)}$ and $k(n) \to \infty$. By the exactness of the original sequence at the third node, the map $\psi_{k(n)} \colon B_{k(n)} \to C_{k(n)}$ is surjective for every integer $k(n)$. 
    
    Thus, for every $n \ge 0$, we can choose an element $b_{k(n)} \in B_{k(n)}$ such that $\psi_{k(n)}(b_{k(n)}) = c_{k(n)}.$ Because $k(n) \to \infty$, the term-by-term pullback sequence $\{b_{k(n)}\}$ is a valid asymptotic sequence that defines a class $[\{b_{k(n)}\}] \in \wp(\underline{B})$. Applying the induced map, we have $\wp(\Psi)([\{b_{k(n)}\}]) = [\{\psi_{k(n)}(b_{k(n)})\}] = [\{c_{k(n)}\}].$ Therefore, $\wp(\Psi)$ is surjective, completing the proof of exactness.
\end{proof}
The proof of the following theorem is also straightforward and is therefore omitted.
\begin{theorem}[Finite Products]
    Let $\underline{G} = \{G_n, p_n^m\}$ and $\underline{H} = \{H_n, q_n^m\}$ be inverse systems of groups. The product of the inverse systems is defined level-wise as $\underline{G} \times \underline{H} = \{G_n \times H_n, p_n^m \times q_n^m\}$. Then the map $\Phi\colon \wp(\underline{G} \times \underline{H}) \to \wp(\underline{G}) \times \wp(\underline{H})$ defined by
    \[
        \Phi\left([\{(g_{k(n)}, h_{k(n)})\}]\right) = \left( [\{g_{k(n)}\}], [\{h_{k(n)}\}] \right)
    \]
    is an isomorphism. 
\end{theorem}
\subsection{Characterization of Pro-Triviality via \texorpdfstring{$\wp$}{℘}}\label{protriv}

The aim of this section is to investigate to what extent the limiting data derived from various functors---such as the inverse limit, the first derived limit, and the asymptotic limit---determine an inverse system of groups up to isomorphism. To motivate this specific notion of isomorphism before its formal definition, let us first consider a geometric example.

Consider the complex plane $\mathbb{C}$. The simplest way to compactify $\mathbb{C}$ is to add an extra abstract point, often denoted by $\infty$. We can represent this point purely using the topology of $\mathbb{C}$ by defining it via a nested sequence $\mathcal{U} \coloneqq (U_0 \supseteq U_1 \supseteq U_2 \supseteq \cdots)$ of connected, unbounded open subsets of $\mathbb{C}$ that eventually misses every compact subset. However, to ensure that the representation of $\infty$ is well-defined (i.e., independent of the specific choice of the nested sequence $\mathcal{U}$), it is natural to define an equivalence relation on all such sequences. If $\mathcal{V} = (V_0 \supseteq V_1 \supseteq V_2 \supseteq \cdots)$ is another such nested sequence, we say $\mathcal{V}$ is equivalent to $\mathcal{U}$ if, for every index $n$, there is a corresponding index $m$ such that $V_m \subseteq U_n$, and vice versa. Equivalently, $\mathcal{V}$ is isomorphic to $\mathcal{U}$ if we can form the following commutative ladder diagram, where all maps are inclusions:
\[
\begin{tikzcd}[column sep=1.5em] 
    U_{i_0} &                          & U_{i_1} \arrow[ll, hook'] \arrow[ld, hook] &                                            & U_{i_2} \arrow[ll, hook'] \arrow[ld, hook] &                                            & U_{i_3} \arrow[ll, hook'] \arrow[ld, hook] &                                            & \cdots \arrow[ll, hook'] \arrow[ld, hook] &                          \\         & V_{j_0} \arrow[lu, hook] &                                            & V_{j_1} \arrow[lu, hook] \arrow[ll, hook'] &                                            & V_{j_2} \arrow[lu, hook] \arrow[ll, hook'] &                                            & V_{j_3} \arrow[lu, hook] \arrow[ll, hook'] &                                           & \cdots \arrow[ll, hook'] 
\end{tikzcd}
\]

The above diagram motivates us to define a notion of isomorphism for inverse systems, which we will call pro-isomorphism. The prefix ``pro'' is derived from ``projective'', as some authors refer to inverse systems as projective systems. 

\begin{definition}\cite[p. 66]{MR3598162} \cite[p. 740]{MR784012} Two inverse systems $\underline{G} = \{G_n, p^m_n\}$ and $\underline{H} = \{H_n, q^m_n\}$ are said to be \emph{pro-isomorphic}, denoted $\underline{G} \cong \underline{H}$, if there exist strictly increasing sequences of indices $i_0 < i_1 < i_2 < \cdots$ and $j_0 < j_1 < j_2 < \cdots$, along with cross-homomorphisms $\lambda_n \colon G_{i_{n+1}} \to H_{j_n}$ and $\mu_n \colon H_{j_n} \to G_{i_n}$, such that the following mutually commuting ladder diagram can be formed:
    \[
    \begin{tikzcd}[column sep=1.5em]
        G_{i_0} & & G_{i_1} \arrow[ll, "p^{i_1}_{i_0}"'] \arrow[ld, "\lambda_0"'] & & G_{i_2} \arrow[ll, "p^{i_2}_{i_1}"'] \arrow[ld, "\lambda_1"'] & & G_{i_3} \arrow[ll, "p^{i_3}_{i_2}"'] \arrow[ld, "\lambda_2"'] & & \cdots \arrow[ll] \arrow[ld] & \\
                & H_{j_0} \arrow[lu, "\mu_0"] & & H_{j_1} \arrow[lu, "\mu_1"] \arrow[ll, "q^{j_1}_{j_0}"] & & H_{j_2} \arrow[lu, "\mu_2"] \arrow[ll, "q^{j_2}_{j_1}"] & & H_{j_3} \arrow[lu, "\mu_3"] \arrow[ll, "q^{j_3}_{j_2}"] & & \cdots \arrow[ll]
    \end{tikzcd}
    \]
    Explicitly, this requires that every triangle in the diagram commutes. That is, $\mu_n \circ \lambda_n = p^{i_{n+1}}_{i_n}$ and $\lambda_n \circ \mu_{n+1} = q^{j_{n+1}}_{j_n}$ for all integers $n \ge 0$. Pro-isomorphism is an equivalence relation on the class of all inverse systems of groups, namely $\operatorname{obj}(\mathsf{InvGrp})$. 
    
\end{definition}
\begin{remark}
    More generally, for any category $\mathcal{C}$, one can construct the category pro-$\mathcal{C}$ whose objects are inverse systems in $\mathcal{C}$ indexed by arbitrary directed sets. By defining morphisms as equivalence classes of compatible maps, pro-$\mathcal{C}$ is formed as a quotient category in which the categorical isomorphisms generalize the exact notion of pro-isomorphism defined above (see \cite[\S 11.2]{MR2365352}). 
\end{remark}   

It is a well-known fact that pro-isomorphic inverse systems of groups have isomorphic inverse limits ($\varprojlim$) and isomorphic first derived limits ($\varprojlimone$). However, the converse of either statement is generally false.

For the inverse limit, consider the trivial inverse system $\underline{\bm{1}}$ and the system $\underline{G}$, defined respectively by: $\underline{\bm{1}} \coloneqq 1 \longleftarrow 1 \longleftarrow 1 \longleftarrow \cdots$ and 
        $\underline{G} \coloneqq \mathbb{Z} \xleftarrow{\times 2} \mathbb{Z} \xleftarrow{\times 2} \mathbb{Z} \xleftarrow{\times 2} \cdots$ 
    Clearly, these two systems are not pro-isomorphic. However, for the trivial system, the shift operator $\Delta_{\underline{\bm{1}}}$ is the zero map, so $\varprojlim \underline{\bm{1}} =\ker\Delta_{\underline{\bm{1}}}= 1$. For $\underline{G}$, the shift operator $\Delta_{\underline{G}}\colon \prod_{n=0}^\infty \mathbb{Z} \to \prod_{n=0}^\infty \mathbb{Z}$ is given by  $\Delta_{\underline{G}}(\{x_n\}) = \{y_n\}$, where $ y_n=x_n - 2x_{n+1}$. So $\varprojlim \underline G=\ker \Delta_{\underline{G}}= 1$. Thus, both systems have trivial inverse limits, yet they are not pro-isomorphic.

    Similarly, the converse fails for the first derived limit, $\varprojlimone$, which is computed as the cokernel of the shift operator $\Delta$. Consider the system: $\underline{H} \coloneqq \mathbb{Z} \xleftarrow{\mathrm{id}} \mathbb{Z} \xleftarrow{\mathrm{id}} \mathbb{Z} \xleftarrow{\mathrm{id}} \cdots$. The shift operator $\Delta_{\underline{H}} \colon \prod_{n=0}^\infty \mathbb{Z} \to \prod_{n=0}^\infty \mathbb{Z}$ is given by $\{x_n\} \mapsto \{y_n\}$, where $y_n=x_n - x_{n+1}$. Because $\Delta_{\underline{H}}$ is surjective, $\varprojlimone \underline{H} = \operatorname{coker}(\Delta_{\underline{H}}) = 1$ and   $\varprojlimone \underline{\bm{1}}= \operatorname{coker}(\Delta_{\underline{\bm{1}}})=1 $. Although both systems have a trivial $\varprojlimone$, they are clearly not pro-isomorphic.
    
 In contrast, we have the following.
\begin{theorem}\label{thm:trivial_wp}
    Let $\underline{G} = \{G_n, p^m_n\}$ be an inverse system of groups. If $\wp(\underline{G})$ is trivial, then $\underline{G}$ is pro-isomorphic to the trivial inverse system $\underline{\bm{1}}$.
\end{theorem}
\begin{proof}
    We proceed in two steps. First, we establish that $\underline{G}$ must be \emph{pro-trivial}, meaning that for every index $n$, there exists an integer $m > n$ such that the bonding map $p^m_n \colon G_m \to G_n$ is the trivial homomorphism (i.e., $p^m_n(x) = 1$ for all $x \in G_m$). 
    
    Suppose, for the sake of contradiction, that $\underline{G}$ is not pro-trivial. Then there exists some fixed index $n_0$ such that for all $m \ge n_0$, the bonding map $p^m_{n_0}$ is not trivial. Thus, for each $m \ge n_0$, we can choose an element $x_m \in G_m$ such that $p^m_{n_0}(x_m) \neq 1.$ For $m < n_0$, we define $x_m = 1$. This construction yields a diagonal sequence $\{x_m\}_{m \ge 0} \in \prod_{m=0}^\infty G_m$. Because we can view this as an asymptotic sequence with index function $k(m) = m$, it defines a class $[\{x_m\}] \in \wp(\underline{G})$. 
    
    By hypothesis, $\wp(\underline{G})$ is trivial, so this sequence must be asymptotically equivalent to the trivial sequence $\{1_m\}$. By definition of $\sim$, there exists a bounding sequence $c(m) \to \infty$ with $c(m) \le m$ such that $p^m_{c(m)}(x_m) = 1$ for all sufficiently large $m$. Because $c(m) \to \infty$, there exists some $M$ such that $c(m) \ge n_0$ for all $m \ge M$. For any such $m$, we have $n_0 \le c(m) \le m$. Projecting the relation down to the level $n_0$, we apply $p^{c(m)}_{n_0}$ to both sides:
    \[
    p^m_{n_0}(x_m) = p^{c(m)}_{n_0}\left(p^m_{c(m)}(x_m)\right) = p^{c(m)}_{n_0}(1) = 1.
    \]
    But this strictly contradicts our explicit choice of $x_m$, which guaranteed $p^m_{n_0}(x_m) \neq 1$. Therefore, the assumption fails, and $\underline{G}$ must be pro-trivial.

     For the second step, we use this pro-triviality to construct a pro-isomorphism between $\underline{G}$ and the trivial system $\underline{\bm{1}}$. Indeed, let $j_n=n$ for all $n \ge 0$. By pro-triviality, we can choose a strictly increasing sequence $0 = i_0 < i_1 < i_2 < i_3 < \cdots$ such that the bonding maps satisfy $p^{i_{n+1}}_{i_n}(x) = 1$ for all $x \in G_{i_{n+1}}$. Then, taking $\lambda_n \equiv 1$ and $\mu_n \equiv 1$ to be the constant trivial cross-homomorphisms, the required ladder diagram trivially commutes. Thus, $\underline{G} \cong \underline{\bm{1}}$.
\end{proof}
\begin{theorem}[Pro-isomorphism Invariance]\label{proinv}
    Let $\underline{G} = \{G_n, p^m_n\}$ and $\underline{H} = \{H_n, q^m_n\}$ be two inverse systems of groups. If $\underline{G}$ and $\underline{H}$ are pro-isomorphic, then $\wp(\underline{G}) \cong \wp(\underline{H})$.
\end{theorem}
\begin{proof}
    Follows from Lemma~\ref{SI} and Lemma~\ref{AI} below.
\end{proof}

\begin{remark} The converse of Theorem~\ref{proinv} is generally false. For instance, consider the constant inverse systems $\underline{G}$ and $\underline{H}$ defined by identity bonding maps:
    \begin{align*}
        \underline{G} &\coloneqq \mathbb{Z} \xleftarrow{\mathrm{id}} \mathbb{Z} \xleftarrow{\mathrm{id}} \mathbb{Z} \xleftarrow{\mathrm{id}} \cdots \\
        \underline{H} &\coloneqq (\mathbb{Z} \oplus \mathbb{Z}) \xleftarrow{\mathrm{id}} (\mathbb{Z} \oplus \mathbb{Z}) \xleftarrow{\mathrm{id}} (\mathbb{Z} \oplus \mathbb{Z}) \xleftarrow{\mathrm{id}} \cdots
    \end{align*}
    These systems are not pro-isomorphic. If they were, the definition of pro-isomorphism would require the existence of cross-homomorphisms $\mu_{n+1} \colon \mathbb{Z} \oplus \mathbb{Z} \to \mathbb{Z}$ and $\lambda_n \colon \mathbb{Z} \to \mathbb{Z} \oplus \mathbb{Z}$ such that the triangle commutes: $\lambda_n \circ \mu_{n+1} = q^{j_{n+1}}_{j_n} = \mathrm{id}_{\mathbb{Z} \oplus \mathbb{Z}}, $ which is impossible.

    By Theorem~\ref{constant}, we have $\wp(\underline{G}) \cong \left.\prod_{n=0}^\infty \mathbb{Z} \right/ \bigoplus_{n=0}^\infty \mathbb{Z}$ and $\wp(\underline{H}) \cong \left.\prod_{n=0}^\infty (\mathbb{Z} \oplus \mathbb{Z}) \right/ \bigoplus_{n=0}^\infty (\mathbb{Z} \oplus \mathbb{Z})$. Thus, their asymptotic groups are isomorphic. Indeed, the isomorphism from $\prod_{n=0}^\infty \mathbb{Z}$  to $\prod_{n=0}^\infty (\mathbb{Z} \oplus \mathbb{Z})$ given by $ (x_0, x_1, x_2, x_3, \dots) \longmapsto ((x_0, x_1), (x_2, x_3), \dots)$ induces an isomorphism from  $\left.\prod_{n=0}^\infty \mathbb{Z} \right/ \bigoplus_{n=0}^\infty \mathbb{Z}$ to $\left.\prod_{n=0}^\infty (\mathbb{Z} \oplus \mathbb{Z}) \right/ \bigoplus_{n=0}^\infty (\mathbb{Z} \oplus \mathbb{Z})$.
\end{remark}
However, unlike $\varprojlim$ and $\varprojlimone$, the functor $\wp$ perfectly detects pro-triviality.

\begin{corollary}\label{protrivdet}
    Let $\underline{G} = \{G_n, p^m_n\}$ be an inverse system of groups. Then $\wp(\underline{G})$ is trivial if and only if $\underline{G}$ is pro-isomorphic to the trivial inverse system $\underline{\bm{1}}$.
\end{corollary}
To prove Theorem~\ref{proinv}, we first establish the following two lemmas.
\begin{lemma}[Subsequence Invariance]\label{SI}
    Let $\underline{G} = \{G_n, p^m_n\}$ be an inverse system of groups, and let $0 \le i_0 < i_1 < i_2 < \dots$ be a strictly increasing sequence of indices. Let $\underline{H} = \{H_n, q^m_n\}$ be the inverse system defined by the corresponding subsequence, where $H_n = G_{i_n}$ and $q^m_n = p^{i_m}_{i_n}$. Then $\wp(\underline{G}) \cong \wp(\underline{H}).$
\end{lemma}

\begin{proof}
    We will construct canonical, mutually inverse homomorphisms $\iota \colon \wp(\underline{H}) \to \wp(\underline{G})$ and $\pi \colon \wp(\underline{G}) \to \wp(\underline{H})$.

    Let $[\{y_{l(n)}\}] \in \wp(\underline{H})$ be represented by an asymptotic sequence with $y_{l(n)} \in H_{l(n)}$ and $l(n) \to \infty$. Because $H_{l(n)} = G_{i_{l(n)}}$, we can view $\{y_{l(n)}\}$ as a sequence in $\underline{G}$ with the index sequence $k(n) \coloneqq i_{l(n)}$. Since $l(n) \to \infty$ and $\{i_n\}$ is strictly increasing, it follows that $k(n) \to \infty$. We define the inclusion map by taking the equivalence class of this sequence in the larger system:
    \[
    \iota([\{y_{l(n)}\}]) \coloneqq [\{\overline y_{k(n)}\}], \quad \text{where } k(n) = i_{l(n)} \text{ and } \overline y_{k(n)} = y_{l(n)}.
    \]
    This map is well-defined because any bounding sequence $m(n) \to \infty$ establishing an equivalence $\{y_{l(n)}\} \sim \{y'_{l'(n)}\}$ in $\underline{H}$ immediately yields a bounding sequence $i_{m(n)} \to \infty$ establishing $\{\overline y_{k(n)}\} \sim \{\overline y'_{k'(n)}\}$ in $\underline{G}$, where $k(n)=i_{l(n)}$ and $k'(n)=i_{l'(n)}$. It is straightforward to verify that $\iota$ preserves the group operation.

    Now, we construct the projection map $\pi$. Let $[\{x_{k(n)}\}] \in \wp(\underline{G})$ be represented by a sequence with $x_{k(n)} \in G_{k(n)}$ and $k(n) \to \infty$. We map this sequence onto the indices of $\underline{H}$. For all sufficiently large $n$, we can define $l(n)$ to be the unique non-negative integer such that $i_{l(n)} \le k(n) < i_{l(n)+1}$. Because $k(n) \to \infty$, the index $l(n)$ must also diverge to infinity. Let $z_{l(n)} \coloneqq p^{k(n)}_{i_{l(n)}}(x_{k(n)}) \in G_{i_{l(n)}} = H_{l(n)}$. We define the projection map by 
    \[
    \pi([\{x_{k(n)}\}]) \coloneqq [\{z_{l(n)}\}], \quad \text{where } i_{l(n)} \le k(n) < i_{l(n)+1}.
    \]

    To verify that $\pi$ is well-defined, suppose $\{x_{k(n)}\} \sim \{x'_{k'(n)}\}$ in $\underline{G}$. There exists a bounding sequence $b(n) \to \infty$ such that $p^{k(n)}_{b(n)}(x_{k(n)}) = p^{k'(n)}_{b(n)}(x'_{k'(n)})$ for all sufficiently large $n$. We choose an index sequence $r(n) \to \infty$ such that $i_{r(n)} \le \min\{b(n), i_{l(n)}, i_{l'(n)}\}$. Projecting the equivalence down to the level $i_{r(n)}$ yields $p^{k(n)}_{i_{r(n)}}(x_{k(n)}) = p^{k'(n)}_{i_{r(n)}}(x'_{k'(n)})$. Because $i_{r(n)} \le i_{l(n)} \le k(n)$ and $i_{r(n)} \le i_{l'(n)} \le k'(n)$, we can factor these projections through the respective subsequence levels:\begin{multline*}
    q^{l(n)}_{r(n)}(z_{l(n)}) = p^{i_{l(n)}}_{i_{r(n)}}(z_{l(n)}) = p^{i_{l(n)}}_{i_{r(n)}}\left(p^{k(n)}_{i_{l(n)}}(x_{k(n)})\right) = p^{k(n)}_{i_{r(n)}}(x_{k(n)}) = p^{k'(n)}_{i_{r(n)}}(x'_{k'(n)}) \\[1ex]
    = p^{i_{l'(n)}}_{i_{r(n)}}\left(p^{k'(n)}_{i_{l'(n)}}(x'_{k'(n)})\right) = p^{i_{l'(n)}}_{i_{r(n)}}(z'_{l'(n)}) = q^{l'(n)}_{r(n)}(z'_{l'(n)}).
\end{multline*}Thus, $\{z_{l(n)}\} \sim \{z'_{l'(n)}\}$ in $\underline{H}$, proving that $\pi$ is well-defined. A similar argument confirms $\pi$ is a homomorphism.

    We now show that $\iota$ and $\pi$ are mutual inverses. First, consider the composition $\iota \circ \pi$ applied to a class $[\{x_{k(n)}\}] \in \wp(\underline{G})$. Applying $\pi$ yields the sequence $z_{l(n)} = p^{k(n)}_{i_{l(n)}}(x_{k(n)})$. Applying $\iota$ embeds this sequence back into $\underline{G}$ at the index $j(n) \coloneqq i_{l(n)}$. The resulting sequence is $\{\overline z_{j(n)}\}$, where $\overline z_{j(n)} = z_{l(n)}$. To show $\{\overline z_{j(n)}\} \sim \{x_{k(n)}\}$ in $\underline{G}$, we choose the common bounding sequence $m(n) \coloneqq i_{l(n)}$. Since $l(n) \to \infty$, $m(n) \to \infty$. Evaluating both terms down to level $m(n)$ yields:
    \[
    p^{j(n)}_{m(n)}(\overline z_{j(n)}) = p^{i_{l(n)}}_{i_{l(n)}}(z_{l(n)}) = z_{l(n)} = p^{k(n)}_{i_{l(n)}}(x_{k(n)}).
    \]
    Because the projections are identical, $\{\overline z_{j(n)}\} \sim \{x_{k(n)}\}$, meaning $\iota(\pi([\{x_{k(n)}\}])) = [\{x_{k(n)}\}]$.

    Second, consider $\pi \circ \iota$ applied to a class $[\{y_{l(n)}\}] \in \wp(\underline{H})$. Applying $\iota$ yields the sequence $\overline y_{k(n)} = y_{l(n)}$ at index $k(n) = i_{l(n)}$. Applying $\pi$ requires finding the unique integer $t(n)$ such that $i_{t(n)} \le k(n) < i_{t(n)+1}$. Since $k(n) = i_{l(n)}$, it is immediate that $t(n) = l(n)$. The projection map then yields the sequence:
    \[
    z_{l(n)} = p^{k(n)}_{i_{l(n)}}(\overline y_{k(n)}) = p^{i_{l(n)}}_{i_{l(n)}}(y_{l(n)}) = y_{l(n)}.
    \]
    Thus, $\pi(\iota([\{y_{l(n)}\}])) = [\{y_{l(n)}\}]$. Since $\iota$ and $\pi$ are well-defined, mutually inverse homomorphisms, we conclude that $\wp(\underline{G}) \cong \wp(\underline{H})$.
\end{proof}

\begin{lemma}\label{AI}
    Let $\underline{G} = \{G_n, p^m_n\}$ and $\underline{H} = \{H_n, q^m_n\}$ be two inverse systems of groups. Suppose we have cross-homomorphisms $\lambda_n \colon G_{n+1} \to H_{n}$ and $\mu_n \colon H_{n} \to G_{n}$, such that $\mu_n \circ \lambda_n = p^{n+1}_{n}$ and $\lambda_n \circ \mu_{n+1} = q^{n+1}_{n}$ for all integers $n \ge 0$.
    \[
    \begin{tikzcd}[column sep=1.5em]
        G_{0} & & G_{1} \arrow[ll, "p^{1}_{0}"'] \arrow[ld, "\lambda_0"'] & & G_{2} \arrow[ll, "p^{2}_{1}"'] \arrow[ld, "\lambda_1"'] & & G_{3} \arrow[ll, "p^{3}_{2}"'] \arrow[ld, "\lambda_2"'] & & \cdots \arrow[ll] \arrow[ld] & \\
                & H_{0} \arrow[lu, "\mu_0"] & & H_{1} \arrow[lu, "\mu_1"] \arrow[ll, "q^{1}_{0}"] & & H_{2} \arrow[lu, "\mu_2"] \arrow[ll, "q^{2}_{1}"] & & H_{3} \arrow[lu, "\mu_3"] \arrow[ll, "q^{3}_{2}"] & & \cdots \arrow[ll]
    \end{tikzcd}
    \]Then $\wp(\underline{G}) \cong \wp(\underline{H})$.
\end{lemma}
\begin{proof}
     By induction, for all integers $m \ge c \ge 0$, we have the intertwining identities:
    \begin{equation}
        \mu_c \circ q^m_c = p^m_c \circ \mu_m \quad \text{and} \quad \lambda_c \circ p^{m+1}_{c+1} = q^m_c \circ \lambda_m. \label{eq:intertwine}
    \end{equation}
     We now construct the forward homomorphism $\mathfrak{F} \colon \wp(\underline{G}) \to \wp(\underline{H})$. Let $[\{x_{k(n)}\}] \in \wp(\underline{G})$. Discarding finitely many terms if necessary, we may assume $k(n) \ge 1$ for all $n$. We define:
    \[
    \mathfrak{F}([\{x_{k(n)}\}]) \coloneqq [\{\lambda_{k(n)-1}(x_{k(n)})\}].
    \]Applying the second identity of \eqref{eq:intertwine}, $\mathfrak F$ can be shown to be a well-defined group homomorphism.
    
    Next, we construct the reverse homomorphism $\mathfrak{G} \colon \wp(\underline{H}) \to \wp(\underline{G})$. Let $[\{z_{s(n)}\}] \in \wp(\underline{H})$. We define:
    \[
    \mathfrak{G}([\{z_{s(n)}\}]) \coloneqq [\{\mu_{s(n)}(z_{s(n)})\}].
    \]Applying the first identity of \eqref{eq:intertwine}, $\mathfrak G$ can be shown to be a well-defined group homomorphism.

    Finally, it is straightforward to show that $\mathfrak{F}$ and $\mathfrak{G}$ are mutual inverses. Since $\mathfrak{F}$ and $\mathfrak{G}$ are mutually inverse homomorphisms, we conclude that $\wp(\underline{G}) \cong \wp(\underline{H})$.  
\end{proof}
\subsection{The Trivial-or-Uncountable Dichotomy for \texorpdfstring{$\wp(\underline G)$}{℘(G)}}\label{dichotomy}
In this section, we prove that the asymptotic limit of any non-pro-trivial system is uncountable. We utilize this result later, in Section~\ref{section:exotic}, to show that the proper fundamental group of an exotic, open, contractible $3$-manifold is uncountable---—a stark contrast to the classical fact that the ordinary fundamental group of a connected manifold is countable \cite[Theorem 7.21]{MR2766102}.

\begin{theorem}\label{uncountable}Let $\underline{G} = \{G_n, p^m_n\}$ be an inverse system of groups. Then the group $\wp(\underline{G})$ is either the trivial group or uncountably infinite.
\end{theorem}
Let $P = \prod_{n=0}^\infty G_n$ be the infinite direct product of the groups $G_n$. Define a subset $ \mathscr Z$ of $P$ as follows:
    \[
    \mathscr Z \coloneqq \left\{ \{x_n\}_{n \ge 0} \in P \;\middle|\; \exists\, m(n) \to \infty \text{ with } m(n) \le n \text{ s.t. } p^n_{m(n)}(x_n) = 1\,  \text{ for all sufficiently large } n \right\}.
    \] 
\begin{lemma}The set $ \mathscr Z$ is a normal subgroup of $P$, and the canonical map $\iota \colon P/ \mathscr Z \hookrightarrow \wp(\underline{G})$, induced by the natural inclusion of a diagonal sequence $\{x_n\} \in P$ into the space of all asymptotic sequences, is a well-defined injective group homomorphism.
\end{lemma}

\begin{proof}
    Every element $\{x_n\}_{n \ge 0} \in P$ can be viewed as an asymptotic sequence where the index sequence is simply $k(n) = n$. Because $n \to \infty$ as $n \to \infty$, every sequence in $P$ defines a valid equivalence class in $\wp(\underline{G})$. Furthermore, multiplication in $P$ is defined coordinate-wise, which matches the definition of the group product in $\wp(\underline{G})$ when taking the bounding index $m(n) = n$. Thus, the natural projection map $\pi \colon P \to \wp(\underline{G})$ defined by $\pi(\{x_n\}) = [\{x_n\}]$ is a valid group homomorphism.

    To prove that the induced map $\iota$ is injective, we characterize the kernel of $\pi$. Suppose a diagonal sequence $\{x_n\} \in P$ maps to the identity class in $\wp(\underline{G})$. The identity class is represented by the constant sequence of identities $\{1_n\}$, where $1_n \in G_n$. By the definition of the asymptotic equivalence relation $\sim$, we have $\{x_n\} \sim \{1_n\}$ if and only if there exists a common bounding sequence $m(n) \to \infty$ with $m(n) \le n$ such that for all sufficiently large $n$,
    \[
    p^n_{m(n)}(x_n) = p^n_{m(n)}(1_n) = 1.
    \]
    This condition is exactly the defining property of the subgroup $ \mathscr Z$. Therefore, an element $\{x_n\}$ is in the kernel of $\pi$ if and only if $\{x_n\} \in \mathscr Z$, establishing that $\ker(\pi) =  \mathscr Z$. 
    
    By the First Isomorphism Theorem, the induced map $\iota \colon P/N \to \wp(\underline{G})$ is injective.
\end{proof}

\begin{remark}
    The map $\iota$ (and consequently $\pi$) is generally not surjective. To see this explicitly, consider a tower where $G_n = \mathbb{Z}$ for all $n \ge 0$, and the bonding maps $p^m_n \colon G_m \to G_n$ are defined by multiplication by $2^{m-n}$. Let $[\{y_{k(n)}\}] \in \wp(\underline{G})$ be an equivalence class represented by a sequence given by $y_{k(n)} = 1 \in G_{\lfloor n/2 \rfloor}$, where $k(n) = \lfloor n/2 \rfloor$. 
    
    If $\pi$ were surjective, there would exist a diagonal sequence $\{x_n\}_{n \ge 0} \in P$ (meaning $x_n \in G_n$) equivalent to $\{y_{k(n)}\}$. By definition, there must exist an index sequence $m(n) \to \infty$ with $m(n) \le \lfloor n/2 \rfloor$ where their projections eventually coincide: $p^n_{m(n)}(x_n) = p^{\lfloor n/2 \rfloor}_{m(n)}(1).$ Applying the definition of the bonding maps yields $2^{n - m(n)} x_n = 2^{\lfloor n/2 \rfloor - m(n)}.$ Dividing out the common factor of $2^{-m(n)}$ gives $2^{n - \lfloor n/2 \rfloor} x_n = 1.$
    Since $n - \lfloor n/2 \rfloor \ge 1$ for all $n \ge 2$, this requires solving equations like $2x_2 = 1$, $2x_3 = 1$, $4x_4 = 1$, etc., within the integers $\mathbb{Z}$. Because there are no such integer solutions, the element $[\{y_{k(n)}\}]$ cannot be lifted up to $P$. Therefore, $\pi$ is not surjective.
\end{remark}

\begin{proposition} If $P/ \mathscr Z$ is trivial, then $\wp(\underline{G})$ is also trivial.
\end{proposition}

\begin{proof}
    Suppose, for the sake of contradiction, that there exists some fixed index $m_0$ such that for all $n \ge m_0$, the bonding map $p^n_{m_0} \colon G_n \to G_{m_0}$ is not the trivial homomorphism. Therefore, for every $n \ge m_0$, we can choose an element $x_n \in G_n$ such that $p^n_{m_0}(x_n) \neq 1.$ For $n < m_0$, we simply set $x_n = 1$. This defines a diagonal sequence $\{x_n\}_{n \ge 0} \in P$. 

    By our hypothesis, $P = \mathscr Z$, so there exists a sequence $m(n) \to \infty$ with $m(n) \le n$ such that $p^n_{m(n)}(x_n) = 1$ for all sufficiently large $n$. Because $m(n) \to \infty$, there exists an $N$ such that for all $n \ge N$, we have $m(n) \ge m_0$. 
    
    For any $n \ge \max\{m_0, N\}$, we can project the relation $p^n_{m(n)}(x_n) = 1$ down to the level $m_0$:
    \[
    p^n_{m_0}(x_n) = p^{m(n)}_{m_0} \left( p^n_{m(n)}(x_n) \right) = p^{m(n)}_{m_0}(1) = 1.
    \]
    But this strictly contradicts our explicit choice of $x_n$, which was chosen precisely so that $p^n_{m_0}(x_n) \neq 1$. Thus, for every index $m$, there exists an integer $N(m) \ge m$ such that the map $p^{N(m)}_m \colon G_{N(m)} \to G_m$ is trivial. This means that for every index $m$, there exists an integer $N(m) \ge m$ such that for all $n \ge N(m)$, the map $p^n_m=p^{N(m)}_m\circ p^{n}_{N(m)}\colon G_n \to G_m$ is trivial.

    Now, let $[\{y_{k(n)}\}] \in \wp(\underline{G})$ be an arbitrary asymptotic class. We want to show that $\{y_{k(n)}\} \sim \{1_{k(n)}\}$. Without loss of generality, we may choose the sequence $N(m)$ to be strictly increasing with respect to $m$. We define a bounding index sequence $c(n)$ by:
    \[
    c(n) \coloneqq \max \{ m \in \mathbb{N} \mid N(m) \le k(n) \}.
    \]
    Because $k(n) \to \infty$ as $n \to \infty$, and $N(m)$ is strictly increasing, it immediately follows that $c(n) \to \infty$. Furthermore, since $N(m) \ge m$, we have $c(n) \le N(c(n)) \le k(n)$, making $c(n)$ a valid bounding sequence.

    Projecting $y_{k(n)}$ down to level $c(n)$ applies the map $p^{k(n)}_{c(n)}$. By our definition of $c(n)$, we have $k(n) \ge N(c(n))$. Because $k(n)$ exceeds the triviality threshold for level $c(n)$, the map $p^{k(n)}_{c(n)}$ is the trivial homomorphism. Therefore, $p^{k(n)}_{c(n)}(y_{k(n)}) = 1$ for all $n$. This precisely satisfies the definition of equivalence to the identity sequence. Hence, $[\{y_{k(n)}\}] = 1$, and $\wp(\underline{G})$ is trivial.
\end{proof}

We are now ready to prove Theorem~\ref{uncountable}.

\begin{proof}[Proof of Theorem~\ref{uncountable}]
    Let $P = \prod_{n=0}^\infty G_n$ be the infinite direct product. Endowing each discrete group $G_n$ with the discrete topology equips the product space $P$ with the standard product topology. Induced by the Baire metric, defined by $d(\{a_n\},\{b_n\}) = 2^{-\min\{n \mid a_n \neq b_n\}}$ if $\{a_n\}\neq \{b_n\}$ and $d(\{a_n\},\{b_n\}) = 0$ if $\{a_n\}= \{b_n\}$, the space $P$ is completely metrizable. 

    Recall the canonical injection $\iota \colon P/\mathscr{Z} \hookrightarrow \wp(\underline{G})$. If $\wp(\underline{G})$ is countable, then $P/\mathscr{Z}$ must also be countable. Consequently, $\mathscr{Z}$ has a countable index in $P$, allowing $P$ to be partitioned into a countable union of left cosets defined by a set of sequence representatives $\{y^{(i)}\}_{i=1}^\infty$:
\[
    P = \bigcup_{i=1}^\infty y^{(i)} \mathscr{Z}.
\]Define $C_{k,N} = \{ \{x_n\}\in P \mid p_k^n(x_n) = 1 \text{ for all } n \ge N \}$. We establish three properties of $C_{k,N}$.\medskip

 \emph{$C_{k,N}$ is a subgroup:} Let $x=\{x_n\}, y=\{y_n\} \in C_{k,N}$. For any $n \ge N$, since the bonding map $p_k^n$ is a homomorphism, it preserves group operations. Thus, $p_k^n(x_n y_n^{-1}) = p_k^n(x_n) (p_k^n(y_n))^{-1} = 1 \cdot 1^{-1} = 1$. The sequence $x y^{-1}$ satisfies the defining condition, proving $C_{k,N}$ is a subgroup.\medskip

 \emph{$C_{k,N}$ is a closed subset:} The projection map $\pi_n \colon P \to G_n$ is continuous by the definition of the product topology. Because $G_k$ is equipped with the discrete topology, the composite map $f_n(x) = p_k^n(\pi_n(x)) = p_k^n(x_n)$ is continuous. The subset can therefore be expressed as the infinite intersection $C_{k,N} = \bigcap_{n=N}^\infty f_n^{-1}(\{1\})$. Being an intersection of continuous preimages of closed singleton sets, $C_{k,N}$ is closed in $P$.\medskip

 \emph{$\mathscr{Z}$ is a subset of $\bigcup_{N=0}^\infty C_{k,N}$:} Let $x=\{x_n\} \in \mathscr{Z}$. By definition, there exists an index sequence $m(n) \to \infty$ and a threshold $N_0$ such that $p_{m(n)}^n(x_n) = 1$ for all $n \ge N_0$. For any arbitrary target level $k$, the divergence $m(n) \to \infty$ guarantees the existence of an index $N_1$ where $m(n) \ge k$ for all $n \ge N_1$. Let $N = \max(N_0, N_1)$. Thus, for all $n \ge N$, we have $p_k^n(x_n) = p_k^{m(n)}(p_{m(n)}^n(x_n)) = p_k^{m(n)}(1) = 1$.
Since $p_k^n(x_n) = 1$ for all $n \ge N$, the sequence $x\in C_{k,N}$. Therefore, $\mathscr{Z}$ is a subset of $\bigcup_{N=0}^\infty C_{k,N}$.\medskip

Applying this covering property specifically for $k=0$ yields $\mathscr{Z} \subseteq \bigcup_{N=0}^\infty C_{0,N}$. We can thus expand our partition of the complete metric space $P$ into a countable union of closed sets:
\[
    P = \bigcup_{i=1}^\infty y^{(i)} \mathscr{Z} = \bigcup_{i=1}^\infty y^{(i)} \left( \bigcup_{N=0}^\infty C_{0,N} \right) = \bigcup_{i=1}^\infty \bigcup_{N=0}^\infty y^{(i)} C_{0,N}
\]
By the Baire Category Theorem, a completely metrizable space cannot be a countable union of nowhere dense sets. Thus, at least one closed coset $y^{(i)} C_{0,N}$ must possess a non-empty interior. Because left-translation by the sequence $y^{(i)}$ is a homeomorphism of the topological group, the closed subgroup $C_{0,N}$ itself must have a non-empty interior.

A subgroup with a non-empty interior must contain a basic open neighborhood of the identity element. There must exist an integer $M$ such that $C_{0,N}$ contains the basic open set:
\[
    U = \prod_{n=0}^M \{1\} \times \prod_{n=M+1}^\infty G_n \subseteq C_{0,N}
\]
For this containment to hold, placing any arbitrary element $g \in G_n$ at an index $n > \max(M, N)$ must yield the identity under the projection $p_0^n$. This forces the bonding map $p_0^n \colon G_n \to G_0$ to be identically the trivial homomorphism.

Applying this exact logic to an arbitrary level $k$, there must exist a higher level $n > k$ such that the bonding map $p_k^n$ is trivial. This condition dictates that $\underline G$ is pro-isomorphic to the trivial system $\underline{\bm 1}$.  By Theorem~\ref{proinv}, $\wp(\underline{G}) \cong 1$.

Therefore, if $\wp(\underline G)$ is countable, then $\wp(\underline{G}) \cong 1$. Thus, if $\wp(\underline{G})$ is non-trivial, then $\wp(\underline{G})$ must be uncountably infinite.
\end{proof}

\section{Connection Between Asymptotic Limits and Derived Limits} \label{connection}

To further illustrate how the classical limit functors fall short of the asymptotic limit, we present a four-term exact sequence associated with every inverse system $\underline{G} = \{G_n, p^m_n\}$ of abelian groups.

\begin{theorem}\label{4ses}
If each $G_n$ is abelian, then there is an exact sequence of groups:
\[ 0 \longrightarrow \varprojlim \underline{G} \longrightarrow \wp(\underline{G}) \longrightarrow \wp(\underline{G}) \longrightarrow {\varprojlim}^1 \underline{G} \longrightarrow 0. \]
\end{theorem}

Specifically, we show that $\varprojlim \underline{G}$ and $\varprojlimone \underline{G}$ are isomorphic to the kernel and cokernel, respectively, of a homomorphism $\Phi^\infty\colon \wp(\underline{G})\to \wp(\underline{G})$. This mirrors the classical setting, where $\varprojlim$ and $\varprojlimone$ are defined as the kernel and cokernel, respectively, of the classical shift operator $\Delta\colon \prod_{n=0}^\infty G_n\to \prod_{n=0}^\infty G_n$. The construction of $\Phi^\infty$ relies on a direct limit presentation of $\wp(\underline{G})$. More precisely, let $\mathscr{I}_{\operatorname{nd}}$ be the set of all non-decreasing divergent sequences of non-negative integers, directed under the reverse pointwise order. We will first show that $\wp(\underline{G})$ is isomorphic to the direct limit $\varinjlim_{\bm{k}\in \mathscr{I}_{\operatorname{nd}}}Q_{\bm{k}}$, where 
$Q_{\bm{k}} \coloneqq \left. \prod_{n=0}^\infty G_{k(n)} \right/ \bigoplus_{n=0}^\infty G_{k(n)}$.

For every $\bm{k} \in \mathscr{I}_{\operatorname{nd}}$, the product $\prod_{n=0}^\infty G_{k(n)}$ admits a natural shift operator $\Delta_{\bm{k}}$, defined coordinate-wise by
\[ (\Delta_{\bm{k}}(\{x_{k(n)}\}))_{k(n)} = x_{k(n)} \left( p_{k(n)}^{k(n+1)}(x_{k(n+1)}) \right)^{-1}. \]
This operator preserves the direct sum, thereby inducing a quotient group homomorphism $\Delta_{\bm{k}}^\infty\colon Q_{\bm{k}} \to Q_{\bm{k}}$. Passing to the direct limit over $\bm{k}$ yields the desired homomorphism $\Phi^\infty\colon \wp(\underline{G}) \to \wp(\underline{G})$. The remainder of this section is dedicated to making the above idea precise and showing that the kernel and cokernel of $\Phi^\infty$ are indeed isomorphic to $\varprojlim \underline{G}$ and $\varprojlimone \underline{G}$, respectively.

\subsection{Colimit Representation and Asymptotic Shift Operator}
First, we present $\wp(\underline{G})$ as a colimit. For this, we need the following directed set.

Let $\mathscr I_{\operatorname{div}}$ be the set of all index sequences $\bm{k} = \{k(n)\}_{n=0}^\infty$ such that $k(n) \to \infty$ as $n \to \infty$.  We equip the collection $\mathscr I_{\operatorname{div}}$ with the structure of a directed set. Define a preorder $\preceq$ on $\mathscr I_{\operatorname{div}}$ as follows: $\bm{k} \preceq \bm{l} \iff k(n) \ge l(n)$ for all  $n \ge 0.$ Notice that for any two elements $\bm{k}, \bm{l} \in \mathscr I_{\operatorname{div}}$, the sequence $\bm{m}$ defined by $m(n) \coloneqq \min\{k(n), l(n)\}$ is an element of $\mathscr{I}_{\operatorname{div}}$ satisfying $\bm{k} \preceq \bm{m}$ and $\bm{l} \preceq \bm{m}$. Therefore, $\mathscr I_{\operatorname{div}}$ is a directed set.

We now construct a direct system. For each $\bm{k} \in \mathscr I_{\operatorname{div}}$, our objects are the quotient groups $Q_{\bm{k}} \coloneqq \left. \prod_{n=0}^\infty G_{k(n)} \right/ \bigoplus_{n=0}^\infty G_{k(n)}$ (cf. Theorem~\ref{constant}). Whenever $\bm{k} \preceq \bm{l}$, we define the transition homomorphism $\varphi_{\bm{k},\bm{l}} \colon Q_{\bm{k}} \to Q_{\bm{l}}$ coordinate-wise via the inverse system's bonding maps: $$\varphi_{\bm{k},\bm{l}}\left([\{x_{k(n)}\}]\right) \coloneqq \left[\left\{p^{k(n)}_{l(n)}(x_{k(n)})\right\}\right]$$ 
This map is well-defined. Moreover, $\varphi_{\bm{k},\bm{k}} = \mathrm{id}_{Q_{\bm{k}}}$, and if $\bm{k} \preceq \bm{l} \preceq \bm{m}$, then $p^{l(n)}_{m(n)} \circ p^{k(n)}_{l(n)} = p^{k(n)}_{m(n)}$ ensures that $\varphi_{\bm{l},\bm{m}} \circ \varphi_{\bm{k},\bm{l}} = \varphi_{\bm{k},\bm{m}}$. Thus, $\{Q_{\bm{k}}, \varphi_{\bm{k},\bm{l}}\}$ forms a direct system.

\begin{theorem}[Colimit Representation]\label{colimrep}
    The asymptotic group $\wp(\underline G)$ is canonically isomorphic to the direct limit $\varinjlim_{\bm{k} \in \mathscr I_{\operatorname{div}}} Q_{\bm{k}}$.
\end{theorem}

Recall that as a set, $\varinjlim_{\bm{k} \in \mathscr{I}_{\operatorname{div}}} Q_{\bm{k}}$ is the disjoint union of the quotient spaces modulo an equivalence relation:
    \[ \varinjlim_{\bm{k} \in \mathscr{I}_{\operatorname{div}}} Q_{\bm{k}} = \bigsqcup_{\bm{k} \in \mathscr{I}} Q_{\bm{k}} \bigg/ \sim \]
    where for $[x] \in Q_{\bm{k}}$ and $[y] \in Q_{\bm{l}}$, we define $[x] \sim [y]$ if and only if there exists some index sequence $\bm{m} \in \mathscr{I}_{\operatorname{div}}$ with $\bm{k} \preceq \bm{m}$ and $\bm{l} \preceq \bm{m}$ such that $\varphi_{\bm{k},\bm{m}}([x]) = \varphi_{\bm{l},\bm{m}}([y])$ in $Q_{\bm{m}}$. The canonical homomorphisms $i_{\bm{k}} \colon Q_{\bm{k}} \to \varinjlim_{\bm{k} \in \mathscr{I}_{\operatorname{div}}} Q_{\bm{k}}$ map an element $[x]$ to its equivalence class $i_{\bm{k}}([x])$. These satisfy $i_{\bm{k}} = i_{\bm{l}} \circ \varphi_{\bm{k},\bm{l}}$ for all $\bm{k} \preceq \bm{l}$. The set forms a group under the operation:
    \[ i_{\bm{k}}([x]) \cdot i_{\bm{l}}([y]) = i_{\bm{m}}\left(\varphi_{\bm{k},\bm{m}}([x]) \cdot \varphi_{\bm{l},\bm{m}}([y])\right) \]
    where $\bm{m} \succeq \bm{k}, \bm{l}$. The identity element is $1=i_{\bm{k}}([1])$ for any $\bm{k} \in \mathscr{I}_{\operatorname{div}}$ (where $[1]$ is the identity of $Q_{\bm{k}}$). Thus, if $[x] \in Q_{\bm{k}}$, then $i_{\bm{k}}([x]) = 1$ if and only if there exists some sequence $\bm{l} \in \mathscr{I}_{\operatorname{div}}$ with $\bm{k} \preceq \bm{l}$ such that $\varphi_{\bm{k},\bm{l}}([x]) = [1]$ in $Q_{\bm{l}}$.

\begin{proof}[Proof of Theorem~\ref{colimrep}]
    Two classes $[\{x_{k(n)}\}] \in Q_{\bm{k}}$ and $[\{y_{l(n)}\}] \in Q_{\bm{l}}$ are equivalent in the colimit if and only if there exists a common upper bound $\bm{m} \in \mathscr I_{\operatorname{div}}$ such that their images are strictly equal in $Q_{\bm{m}}$:
    \[
    \varphi_{\bm{k},\bm{m}}\left([\{x_{k(n)}\}]\right) = \varphi_{\bm{l},\bm{m}}\left([\{y_{l(n)}]\right).
    \]

    The relation is identically the definition of asymptotic equivalence in $\wp(\underline{G})$. In $\wp(\underline{G})$, two sequences $\{x_{k(n)}\}$ at index $\bm{k}$ and $\{y_{l(n)}\}$ at index $\bm{l}$ are equivalent if and only if there exists a common diverging bounding sequence $\bm{m}$ such that their projections agree for all sufficiently large $n$. Because the underlying sets and the bounding mechanisms are identical, and the equivalence relations strictly coincide ($\sim_{\mathrm{colim}} \iff \sim_\wp$), mapping a representative class in the direct limit to its corresponding asymptotic class in $\wp(\underline{G})$ yields a canonical, well-defined, bijective group homomorphism. Therefore, the asymptotic group is isomorphic to the colimit.
\end{proof}

Let $\mathscr{I}_{\operatorname{nd}} \subseteq \mathscr{I}_{\operatorname{div}}$ be the subset of sequences $\bm{m} = \{m(n)\}_{n=0}^\infty$ that are monotonically non-decreasing, meaning $m(n) \le m(n+1)$ for all $n \ge 0$. Then $\mathscr{I}_{\operatorname{nd}}$ is a \emph{cofinal subset} of $\mathscr{I}_{\operatorname{div}}$ (that is, for any $\bm{k} \in \mathscr{I}_{\operatorname{div}}$, there exists $\bm{m} \in \mathscr{I}_{\operatorname{nd}}$ such that $\bm{k} \preceq \bm{m}$). Indeed if $\bm{k}\in\mathscr{I}_{\operatorname{div}}$, then $\bm{m} \in \mathscr{I}_{\operatorname{nd}}$ defined by $m(n) \coloneqq \min_{i \ge n} k(i)$ satisfy $\bm{k} \preceq \bm{m}$.

\begin{theorem}[Another Colimit Representation]\label{colimrepnd}
    The asymptotic group $\wp(\underline G)$ is canonically isomorphic to the direct limit $\varinjlim_{\bm{k} \in \mathscr I_{\operatorname{nd}}} Q_{\bm{k}}$.
\end{theorem}
\begin{proof}
    Since $\mathscr{I}_{\operatorname{nd}}$ is a cofinal subset of $\mathscr{I}_{\operatorname{div}}$, by \cite[Proposition 8.2]{MR1757274}, $\varinjlim_{\bm{k} \in \mathscr I_{\operatorname{nd}}} Q_{\bm{k}}$ is isomorphic to $\varinjlim_{\bm{k} \in \mathscr I_{\operatorname{div}}} Q_{\bm{k}}$. Now, the proof follows from Theorem~\ref{colimrep}. 
\end{proof}

\begin{definition}
For a fixed index sequence $\bm{k} \in \mathscr{I}_{\operatorname{nd}}$, the classical subsequence shift operator $\Delta_{\bm{k}} \colon \prod_{n=0}^\infty G_{k(n)} \to \prod_{n=0}^\infty G_{k(n)}$ is defined coordinate-wise for a sequence $x = \{x_{k(n)}\}$ by:
\[ (\Delta_{\bm{k}}(x))_{k(n)} = x_{k(n)} \left( p_{k(n)}^{k(n+1)}(x_{k(n+1)}) \right)^{-1}. \]
Because $\Delta_{\bm{k}}$ strictly maps sequences of finite support (i.e., elements of $\bigoplus_{n=0}^\infty G_{k(n)}$) to sequences of finite support, it restricts to an endomorphism on the direct sum. Consequently, it induces a well-defined reduced shift operator on the quotient space $Q_{\bm{k}}$:
\[ \Delta_{\bm{k}}^\infty \colon Q_{\bm{k}} \to Q_{\bm{k}}, \quad \Delta_{\bm{k}}^\infty([x]) = [\Delta_{\bm{k}}(x)]. \]
\end{definition}

\begin{theorem}
The global asymptotic shift operator 
\[ \Phi^\infty \coloneqq \varinjlim_{\bm{k}\in \mathscr{I}_{\operatorname{nd}}} \Delta_{\bm{k}}^\infty \colon \varinjlim_{\bm{k}\in \mathscr{I}_{\operatorname{nd}}} Q_{\bm{k}} \longrightarrow \varinjlim_{\bm{k}\in \mathscr{I}_{\operatorname{nd}}} Q_{\bm{k}} \]
is a well-defined endomorphism. Thus, $\Phi^\infty\circ i_{\bm{k}}=i_{\bm{k}}\circ \Delta^\infty_{\bm{k}}$ for all $\bm{k}\in \mathscr I_{\operatorname{nd}}$.
\end{theorem}

\begin{proof}
To show that $\Phi^\infty$ is well-defined on the direct limit, we must demonstrate that the local shift operators $\Delta_{\bm{k}}^\infty$ commute with the transition homomorphisms of the directed system $(Q_{\bm{k}}, \varphi_{\bm{k},\bm{l}})$. That is, for any $\bm{k} \le \bm{l}$ in $\mathscr{I}_{\operatorname{nd}}$, the following diagram must commute:
\[
\begin{tikzcd}
Q_{\bm{k}} \arrow[r, "\Delta_{\bm{k}}^\infty"] \arrow[d, "\varphi_{\bm{k},\bm{l}}"'] & Q_{\bm{k}} \arrow[d, "\varphi_{\bm{k},\bm{l}}"] \\
Q_{\bm{l}} \arrow[r, "\Delta_{\bm{l}}^\infty"'] & Q_{\bm{l}}
\end{tikzcd}
\]
It suffices to verify this commutativity coordinate-wise on the classical direct products before passing to the quotients. Let $x = \{x_{k(n)}\} \in \prod_{n=0}^\infty G_{k(n)}$. We compute the $n$-th coordinate for both paths of the diagram.

First, applying $\Delta_{\bm{k}}$ followed by the transition map $\varphi_{\bm{k},\bm{l}}$:
The shift operator gives $(\Delta_{\bm{k}}(x))_{k(n)} = x_{k(n)} \left( p_{k(n)}^{k(n+1)}(x_{k(n+1)}) \right)^{-1}$. 
Applying the transition map $p_{l(n)}^{k(n)}$ yields:
\begin{align*}
\left( \varphi_{\bm{k},\bm{l}}(\Delta_{\bm{k}}(x)) \right)_{l(n)} &= p_{l(n)}^{k(n)} \left( x_{k(n)} \left( p_{k(n)}^{k(n+1)}(x_{k(n+1)}) \right)^{-1} \right) \\
&= p_{l(n)}^{k(n)}(x_{k(n)}) \cdot \left( p_{l(n)}^{k(n)} \circ p_{k(n)}^{k(n+1)}(x_{k(n+1)}) \right)^{-1} \\
&= p_{l(n)}^{k(n)}(x_{k(n)}) \cdot \left( p_{l(n)}^{k(n+1)}(x_{k(n+1)}) \right)^{-1},
\end{align*}
where the last equality follows from the functoriality of the inverse system bonding maps.

Next, applying the transition map $\varphi_{\bm{k},\bm{l}}$ followed by $\Delta_{\bm{l}}$:
The transition map gives $(\varphi_{\bm{k},\bm{l}}(x))_n = p_{l(n)}^{k(n)}(x_{k(n)})$.
Applying the shift operator $\Delta_{\bm{l}}$ to this sequence yields:
\begin{align*}
\left( \Delta_{\bm{l}}(\varphi_{\bm{k},\bm{l}}(x)) \right)_{l(n)} &= (\varphi_{\bm{k},\bm{l}}(x))_{l(n)} \cdot \left( p_{l(n)}^{l(n+1)} \left( (\varphi_{\bm{k},\bm{l}}(x))_{n+1} \right) \right)^{-1} \\
&= p_{l(n)}^{k(n)}(x_{k(n)}) \cdot \left( p_{l(n)}^{l(n+1)} \left( p_{l(n+1)}^{k(n+1)}(x_{k(n+1)}) \right) \right)^{-1} \\
&= p_{l(n)}^{k(n)}(x_{k(n)}) \cdot \left( p_{l(n)}^{k(n+1)}(x_{k(n+1)}) \right)^{-1},
\end{align*}
again utilizing the condition $p_{l(n)}^{l(n+1)} \circ p_{l(n+1)}^{k(n+1)} = p_{l(n)}^{k(n+1)}$.

Since both paths produce the exact same sequence coordinate-wise, we have $\varphi_{\bm{k},\bm{l}} \circ \Delta_{\bm{k}} = \Delta_{\bm{l}} \circ \varphi_{\bm{k},\bm{l}}$. This descends to a commuting diagram on the quotient spaces $Q_{\bm{k}}$. Consequently, the local endomorphisms $\{\Delta_{\bm{k}}^\infty\}_{\bm{k} \in \mathscr{I}_{\operatorname{nd}}}$ induce a well-defined global endomorphism $\Phi^\infty$ on the colimit.
\end{proof}

\subsection{Inverse Limit as Kernel}
\begin{theorem}\label{computekernel}
The kernel of the global asymptotic shift operator $\Phi^\infty$ is canonically isomorphic to the inverse limit of the system, $\varprojlim \underline{G}$.
\end{theorem}
\begin{proof}
    Let $L = \varprojlim_{m} G_m$ denote the inverse limit, consisting of coherent threads $g = \{g_m\}_{m=0}^\infty$ satisfying $p_m^n(g_n) = g_m$ for all $n \ge m$. For any thread $g \in L$ and any index sequence $\bm{k} \in \mathscr{I}_{\operatorname{nd}}$, let $g_{\bm{k}} = \{g_{k(n)}\}_{n=0}^\infty$ be the projection of $g$ onto the indices of $\bm{k}$. This sequence defines an equivalence class $[g_{\bm{k}}] \in Q_{\bm{k}} = \left. \prod_{n=0}^\infty G_{k(n)} \right/ \bigoplus_{n=0}^\infty G_{k(n)}$. Notice that if $\bm{k} \preceq \bm{l}$ (meaning $k(n) \ge l(n)$), the transition map $\varphi_{\bm{k},\bm{l}}$ applies the inverse system projections coordinate-wise:
\[ p_{l(n)}^{k(n)}(g_{k(n)}) = g_{l(n)}. \]
Thus, $\varphi_{\bm{k},\bm{l}}([g_{\bm{k}}]) = [g_{\bm{l}}]$. Because of this strict compatibility with the directed system, the element $i_{\bm{k}}([g_{\bm{k}}])$ defines a unique element in $\varinjlim_{\bm{k} \in \mathscr{I}_{\operatorname{nd}}} Q_{\bm{k}}$, independent of the choice of $\bm{k}$. We define $\Psi \colon L \to \varinjlim_{\bm{k}\in \mathscr{I}_{\operatorname{nd}}} Q_{\bm{k}}$ by $$\Psi(g) = i_{\bm{k}}([g_{\bm{k}}]).$$ 

We claim that $\operatorname{im}(\Psi)=\ker \Phi^\infty$. First, observe that $\operatorname{im}(\Psi)\subseteq\ker \Phi^\infty$. Indeed if $g\in L$, then $$\Delta_{\bm{k}}(g_{\bm{k}}))_{k(n)} = g_{k(n)} \left( p_{k(n)}^{k(n+1)}(g_{k(n+1)}) \right)^{-1}=g_{k(n)}g_{k(n)}^{-1}=1\text{ for all }n.$$Thus, $\Delta_{\bm{k}}^\infty([g_{\bm{k}}]) = [e]$ in $Q_{\bm{k}}$, which implies $\Phi^\infty(\Psi(g)) = 0$. Thus, $\operatorname{im}(\Psi)\subseteq\ker \Phi^\infty$.

Next, we show that $\operatorname{im}(\Psi)\supseteq\ker \Phi^\infty$. Let $c \in \ker(\Phi^\infty)$. Then $c = i_{\bm{k}}([x])$ for some $\bm{k} \in \mathscr{I}_{\operatorname{nd}}$ and $[x] \in Q_{\bm{k}}$. The hypothesis $\Phi^\infty(c) = 1$ implies $i_{\bm{k}}(\Delta_{\bm{k}}^\infty([x])) = 1$. Thus, there exists an index sequence $\bm{l} \in \mathscr{I}$ with $\bm{k} \preceq \bm{l}$ such that $\varphi_{\bm{k},\bm{l}}(\Delta_{\bm{k}}^\infty([x]))=[1]$. Since the local shift operators $\Delta_{\bm{k}}^\infty$ commute with the transition homomorphisms of the directed system $(Q_{\bm{k}}, \varphi_{\bm{k},\bm{l}})$, we have  $\Delta^\infty_{\bm{l}}(\varphi_{\bm{k},\bm{l}}([x]))=[1]$. Define $y_{l(n)}= p_{l(n)}^{k(n)}(x_{k(n)})$, and let $x=\{x_{k(n)}\}$ and $y=\{y_{l(n)}\}$. Then, $\Delta_{\bm{l}}^\infty([y])=\Delta^\infty_{\bm{l}}(\varphi_{\bm{k},\bm{l}}([x]))=[1]$. So there exists an integer $N \ge 0$ such that for all $n \ge N$:
\[ y_{l(n)} \left( p_{l(n)}^{l(n+1)}(y_{l(n+1)}) \right)^{-1} = 1 \implies y_{l(n)} = p_{l(n)}^{l(n+1)}(y_{l(n+1)}). \]Because $l(n) \to \infty$, we can construct a thread $g \in L$. For any fixed index $r \ge 0$, choose $M \ge N$ such that $l(i) \ge r$ for all $i \ge M$. Consider the element $p_r^{l(i)}(y_{l(i)}) \in G_r$. By the local coherence of the tail of $y$, this element is constant for all $i \ge M$:
\[ p_r^{l(i)}(y_{l(i)}) = p_r^{l(i)} \left( p_{l(i)}^{l(i+1)}(y_{l(i+1)}) \right) = p_r^{l(i+1)}(y_{l(i+1)}). \] Define $g_r$ to be this constant value. By construction, $p_r^s(g_s) = g_r$ for all $s \ge r$, so the sequence $g = \{g_r\}_{r=0}^\infty$ is a coherent thread in $L$. Finally, projecting $g$ onto $\bm{l}$ for indices $n \ge N$ yields $g_{l(n)} = p_{l(n)}^{l(i)}(y_{l(i)}) = y_{l(n)}$ (by choosing  $i=n$).
Since $g_{\bm{l}}$ and $y$ agree entirely for $n \ge N$, they represent the same class in $Q_{\bm{l}}$, meaning $[g_{\bm{l}}] = [y]$. Therefore, in the direct limit:
\[ \Psi(g) = i_{\bm{l}}([g_{\bm{l}}]) = i_{\bm{l}}([y]) = i_{\bm{k}}([x]) = c.\] Therefore, $\operatorname{im}(\Psi)=\ker \Phi^\infty$.

To complete the proof, we show that that the homomorphism $\Psi \colon L \to \ker(\Phi^\infty)$ is injective. Assume $\Psi(g) = 1$ for some coherent threads $g=\{g_m\}^\infty_{m=0}$. Then $i_{\bm{k}}([g_{\bm{k}}])=1$ for every $\bm{k}\in \mathscr I_{\operatorname{nd}}$. Thus, there exists some sequence $\bm{l} \in \mathscr{I}_{\operatorname{nd}}$ with $\bm{k} \preceq \bm{l}$ such that $[g_{\bm{l}}]=\varphi_{\bm{k},\bm{l}}([g_{\bm{k}}]) = [1]$ in $Q_{\bm{l}}$. This implies the sequence $g_{\bm{l}} = \{g_{l(n)}\}$ has finite support. Therefore, there exists an integer $N \ge 0$ such that $g_{l(n)} = 1$ for all $n \ge N$.
Let $r \ge 0$ be an arbitrary index. Because $\bm{l} \in \mathscr{I}_{\operatorname{nd}}$, we have $l(n) \to \infty$. Choose $n \ge N$ large enough so that $l(n) \ge r$. By the global coherence of $g$:
\[ g_r = p_r^{l(n)}(g_{l(n)}) = p_r^{l(n)}(1) = 1. \]
Since $r$ was arbitrary, $g$ is the trivial identity thread. Hence, $\Psi$ is injective.
\end{proof}

\subsection{Derived Limit as Cokernel}
\begin{theorem}\label{cokerlim1}
If each $G_n$ is abelian, then the cokernel of the global asymptotic shift operator $\Phi^\infty$ is canonically isomorphic to the first derived limit of the system, $\varprojlimone \underline{G}$.
\end{theorem}
From now on, assume that each \(G_n\) is abelian. Accordingly, we may occasionally use additive notation for the group operation and subtraction notation for inverses. 

Let $P = \prod_{m=0}^\infty G_m$, and let $\bm{id} = \{0, 1, 2, \dots\} \in \mathscr{I}_{\operatorname{nd}}$ be the identity sequence. Recall that the  classical shift operator  $\Delta_{\bm{id}} \colon P \to P$ is defined by \( (\Delta_{\bm{id}}(h))_m = h_m- p_m^{m+1}(h_{m+1}). \) The first derived limit is the cokernel of this map: $\varprojlimone \underline G \coloneqq P \big/ \operatorname{im}(\Delta_{\bm{id}}).$ We denote elements of $\varprojlimone \underline G$ as equivalence classes $[[h]]$, where $h \in P$. 

Elements of $\operatorname{coker}(\Phi^\infty)\coloneqq\left. \varinjlim_{\bm{k}\in \mathscr{I}_{\operatorname{nd}}} Q_{\bm{k}} \right/ \operatorname{im}(\Phi^\infty)$ are represented by $\langle i_{\bm{k}}([x]) \rangle$, where $x \in \prod G_{k(n)}$, $[x]$ is its class in $Q_{\bm{k}}$, $i_{\bm{k}}$ is the inclusion into the colimit, and $\langle \cdot \rangle$ denotes the coset modulo $\operatorname{im}(\Phi^\infty)$.

    Define $\Psi^1 \colon {\varprojlim}^1 \underline G \longrightarrow \operatorname{coker}(\Phi^\infty)$ by \[ \Psi^1([[h]]) = \langle i_{\bm{id}}([h]) \rangle. \]To check well-definedness, suppose $[[h]] = [[g]]$ in $\varprojlimone \underline G$. This means $h - g = \Delta_{\bm{id}}(y)$ for some $y \in \prod G_m$. Therefore, in the quotient $Q_{\bm{id}}$, we have $[h] - [g] = [\Delta_{\bm{id}}(y)] = \Delta_{\bm{id}}^\infty([y]).$ Pushing this to the direct limit yields $i_{\bm{id}}([h]) - i_{\bm{id}}([g]) = i_{\bm{id}}(\Delta_{\bm{id}}^\infty([y])) = \Phi^\infty(i_{\bm{id}}([y]))$. 
Since the difference lies exactly in $\operatorname{im}(\Phi^\infty)$, their classes in the cokernel are identical: $\langle i_{\bm{id}}([h]) \rangle = \langle i_{\bm{id}}([g]) \rangle$. Thus, $\Psi^1$ is well-defined.

\begin{theorem}\label{psi1inj}
    The homomorphism $\Psi^1 \colon {\varprojlim}^1 \underline G \longrightarrow \operatorname{coker}(\Phi^\infty)$ is injective.
\end{theorem}

\begin{proof}
    Assume $\Psi^1([[h]]) = 1$. This means $i_{\bm{id}}([h]) \in \operatorname{im}(\Phi^\infty)$. 
Therefore, there exists some index sequence $\bm{k} \in \mathscr{I}_{\operatorname{nd}}$ and some sequence $x \in \prod G_{k(n)}$ such that 
\[ i_{\bm{id}}([h]) = \Phi^\infty(i_{\bm{k}}([x])) = i_{\bm{k}}(\Delta_{\bm{k}}^\infty([x])). \]By the definition of the direct limit, there exists an index sequence $\bm{l} \in \mathscr{I}_{\operatorname{nd}}$ that bounds both $\bm{id}$ and $\bm{k}$ from above (meaning $\bm{id} \preceq \bm{l}$ and $\bm{k} \preceq \bm{l}$, so $l(n) \le n$ and $l(n) \le k(n)$) where these elements become equal in $Q_{\bm{l}}$:
\[ \varphi_{\bm{id},\bm{l}}([h]) = \varphi_{\bm{k},\bm{l}}(\Delta_{\bm{k}}^\infty([x])). \]Let $y \in \prod G_{l(n)}$ be the representative sequence for $\varphi_{\bm{id},\bm{l}}([h])$, given by $y_{l(n)} = p_{l(n)}^n(h_n)$. 
Let $z \in \prod G_{l(n)}$ be the representative sequence for $\varphi_{\bm{k},\bm{l}}([x])$, given by $z_{l(n)} = p_{l(n)}^{k(n)}(x_{k(n)})$. 
By the functoriality of the shift operator, the transition maps commute with the shifts, so $\varphi_{\bm{k},\bm{l}}(\Delta_{\bm{k}}^\infty([x])) = \Delta_{\bm{l}}^\infty([z])$. Thus, in the quotient $Q_{\bm{l}}$, we have $[y] = \Delta_{\bm{l}}^\infty([z])$. 
This implies $y - \Delta_{\bm{l}}(z) \in \bigoplus_{n=0}^\infty G_{l(n)}$. By  Lemma~\ref{rearrangement}, there exists $v \in \bigoplus G_{l(n)}$ such that $\Delta_{\bm{l}}(v) = y - \Delta_{\bm{l}}(z)$. Rearranging gives $y = \Delta_{\bm{l}}(z + v)$.  This proves that $y$ is in the image of the local shift operator $\Delta_{\bm{l}}$. By Lemma~\ref{cofinallim1}, $[[h]] = 1$.
\end{proof}
\begin{lemma}[Bijectivity on Finite Support Sequences]\label{rearrangement}For any fixed index sequence $\bm{k} \in \mathscr{I}_{\operatorname{nd}}$, let $S_{\bm{k}} = \bigoplus_{n=0}^\infty G_{k(n)}$ be the group of sequences with finite support. The restriction of the classical subsequence shift operator to this subgroup, \( \Delta_{\bm{k}}|_{S_{\bm{k}}} \colon S_{\bm{k}} \longrightarrow S_{\bm{k}}, \)is a bijection.
\end{lemma}
\begin{proof}
    Let $1_{k(n)}$ denote the identity element in $G_{k(n)}$, and let $1 \in S_{\bm{k}}$ denote the trivial sequence where every coordinate is the identity.
    \medskip

    \emph{Injectivity:} Suppose $x \in S_{\bm{k}}$ is an element in the kernel of the restricted shift operator, meaning $\Delta_{\bm{k}}(x) = 1$. By the definition of the shift operator, this implies that $x_{k(n)} = p_{k(n)}^{k(n+1)}(x_{k(n+1)})$ for all $n \ge 0$. Because $x$ has finite support ($x \in S_{\bm{k}}$), there exists an integer $N \ge 0$ such that $x_{k(n)} = 1_{k(n)}$ for all $n \ge N$. We proceed by finite backward induction. For $n = N-1$, our coordinate equation yields:
\[ x_{k(N-1)} = p_{k(N-1)}^{k(N)}(x_{k(N)}) = p_{k(N-1)}^{k(N)}(1_{k(N)}) = 1_{k(N-1)}. \]
Repeating this backward substitution for $N-2, N-3, \dots, 0$, it follows immediately that $x_{k(n)} = 1_{k(n)}$ for all $0 \le n < N$. Since $x$ is the identity at every coordinate, $x = 1$. The kernel is trivial, proving that $\Delta_{\bm{k}}|_{S_{\bm{k}}}$ is injective.
\medskip

\emph{Surjectivity:} Let $y \in S_{\bm{k}}$ be an arbitrary sequence of finite support. We must construct a sequence $x \in S_{\bm{k}}$ such that $\Delta_{\bm{k}}(x) = y$. The condition $\Delta_{\bm{k}}(x) = y$ requires that $x_{k(n)} = y_{k(n)} \cdot p_{k(n)}^{k(n+1)}(x_{k(n+1)})$ for all $n \ge 0$. Because $y \in S_{\bm{k}}$, there exists an integer $N \ge 0$ such that $y_{k(n)} = 1_{k(n)}$ for all $n \ge N$. 

We define the candidate pre-image $x$ by setting its tail to the identity. Define $x_{k(n)} = 1_{k(n)}$ for all $n \ge N$. Notice that for any $n \ge N$, our required equation is trivially satisfied: $x_{k(n)} (p_{k(n)}^{k(n+1)}(x_{k(n+1)}))^{-1} =  1_{k(n)} = y_{k(n)}$.
For the coordinates where $0 \le n < N$, we define $x_{k(n)}$ recursively via backward finite induction:
\[ x_{k(n)} := y_{k(n)} \cdot p_{k(n)}^{k(n+1)}(x_{k(n+1)}). \]
Since $x_{k(N)}$ is explicitly defined, $x_{k(N-1)}$ is uniquely determined by $y_{k(N-1)}$ and $x_{k(N)}$. Subsequently, $x_{k(N-2)}$ is determined by $y_{k(N-2)}$ and $x_{k(N-1)}$, and so on down to $x_{k(0)}$. This construction yields a well-defined sequence $x$. Furthermore, because $x_{k(n)} = 1_{k(n)}$ for all $n \ge N$, the sequence $x$ inherently possesses finite support, meaning $x \in S_{\bm{k}}$. By construction, $\Delta_{\bm{k}}(x) = y$, proving that the restricted shift operator is surjective.

Since $\Delta_{\bm{k}}|_{S_{\bm{k}}}$ is both injective and surjective, it is a bijection.
\end{proof}

\begin{lemma}\label{cofinallim1}
Let $\bm{l} \in \mathscr{I}_{\operatorname{nd}}$ be an index sequence such that $\bm{id} \preceq \bm{l}$. Let $h \in \prod_{n=0}^\infty G_n$. If the projected sequence $y \in \prod_{n=0}^\infty G_{l(n)}$, defined by $y_{l(n)} = p_{l(n)}^n(h_n)$, lies in the image of the local shift operator $\Delta_{\bm{l}}$, then $h$ lies in the image of the global shift operator $\Delta_{\bm{id}}$. Consequently, $[[h]] = 1$ in $\varprojlimone\underline G$.
\end{lemma}
\begin{definition}[Telescoping Partial Products]
For a sequence $h = \{h_n\} \in \prod_{n=0}^\infty G_n$ and any indices $k > n \ge 0$, define the partial product $H_n^k \in G_n$ by: $H_n^k \coloneqq h_n p_n^{n+1}(h_{n+1}) \cdots p_n^{k-1}(h_{k-1}).$ For convenience, define the empty product $H_n^n = 1$. 
\end{definition}
This definition yields the recursive properties: $H_n^{k+1} = H_n^k p_n^k(h_k)$ and $h_n p_n^{n+1}(H_{n+1}^k) = H_n^k$.

\begin{proof}[Proof of Lemma~\ref{cofinallim1}]
    By hypothesis, $y = \Delta_{\bm{l}}(v)$ for some $v \in \prod_{n=0}^\infty G_{l(n)}$. Evaluated at the index $l(k)$, the local shift operator gives the condition: \( p_{l(k)}^k(h_k) = y_{k(n)}= v_{l(k)} ( p_{l(k)}^{l(k+1)}(v_{l(k+1)}) )^{-1}. \)

We will explicitly construct a sequence $z \in \prod_{n=0}^\infty G_n$ such that $\Delta_{\bm{id}}(z) = h$. Fix an arbitrary index $n \ge 0$. Because $l(k) \to \infty$ as $k \to \infty$, there exists an integer $N_n \ge n$ such that for all $k \ge N_n$, we have $l(k) \ge n$.

For any such $k \ge N_n$, we can apply the inverse system projection $p_n^{l(k)}$ to the local condition. By functoriality, specifically $p_n^{l(k)} \circ p_{l(k)}^k = p_n^k$ and $p_n^{l(k)} \circ p_{l(k)}^{l(k+1)} = p_n^{l(k+1)}$, we obtain a relation purely in $G_n$:
\begin{equation}\label{tele}
    p_n^k(h_k) = p_n^{l(k)}(v_{l(k)}) \left( p_n^{l(k+1)}(v_{l(k+1)}) \right)^{-1}.
\end{equation}
Now, consider the sequence of elements in $G_n$ defined by $c_k = H_n^k p_n^{l(k)}(v_{l(k)})$. We use \eqref{tele} to evaluate $c_{k+1}$ for any $k \ge N_n$:
\begin{multline*}
    c_{k+1} = H_n^{k+1} p_n^{l(k+1)}(v_{l(k+1)}) = H_n^k p_n^k(h_k) p_n^{l(k+1)}(v_{l(k+1)})=H_n^k \left[ p_n^{l(k)}(v_{l(k)}) \left( p_n^{l(k+1)}(v_{l(k+1)}) \right)^{-1} \right]\cdot\\[1ex] p_n^{l(k+1)}(v_{l(k+1)})= H_n^k p_n^{l(k)}(v_{l(k)})=c_k.
\end{multline*}
This proves that the sequence $c_k$ is exactly constant for all $k \ge N_n$. We define $z_n \in G_n$ to be this stable eventual value. Thus, for any sufficiently large $k$, we have:
\[ z_n = H_n^k p_n^{l(k)}(v_{l(k)}). \]Finally, we verify that this constructed sequence $z = \{z_n\}_{n=0}^\infty$ satisfies $\Delta_{\bm{id}}(z) = h$. We evaluate $z_n$ and $z_{n+1}$ by choosing a common index $k$ large enough such that $k \ge N_n$ and $k \ge N_{n+1}$:
\[
    h_n p_n^{n+1}(z_{n+1})=h_np_n^{n+1} \left( H_{n+1}^k p_{n+1}^{l(k)}(v_{l(k)}) \right)=h_n p_n^{n+1}(H_{n+1}^k) p_n^{l(k)}(v_{l(k)})=H_n^k p_n^{l(k)}(v_{l(k)})=z_n.
\]Rearranging this equation yields $h_n = z_n \left( p_n^{n+1}(z_{n+1}) \right)^{-1}$, which is precisely the definition of $(\Delta_{\bm{id}}(z))_n = h_n$. Therefore, $h \in \operatorname{im}(\Delta_{\bm{id}})$, which concludes the proof that $[[h]] = 1$ in $\varprojlimone \underline G$.
\end{proof}
\begin{theorem}\label{psi1sur}
    The homomorphism $\Psi^1 \colon {\varprojlim}^1 \underline G \longrightarrow \operatorname{coker}(\Phi^\infty)$ is surjective.
\end{theorem}

\begin{proof}
    For notational simplicity, throughout the following proof, for an index sequence \(\bm{k} \in \mathscr{I}_{\operatorname{nd}}\), we write an arbitrary element \(x \in \prod_{n=0}^{\infty} G_{k(n)}\) as \(x = \{x_n\},\) where \(x_n \in G_{k(n)}\) for each \(n \ge 0\).

    Let $c \in \operatorname{coker}(\Phi^\infty)$ be an arbitrary element. Thus, there exists an index sequence $\bm{k} \in \mathscr{I}_{\operatorname{nd}}$ and a sequence $x = \{x_{n}\}  \in \prod_{n=0}^\infty G_{k(n)}$ such that $c = \langle i_{\bm{k}}([x]) \rangle$. We want to find $h \in \prod_{m=0}^\infty G_m$ such that $\Psi^1([[h]]) = c$.

    For any integer $m \ge 0$, we define the preimage $I_m = \{ n \ge 0 \mid k(n) = m \}$. Because $\bm{k}$ is non-decreasing and divergent, $I_m$ is a finite contiguous block of indices. If $I_m \neq \emptyset$, we denote its minimum and maximum indices as $a_m = \min I_m$ and $b_m = \max I_m$. We define the block-accumulated sequence $h=\{h_m\} \in \prod_{m=0}^\infty G_m$ coordinate-wise by:
\[
h_m \coloneqq \begin{cases} 
\prod_{n \in I_m} x_{n} & \text{if } I_m \neq \emptyset, \\
1 & \text{if } I_m = \emptyset 
\end{cases}
\]
where $1$ is the identity element in $G_m$.

We define a sequence $y =\{y_{n}\}\in \prod_{n=0}^\infty G_{k(n)}$ block-by-block. For each $m$ such that $I_m \neq \emptyset$, define the coordinates of $y$ for $j \in I_m$ as follows: $y_{a_m} = e$ and $y_j = (x_{a_m} x_{a_m+1} \dots x_{j-1})^{-1}$ for  $a_m < j \le b_m$. Moreover, define a new sequence $x' \coloneqq x \cdot (\Delta_{\bm{k}}(y))^{-1} \in \prod_{n=0}^\infty G_{k(n)}$. Because $x = x' \cdot \Delta_{\bm{k}}(y)$ holds in the product space, it naturally holds in the quotient space modulo finite support $Q_{\bm{k}}$:
\begin{equation}\label{step3}
    [x] = [x'] \cdot [\Delta_{\bm{k}}(y)] = [x'] \cdot \Delta_{\bm{k}}^\infty([y])
\end{equation}

Now, we want to show that the sequence $x'$ is supported at terminal indices $b_m$. To see this, we evaluate the local shift operator $\Delta_{\bm{k}}(y)_j = y_j ( p_{k(j)}^{k(j+1)}(y_{j+1}) )^{-1}$ across two distinct cases:

\emph{Case 1: Internal block indices ($a_m \le j < b_m$):} Because $j$ and $j+1$ belong to the same block $I_m$, $k(j) = k(j+1) = m$. The bonding map $p_m^m$ is the identity on $G_m$. Thus
\[
\Delta_{\bm{k}}(y)_j = y_j y_{j+1}^{-1} 
= (x_{a_m} \dots x_{j-1})^{-1} \left( (x_{a_m} \dots x_j)^{-1} \right)^{-1} 
= (x_{a_m} \dots x_{j-1})^{-1} (x_{a_m} \dots x_{j-1} x_j) 
= x_j,
\]
and hence $x'_j = x_j x_j^{-1} = 1$.

\emph{Case 2: The terminal block index ($j = b_m$):} The subsequent index $b_m + 1$ is the start of the next non-empty block, say $I_{m'}$, meaning $y_{b_m+1} = y_{a_{m'}} = 1$. So $p_m^{m'}(y_{b_m+1}) = 1$. Thus
\[
\Delta_{\bm{k}}(y)_{b_m} = y_{b_m} 1^{-1} = (x_{a_m} \dots x_{b_m-1})^{-1}, 
\] and hence $x'_{b_m} = x_{b_m} \left( (x_{a_m} \dots x_{b_m-1})^{-1} \right)^{-1} = x_{b_m} x_{a_m} x_{a_m+1} \dots x_{b_m-1}$. Because each $G_m$ is an abelian group, rearranging the terms yields $x'_{b_m} = x_{a_m} x_{a_m+1} \dots x_{b_m-1} x_{b_m} = \prod_{n \in I_m} x_n = h_m$.

Thus, the sequence $x'$ is supported at terminal indices $b_m$ and the sequence $h$ is supported at standard indices $m$. We project them into the common bounding index $\bm{l}$ where $l(n) = \min(k(n), n)$. Let $\widetilde{x} = \varphi_{\bm{k},\bm{l}}(x')$ and $\widetilde{h} = \varphi_{\bm{id},\bm{l}}(h)$. We can decompose these globally as $\widetilde{x} = \prod_{m=0}^\infty E^{(m)}$ and $\widetilde{h} = \prod_{m=0}^\infty H^{(m)}$, where $E^{(m)}$ contains the element $p_{l(b_m)}^m(h_m)$ at coordinate $b_m$, and $1$ elsewhere, and $H^{(m)}$ contains the element $p_{l(m)}^m(h_m)$ at coordinate $m$, and $1$ elsewhere.

\begin{claim}
    For every $m\geq 0$, there exists $u^{(m)}\in \prod_{n=0}^\infty G_{l(n)}$ such that $E^{(m)} \cdot (H^{(m)})^{-1} = \Delta_{\bm{l}}(u^{(m)})$.
\end{claim}
\begin{proof}[Proof of Claim]
    Define the left-shift endomorphism $S \colon \prod_{n=0}^\infty G_{l(n)} \to \prod_{n=0}^\infty G_{l(n)}$ by:
$$S(z)_j = p_{l(j)}^{l(j+1)}(z_{j+1}).$$
The global shift operator $\Delta_{\bm{l}}$ can now be written compactly as: $\Delta_{\bm{l}}(z) = z \cdot S(z)^{-1}$. Moreover, an induction on $k$ shows that $S^k(z)_j = p_{l(j)}^{l(j+k)}(z_{j+k})$ for all indices $j$. Now, fix $m\geq 0$ and consider the following two cases.\medskip

\emph{Case 1: If $b_m \ge m$:}  Let $d = b_m - m$. Then $H^{(m)} = S^d(E^{(m)})$. Moreover, $E^{(m)} \cdot (H^{(m)})^{-1} = E^{(m)} \cdot S^d(E^{(m)})^{-1} = \Delta_{\bm{l}}(u^{(m)})$,
where we define $u^{(m)} \coloneqq \prod_{i=0}^{d-1} S^i(E^{(m)})$.\medskip

\emph{Case 2: If $m > b_m$:} Let $d = m - b_m$. Then, $E^{(m)} = S^d(H^{(m)})$. Moreover, $E^{(m)} \cdot (H^{(m)})^{-1} =S^d(H^{(m)})\cdot (H^{(m)})^{-1}= \Delta_{\bm{l}}((v^{(m)})^{-1}),$ where $v^{(m)} \coloneqq \prod_{i=0}^{d-1} S^i(H^{(m)})$. So we simply set $u^{(m)} \coloneqq (v^{(m)})^{-1}$.\medskip

Therefore, $E^{(m)} \cdot (H^{(m)})^{-1} = \Delta_{\bm{l}}(u^{(m)})$.
\end{proof}

Define $u \coloneqq \prod_{m=0}^\infty u^{(m)}$. Because each $u^{(m)}$ only has non-identity elements strictly between $m$ and $b_m$, the product is coordinate-wise finite and thus well-defined in $\prod_{n=0}^\infty G_{l(n)}$. Therefore, 
$$\Delta_{\bm{l}}(u) = \prod_{m=0}^\infty \Delta_{\bm{l}}(u^{(m)}) = \prod_{m=0}^\infty E^{(m)} \cdot (H^{(m)})^{-1} = \widetilde{x} \cdot \widetilde{h}^{-1}.$$This gives $[\widetilde{x}] = [\widetilde{h}] \cdot \Delta_{\bm{l}}^\infty([u])$ in the quotient space $Q_{\bm{l}}$. Applying the canonical colimit injection $i_{\bm{l}} \colon Q_{\bm{l}} \to \varinjlim_{\bm{k}\in \mathscr{I}_{\operatorname{nd}}} Q_{\bm{k}}$, we get \(i_{\bm{l}}([\widetilde{x}]) = i_{\bm{l}}([\widetilde{h}]) \cdot i_{\bm{l}}(\Delta_{\bm{l}}^\infty([u]))\).
By the definition of the direct limit, the transition maps commute with the injections, giving $i_{\bm{l}} \circ \varphi_{\bm{k},\bm{l}} = i_{\bm{k}}$ and $i_{\bm{l}} \circ \varphi_{\bm{id},\bm{l}} = i_{\bm{id}}$. Furthermore, the injection commutes with the asymptotic shift operator such that $i_{\bm{l}} \circ \Delta_{\bm{l}}^\infty = \Phi^\infty \circ i_{\bm{l}}$. Substituting these relations yields:
\[
i_{\bm{k}}([x']) = i_{\bm{id}}([h]) \cdot \Phi^\infty(i_{\bm{l}}([u]))
\]
Because $\Phi^\infty(i_{\bm{l}}([u])) \in \operatorname{im}(\Phi^\infty)$, we conclude that \(i_{\bm{k}}([x']) \equiv i_{\bm{id}}([h]) \pmod{\operatorname{im}(\Phi^\infty)}\). Combined with the equivalence $[x] \equiv [x'] \pmod{\operatorname{im}(\Delta_{\bm{k}}^\infty)}$ from \eqref{step3}, it follows that $i_{\bm{id}}([h]) \equiv i_{\bm{k}}([x]) \pmod{\operatorname{im}(\Phi^\infty)}$. This completes the proof.
\end{proof}

\begin{proof}[Proof of Theorem~\ref{cokerlim1}]
    By Theorems~\ref{psi1inj} and \ref{psi1sur},  the homomorphism $\Psi^1$ is an isomorphism from  ${\varprojlim}^1 \underline G$ onto $\operatorname{coker}(\Phi^\infty)$.
\end{proof}

\begin{proof}[Proof of Theorem~\ref{4ses}]
    Since the kernel (resp. cokernel) of $\Phi^\infty\colon \varinjlim_{\bm{k} \in \mathscr I_{\operatorname{nd}}} Q_{\bm{k}}\to \varinjlim_{\bm{k} \in \mathscr I_{\operatorname{nd}}} Q_{\bm{k}}$ is isomorphic to $\varprojlim \underline G$ (resp. $\varprojlimone \underline G$), we have an exact sequence of groups
    \[ 0 \longrightarrow \varprojlim \underline{G} \longrightarrow \varinjlim_{\bm{k} \in \mathscr I_{\operatorname{nd}}} Q_{\bm{k}} \longrightarrow \varinjlim_{\bm{k} \in \mathscr I_{\operatorname{nd}}} Q_{\bm{k}}\longrightarrow {\varprojlim}^1 \underline{G} \longrightarrow 0. \] To complete the proof, we note that $\wp(\underline G)$ is isomorphic to $\varinjlim_{\bm{k} \in \mathscr I_{\operatorname{nd}}} Q_{\bm{k}}$ by Theorem~\ref{colimrepnd}. 
\end{proof}

\section{Proper Homology}\label{homology}
In this section, we develop two isomorphic formulations of proper homology. Both formulations will be utilized in Section~\ref{section:Hur} to connect proper homology with Brown's notion of proper homotopy theory via a Hurewicz-type theorem.
\subsection{Ends of spaces}
This section provides a brief overview of the notion of an \emph{end}---intuitively, a topological direction in which a space extends to infinity. Roughly speaking, a space $X$ has at least $n$ ends if there exists a compact subset $K \subseteq X$ such that $X\setminus K$ contains at least $n$ distinct connected components with non-compact closures. For example, $\mathbb{R}$ has exactly two ends, whereas $\mathbb{R}^2$ has exactly one end. To rigorously define the set of ends and its topology, we require the notion of an efficient exhaustion.

Let $X$ be a non-compact topological space. An \emph{efficient exhaustion} of $X$ is a sequence $K_0 \subseteq K_1 \subseteq K_2 \subseteq \cdots$ of compact, connected subsets of $X$ such that $\bigcup_{i=0}^{\infty} K_i = X$ (equivalently, $\bigcap_{i=0}^{\infty} (X \setminus K_i) = \varnothing$), $K_i \subseteq \operatorname{int}(K_{i+1})$ for all $i$, and every connected component of $X \setminus K_i$ has noncompact closure. We will be primarily interested in spaces that admit an efficient exhaustion.

\begin{definition}\label{def:nice}
    A topological space is said to be \emph{nice} if it is non-compact, connected, locally path connected, locally compact, Hausdorff, and second countable.
\end{definition}

There are several important classes of nice spaces. These include any non-compact, connected topological manifold (assuming the standard convention that manifolds are Hausdorff and second-countable), as well as any non-compact, connected, locally finite simplicial or CW complex.

The following theorem is a well-known result. The closest reference in the literature---though stated specifically for ANR spaces---is \cite[Exercise 3.3.4]{MR3598162}. We provide a sketch of the proof for the reader's convenience.
\begin{theorem}\label{eff}
Let $X$ be a nice space. Then $X$ admits an efficient exhaustion.
\end{theorem}

\begin{proof}[Sketch of Proof]
First, note that since $X$ is a connected, locally connected, locally compact, Hausdorff space, for every compact and connected $D\subseteq X$, the \emph{fill} of $D$, defined by
$$\mathsf{Fill}(D) \coloneqq D \,\cup\, \bigcup \left\{ V \;:\; V \text{ is a connected component of } X\setminus D \text{ and } \overline{V}\text{ is compact} \right\},$$ 
is compact and connected. Furthermore, every connected component of $X\setminus \mathsf{Fill}(D)$ has noncompact closure.

Now, since $X$ is locally compact, Hausdorff, and second-countable, $X$ is $\sigma$-compact; i.e., $X$ can be written as a countable union of its compact subsets. Moreover, because $X$ is locally compact and $\sigma$-compact, there exists a sequence of compact subsets $C_0\subseteq C_1\subseteq C_2\subseteq \cdots$ such that $C_i\subseteq \operatorname{int}(C_{i+1})$ for all $i$, and $X=\bigcup_{i=0}^{\infty} C_i$. We inductively construct an exhausting sequence $K_0 \subseteq K_1 \subseteq K_2 \subseteq \cdots$ of $X$.

Since $X$ is connected, locally path connected, and locally compact, there exists a compact connected set $D_0$ such that $C_0\subseteq D_0$. Define $K_0\coloneqq\mathsf{Fill}(D_0)$. By the properties of the fill, $K_0$ is compact and connected, and every connected component of $X\setminus K_0$ has noncompact closure.

Assume that $K_i$ has been constructed. Since $X=\bigcup_{n=0}^{\infty} C_n$, and since $K_i$ is compact, there exists an integer $m$ such that $K_i\subseteq \operatorname{int}(C_m)$. Again, by the same reasoning, there exists a compact connected set $D_{i+1}$ such that $C_m\subseteq D_{i+1}$. Define $K_{i+1}\coloneqq\mathsf{Fill}(D_{i+1})$. Then we have $K_i \subseteq \operatorname{int}(C_m) \subseteq C_m \subseteq D_{i+1} \subseteq K_{i+1}.$  Thus, we obtain $K_i\subseteq \operatorname{int}(K_{i+1})$. By the properties of the fill, $K_{i+1}$ is compact and connected, and every connected component of $X\setminus K_{i+1}$ has noncompact closure. 

Finally, since $X=\bigcup_{i=0}^{\infty} C_i$, and $C_0 \subseteq K_0$ with $C_m \subseteq K_{i+1}$ for arbitrarily large indices, it follows by our construction that $X=\bigcup_{i=0}^{\infty} K_i$.
\end{proof}

Let $X$ be a nice space with an efficient exhaustion $K_0\subseteq K_1\subseteq K_2\subset\cdots$. Let $\mathsf{Ends}(X)$ be the set of all sequences $(U^{(0)}, U^{(1)}, U^{(2)},\ldots)$, where $U^{(i)}$ is a component of $X\setminus K_i$ and $U^{(0)}\supseteq U^{(1)}\supseteq U^{(2)}\supseteq \cdots$. Give $\underline{X}\coloneqq X\cup \mathsf{Ends}(X)$ with the topology generated by the basis consisting of all open subsets of $X$, and all sets $\underline{U^{(i)}}$, where
$$\underline{U^{(i)}}\coloneqq U^{(i)}\cup \left.\left\{\, (V^{(0)}, V^{(1)}, V^{(2)},\ldots)\in \mathsf{Ends}(X)\; \right|\;  V^{(i)}= U^{(i)}\, \right\}.$$Then $\underline{X}$ is separable, compact, and metrizable such that $X$ is an open dense subset of $\underline{X}$ \cite[p. 57]{MR3598162}. We call $\underline{X}$ as the \emph{Freudenthal compactification} of $X$. Recall that we say a space $X_\mathsf{c}$ is a \emph{compactification} of $X$ if $X_\mathsf{c}$ is compact Hausdorff space, and $X$ is a dense subset of $X_\mathsf{c}$.  The subspace $\mathsf{Ends}(X)$ of $\underline{X}$ is a compact, metrizable, and totally disconnected space. Hence $\mathsf{Ends}(X)$ is homeomorphic to a closed subset of the Cantor set. 

A sequence $\{Z_n\}$ of subsets of $X$ is said to \emph{converge to the end} $e=(U^{(0)}, U^{(1)}, U^{(2)},\ldots)$ of $X$ provided that for every $i \geq 0$, there exists an $n_i$ such that $Z_n \subseteq U^{(i)}$ for all $n \geq n_i$. If $Z_n \neq \varnothing$ for all sufficiently large $n$, then the end $e$ is unique.

Every proper map $f \colon X \to Y$ between nice spaces induces a  map $\mathsf{Ends}(f) \colon \mathsf{Ends}(X)\to \mathsf{Ends}(Y)$ that can be used to uniquely extend $f$ to a map $\underline{f} \colon \underline{X} \to \underline{Y}$ between their Freudenthal compactifications. Moreover, $\mathsf{Ends}$ is a functor: the induced map of the identity map $\operatorname{id}_X$ is the identity map on $\mathsf{Ends}(X)$, and for any two proper maps $f \colon X \to Y$ and $g \colon Y \to Z$ between nice spaces, the induced map of their composition is the composition of their induced maps, i.e., $\mathsf{Ends}(g \circ f) = \mathsf{Ends}(g) \circ \mathsf{Ends}(f)$ \cite[Proposition 3.3.12]{MR3598162}.

Recall that a proper homotopy is a homotopy with the additional property that it is a proper map. If two proper maps $f_0, f_1 \colon X \to Y$ between nice spaces are properly homotopic, then $\mathsf{Ends}(f_0) = \mathsf{Ends}(f_1)$ \cite[Proposition 3.3.12]{MR3598162}.

\subsection{Algebraic and Geometric Formulations of Proper Homology}
In this section, we introduce two isomorphic formulations of proper homology. The first, which we call \emph{algebraic proper homology}, is defined by applying the $\wp$ functor to the singular homology of a neighborhood system of a given end. This formulation will be our primary tool for computations. The second formulation, \emph{geometric proper homology}, is constructed from chain complexes of sequences of singular chains whose supports converge to a given end. This geometric perspective is better suited for establishing various properties such as functoriality, proper homotopy invariance, and the existence of long exact sequences.

Let $X$ be a nice space with an efficient exhaustion $K_0\subseteq K_1\subseteq K_2\subset\cdots$, and let $e=(U^{(0)}, U^{(1)}, U^{(2)},\ldots)$ be an end of $X$. For a non-negative integer $k$, we define the proper homology as follows.

\begin{definition}
    Let $\underline{G}$ denote the inverse system $\{H_k(U^{(n)}), (p_n^m)_k\}_{m \ge n \ge 0}$, where the bonding maps $(p_n^m)_k \colon H_k(U^{(m)}) \to H_k(U^{(n)})$ are induced by the inclusions $U^{(m)} \hookrightarrow U^{(n)}$. The \emph{algebraic proper homology group} of $(X,e)$ in degree $k$, denoted $\underline{H}_k^\mathsf{alg}(X, e)$, is defined as $\underline{H}_k^\mathsf{alg}(X, e) \coloneqq \wp(\underline{G})$.
\end{definition}

To show that the algebraic proper homology group is independent of the choice of efficient exhaustion, we introduce the geometric approach and show that the two definitions yield isomorphic groups.

Denote the $k$-th singular chain group of $X$ by $C_k(X)$, and let $\partial_k \colon C_k(X) \to C_{k-1}(X)$ be the boundary operator. Recall that any non-trivial element $c \in C_k(X)$ can be uniquely expressed as a finite formal sum $c = \sum_i m_i \sigma_i$, where the singular $k$-simplices $\sigma_i \colon \Delta^k \to X$ are pairwise distinct and each coefficient $m_i$ is a non-zero integer. The \emph{support} of $c$ is defined as $\operatorname{supp}(c) \coloneqq \bigcup_i \operatorname{im}(\sigma_i)$ \cite[p. 71]{MR957919}. The support of the trivial element $0 \in C_k(X)$ is defined to be $\varnothing$.
  
Consider the direct product $\prod_{n \geq 0} C_k(X),$ equipped with the component-wise boundary map $\prod_{n \geq 0} \partial_k \colon \prod_{n \geq 0} C_k(X) \to \prod_{n \geq 0} C_{k-1}(X).$ The \emph{proper singular chain group} of the pair $(X,e)$ in degree $k$ is defined by
\[
  \underline{C}_k(X,e) \coloneqq \left\{\, \{c_n\}_{n \geq 0} \in \prod_{n \geq 0} C_k(X) \mid \{\operatorname{supp}(c_n)\} \text{ converges to } e \,\right\}.
\]
Then $\underline{C}_k(X,e)$ is a subgroup of $\prod_{n \geq 0} C_k(X)$, and the boundary operator $\prod_{n \geq 0} \partial_k$ restricts to a homomorphism $\underline{\partial}_k \coloneqq \left. \prod_{n \geq 0} \partial_k \, \right| \underline{C}_k(X,e) \to \underline{C}_{k-1}(X,e).$
Define
\begin{align*}
    \underline{Z}_k(X,e) & \coloneqq \left\{\, \{c_n\} \in \underline{C}_k(X,e) \mid \partial_k(c_n)=0 \text{ for all } n\gg 0 \,\right\},\text{ and}\\[1ex]
    \underline{B}_k(X,e) & \coloneqq \left\{\, \{c_n\} \in \underline{C}_k(X,e) \mid \exists\, \{b_n\}\in \underline{C}_{k+1}(X,e) \text{ such that } c_n=\partial_{k+1}(b_n) \text{ for all } n\gg 0 \,\right\}
\end{align*}
where $n \gg 0$ means ``for all sufficiently large $n$''. Then $\underline{B}_k(X,e) \subseteq \underline{Z}_k(X,e)$, and both are subgroups of $\underline{C}_k(X,e)$.

\begin{definition}
    The \emph{geometric proper homology group} of $(X,e)$ in degree $k$, denoted $\underline{H}_k^\mathsf{geo}(X,e)$, is defined as the quotient $\underline{H}_k^\mathsf{geo}(X,e) \coloneqq \underline{Z}_k(X,e) / \underline{B}_k(X,e)$.
\end{definition}

\begin{theorem}\label{computation}
Let $X$ be a nice space and $e$ an end of $X$. The geometric proper homology group $\underline{H}_k^\mathsf{geo}(X,e)$ is isomorphic to the algebraic proper homology group $\underline{H}_k^\mathsf{alg}(X,e)$.
\end{theorem}

\begin{proof}
Let $[\{c_n\}] \in \underline{H}_k^\mathsf{geo}(X,e)$. By definition, there exists an integer $N$ such that $\partial_k(c_n) = 0$ for all $n \ge N$. Since $\{\operatorname{supp}(c_n)\}$ converges to $e$, we can standardize this representative by defining a new sequence $\{c_n^\sharp\}$ as follows:
$$c_n^\sharp \coloneqq\begin{cases}c_n & \text{if } \partial_k(c_n) = 0 \text{ and } \bigcup_{j \ge n} \operatorname{supp}(c_j) \subseteq U^{(0)}, \\0 & \text{otherwise}.\end{cases}$$
Then $c_n = c_n^\sharp$ for all $n \gg 0$, meaning $[\{c_n\}] = [\{c_n^\sharp\}]$ in $\underline{H}_k^\mathsf{geo}(X,e)$. Importantly, $\partial_k(c_n^\sharp) = 0$ for all $n \ge 0$. Because $\{\operatorname{supp}(c_n^\sharp)\}$ converges to the end $e$, there exists a sequence of non-negative integers $\{l(n)\}$ diverging to $\infty$ such that $\operatorname{supp}(c_n^\sharp) \subseteq U^{(l(n))}$ for all $n \ge 0$. Thus, each $c_n^\sharp$ can be viewed as an element of $Z_k(U^{(l(n))})$.

Recall that $\underline{H}_k^\mathsf{alg}(X,e) \coloneqq \wp(\underline{G})$, where $\underline{G} = \{H_k(U^{(n)}), (p_n^m)_k\}_{m \ge n \ge 0}$ is the inverse system of groups and $(p_n^m)_k$ are the inclusion-induced homomorphisms. We define a map $\Phi \colon \underline{H}_k^\mathsf{geo}(X,e) \to \underline{H}_k^\mathsf{alg}(X,e)$ by mapping the geometric class to the asymptotic equivalence class:
$$\Phi([\{c_n\}]) \coloneqq \left[\left\{ [c_n^\sharp] \in H_k(U^{(l(n))}) \right\}\right].$$
We now verify that $\Phi$ is an isomorphism.

\emph{Well-definedness:} Suppose $[\{c_n\}] = [\{d_n\}]$ in $\underline{H}_k^\mathsf{geo}(X,e)$. By definition, there exists a sequence $\{b_n\} \in \underline{C}_{k+1}(X,e)$ and an integer $N$ such that $d_n^\sharp - c_n^\sharp = \partial_{k+1}(b_n)$ for all $n \ge N$. Because $\{\operatorname{supp}(b_n)\}$ converges to $e$, there exists a sequence of non-negative integers $l'(n) \to \infty$ such that $b_n \in C_{k+1}(U^{(l'(n))})$ for all $n \ge 0$. Let $l(n)$ and $m(n)$ be the diverging index sequences associated with $c_n^\sharp$ and $d_n^\sharp$ respectively, i.e., $\operatorname{supp}(c_n^\sharp) \subseteq U^{(l(n))}$ and $\operatorname{supp}(d_n^\sharp) \subseteq U^{(m(n))}$ for all $n \ge 0$. Define $o(n) \coloneqq \min\{l(n), m(n), l'(n)\}$. Then $o(n) \to \infty$, and in the chain complex $C_*(U^{(o(n))})$, we have $d_n^\sharp - c_n^\sharp = \partial_{k+1}(b_n)$ for all $n \ge N$. Therefore,
$$(p_{o(n)}^{m(n)})_k([d_n^\sharp]) = (p_{o(n)}^{l(n)})_k([c_n^\sharp])$$
in $H_k(U^{(o(n))})$ for all $n \ge N$. This shows the sequence $\{[c_n^\sharp]\}$ is asymptotically equivalent to $\{[d_n^\sharp]\}$, making $\Phi$ well-defined.

\emph{Homomorphism:} Let $[\{c_n\}], [\{d_n\}] \in \underline{H}_k^\mathsf{geo}(X,e)$. Consider sequences $l(n), m(n) \to \infty$ such that $c_n^\sharp \in C_k(U^{(l(n))})$ and $d_n^\sharp \in C_k(U^{(m(n))})$. Define $o(n) \coloneqq \min\{l(n), m(n)\}$. By the group operation in $\wp(\underline{G})$,\begin{align*}
\Phi([\{c_n\}]) + \Phi([\{d_n\}]) &= \left[\left\{ (p_{o(n)}^{l(n)})_k([c_n^\sharp]) + (p_{o(n)}^{m(n)})_k([d_n^\sharp]) \right\}\right] \\
&= \left[\left\{ [c_n^\sharp + d_n^\sharp] \in H_k(U^{(o(n))}) \right\}\right] \\
&= \Phi([\{c_n + d_n\}]).
\end{align*}

\emph{Injectivity:} Suppose $\Phi([\{c_n\}])$ is the trivial element of $\underline{H}_k^\mathsf{alg}(X,e)$. Consider the sequence $l(n) \to \infty$ where $c_n^\sharp \in C_k(U^{(l(n))})$. The triviality in $\wp(\underline{G})$ implies there exists a sequence $o(n) \to \infty$ with $o(n) \le l(n)$ such that $(p_{o(n)}^{l(n)})_k([c_n^\sharp]) = 0$ in $H_k(U^{(o(n))})$ for all $n \gg 0$. This means $c_n^\sharp \in B_k(U^{(o(n))})$ for all $n \gg 0$. Thus, for all $n \gg 0$, there exists $b_n \in C_{k+1}(U^{(o(n))})$ such that $\partial_{k+1}(b_n) = c_n^\sharp$. By setting $b_n = 0$ for smaller $n$, we obtain a sequence $\{b_n\} \in \prod C_{k+1}(X)$. Because $\operatorname{supp}(b_n) \subseteq U^{(o(n))}$ and $o(n) \to \infty$, the supports strictly converge to $e$, meaning $\{b_n\} \in \underline{C}_{k+1}(X,e)$. Hence, its boundary sequence $\{\partial_{k+1}(b_n)\}$ belongs to $\operatorname{im}(\underline{\partial}_{k+1})$, and since it equals $c_n^\sharp$ for $n \gg 0$, we conclude $\{c_n^\sharp\} \in \underline{B}_k(X,e)$. Therefore, $[\{c_n\}] = 0$.

\emph{Surjectivity:} Let $x \in \underline{H}_k^\mathsf{alg}(X,e)$. Then $x$ is represented by a class $\left[\left\{ g_{l(n)} \in H_k(U^{(l(n))}) \right\}\right]$, where $l(n) \to \infty$. For each $n$, choose a representative cycle $c_{l(n)} \in Z_k(U^{(l(n))})$ such that $[c_{l(n)}] = g_{l(n)}$. Define $c'_n \in C_k(X)$ by $c'_n \coloneqq c_{l(n)}$ for all $n \ge 0$. Because $l(n) \to \infty$, the sequence of supports $\{\operatorname{supp}(c'_n)\}$ converges to $e$, meaning $\{c'_n\} \in \underline{C}_k(X,e)$. Furthermore, $\partial_k(c'_n) = 0$ for all $n \ge 0$, so $\{c'_n\} \in \underline{Z}_k(X,e)$. By construction, $\Phi([\{c'_n\}]) = \left[\left\{ [c'_n] \in H_k(U^{(l(n))}) \right\}\right] = x$. Thus, $\Phi$ is surjective, completing the proof.
\end{proof}

\subsection{Basic Properties of Geometric Proper Homology} 
The aim of this section is to establish the basic properties of proper homology groups using the geometric formulation developed earlier. We detail the proofs for functoriality, proper homotopy invariance, and the long exact sequence. However, we omit explicit statements and proofs of other standard homological properties, such as the proper counterparts to the good cover proposition \cite[Proposition 2.21]{MR1867354} and the excision theorem \cite[Theorem 6.17]{MR957919}. \emph{These omitted results follow by adapting standard arguments from the ordinary singular setting, provided careful attention is paid to the convergence of the supports of the chains.}

Let $X$ and $Y$ be nice spaces. Suppose $f \colon X \to Y$ is a proper map, let $e$ be an end of $X$, and let $e' = \mathsf{Ends}(f)(e)$ be the corresponding end of $Y$. Since $f$ is continuous, it induces a chain map $C_k(f) \colon C_k(X) \to C_k(Y)$ defined by sending an element $\sum_i m_i \sigma_i \in C_k(X)$ to the element $\sum_i m_i (f \circ \sigma_i) \in C_k(Y)$. 

Notice that the support of a chain pushes forward naturally: $\operatorname{supp}\big(C_k(f)(c)\big) = f(\operatorname{supp}(c))$. Moreover, since $f$ is proper, the preimage of any compact set in $Y$ is compact in $X$. Therefore, if $\{c_n\}_{n \geq 0} \in \prod_{n \geq 0} C_k(X)$ is a sequence of chains such that $\{\operatorname{supp}(c_n)\}$ converges to the end $e$, the sequence of image supports $\{f(\operatorname{supp}(c_n))\}$ must converge to the end $e'$. 

Therefore, $f$ induces a well-defined homomorphism $\underline{C}_k(f) \colon \underline{C}_k(X,e) \to \underline{C}_k(Y,e')$ given by $\{c_n\}\longmapsto \{C_k(f)(c_n)\}$. Because $C_k(f)$ commutes with the standard boundary operator $\partial_k$, the map $\underline{C}_k(f)$ sends cycles $\underline{Z}_k(X,e)$ into $\underline{Z}_k(Y,e')$ and boundaries $\underline{B}_k(X,e)$ into $\underline{B}_k(Y,e')$. Therefore, $f$ induces a well-defined homomorphism on the geometric proper homology groups:
\(
\underline{H}_k^\mathsf{geo}(f) \colon \underline{H}_k^\mathsf{geo}(X,e) \to \underline{H}_k^\mathsf{geo}(Y,e').
\)

This assignment is functorial: the identity map $\operatorname{id}_X$ induces the identity homomorphism $\underline{H}_k^\mathsf{geo}(\operatorname{id}_X) = \operatorname{id}$, and if $g \colon Y \to Z$ is another proper map to a nice space $Z$, then $\underline{H}_k^\mathsf{geo}(g \circ f) = \underline{H}_k^\mathsf{geo}(g) \circ \underline{H}_k^\mathsf{geo}(f)$. We now show that proper homology is a proper homotopy invariant.

\begin{theorem}
    Let $X$ and $Y$ be nice spaces, and let $f,g \colon X \to Y$ be proper maps. If $f$ and $g$ are properly homotopic, then $\underline{H}_k^\mathsf{geo}(f) = \underline{H}_k^\mathsf{geo}(g) \colon \underline{H}_k^\mathsf{geo}(X, e) \to \underline{H}_k^\mathsf{geo}(Y, e'),$ 
    where $e' = \mathsf{Ends}(f)(e) = \mathsf{Ends}(g)(e)$.
\end{theorem}
\begin{proof}Let $\mathcal H \colon X \times [0,1] \to Y$ be a proper homotopy from $f$ to $g$. Recall the standard construction of the prism operator $P_k \colon C_k(X) \to C_{k+1}(Y)$ satisfying $\partial_{k+1} P_k + P_{k-1} \partial_k = C_k(g) - C_k(f)$. This operator is defined by subdividing $\Delta^k \times [0,1]$ into $(k+1)$-simplices, where for any singular $k$-simplex $\sigma \colon \Delta^k \to X$, the chain $P_k(\sigma)$ is given by an alternating sum of the compositions $\mathcal H \circ (\sigma \times \operatorname{id}_{[0,1]})$ restricted to these simplices \cite[Proof of Theorem 2.10]{MR1867354}. Thus, $\operatorname{supp}(P_k(\sigma)) \subseteq \mathcal H(\operatorname{im}(\sigma) \times [0,1])$. By linearity, for any arbitrary $k$-chain $c \in C_k(X)$, its support satisfies $\operatorname{supp}(P_k(c)) \subseteq \mathcal H(\operatorname{supp}(c) \times [0,1])$.

    Let $\{c_n\}_{n \ge 0} \in \underline{C}_k(X, e)$. By definition, the sequence of supports $\{\operatorname{supp}(c_n)\}$ converges to the end $e$. Because $\mathcal H$ is a proper map, the preimage of any compact set in $Y$ is compact in $X \times [0,1]$. Thus, the sequence of pushed-forward supports $\{\mathcal H(\operatorname{supp}(c_n) \times [0,1])\}$ must converge to the end $e'$. Since $\operatorname{supp}(P_k(c_n)) \subseteq \mathcal H(\operatorname{supp}(c_n) \times [0,1])$, the sequence $\{\operatorname{supp}(P_k(c_n))\}$ also converges to $e'$. 

    Therefore, the sequence of prism operators restricts to a well-defined family of homomorphisms $\underline{P}_k \colon \underline{C}_k(X, e) \to \underline{C}_{k+1}(Y, e')$ given by $\{c_n\} \mapsto \{P_k(c_n)\}$. This yields the proper chain homotopy relation:
    $$\underline{\partial}_{k+1} \underline{P}_k + \underline{P}_{k-1} \underline{\partial}_k = \underline{C}_k(g) - \underline{C}_k(f).$$

    To conclude, let $[\{c_n\}] \in \underline{H}_k^\mathsf{geo}(X, e)$ be represented by a proper cycle $\{c_n\} \in \underline{Z}_k(X, e)$. By definition, $\partial_k(c_n) = 0$ for all $n \gg 0$. Applying the proper chain homotopy to this sequence gives:
    $$C_k(g)(c_n) - C_k(f)(c_n) = \partial_{k+1}(P_k(c_n)) + P_{k-1}(\partial_k(c_n)).$$
    For all $n \gg 0$, the term $P_{k-1}(\partial_k(c_n))$ evaluates to zero. Thus, $C_k(g)(c_n) - C_k(f)(c_n) = \partial_{k+1}(P_k(c_n))$ for all $n \gg 0$. Because $\{P_k(c_n)\} \in \underline{C}_{k+1}(Y, e')$, this shows that the difference sequence belongs precisely to the geometric proper boundary group, i.e., $\{C_k(g)(c_n) - C_k(f)(c_n)\} \in \underline{B}_k(Y, e')$. Hence, they induce the exact same map on the quotient space, proving $\underline{H}_k^\mathsf{geo}(f) = \underline{H}_k^\mathsf{geo}(g)$.
\end{proof}

We are now ready to define the relative version of geometric proper homology. To motivate our definition, we first recall some terminology from singular homology. Let $A\subseteq X$ be a pair of topological spaces. The relative cycle group in degree $k$, denoted $Z_k(X,A)$, is the subgroup of $C_k(X)$ defined as $Z_k(X,A)\coloneqq\{\, c \in C_k(X) \mid \partial_k(c) \in C_{k-1}(A) \,\}$. The relative boundary group in degree $k$, denoted $B_k(X,A)$, is the subgroup of $C_k(X)$ defined as $B_k(X,A)\coloneqq\{\, \partial_{k+1}(d) + b \mid d \in C_{k+1}(X),\, b \in C_k(A) \,\}$. The $k$-th relative singular homology $H_k(X,A)$ is isomorphic to the quotient group $Z_k(X,A)/B_k(X,A)$ \cite[Theorem 5.11]{MR957919}.

Let $X$ be a nice space and $e$ an end of $X$. Suppose $A$ is a nice subspace of $X$ such that the inclusion $i \colon A \hookrightarrow X$ is a proper map. Let $e'$ be an end of $A$ such that $\mathsf{Ends}(i)(e') = e$. For brevity, we denote the geometric proper chain groups by $\underline{C}_k \coloneqq \underline{C}_k(X,e)$ and $\underline{C}'_k \coloneqq \underline{C}_k(A,e')$. Define the \emph{relative proper cycle group} and the \emph{relative proper boundary group} in degree $k$, respectively, as follows:
\begin{align*}
    \underline{Z}_k((X,e),(A,e')) & \coloneqq \left\{\, \{c_n\} \in \underline{C}_k \;\middle|\; 
    \begin{aligned}
        &\exists\, \{a_n'\} \in \underline{C}'_{k-1} \text{ such that } \\
        &\partial_k(c_n) = a_n' \text{ for all } n \gg 0 
    \end{aligned}
    \,\right\}, \text{ and}\\[1ex]
    \underline{B}_k((X,e),(A,e')) & \coloneqq \left\{\, \{c_n\} \in \underline{C}_k \;\middle|\; 
    \begin{aligned}
        &\exists\, \{d_n\} \in \underline{C}_{k+1} \text{ and } \{b_n'\} \in \underline{C}'_k \text{ such that } \\
        &c_n = \partial_{k+1}(d_n) + b_n' \text{ for all } n \gg 0 
    \end{aligned}
    \,\right\}.
\end{align*}

\begin{definition}
    The \emph{geometric relative proper homology group} in degree $k$, denoted $\underline{H}_k^\mathsf{geo}((X,e), \allowbreak (A,e'))$, is defined as the quotient group 
    $\underline{Z}_k((X,e),(A,e'))/\underline{B}_k((X,e),(A,e')).$
\end{definition}

\begin{theorem}
    The homomorphism $\delta_k \colon \underline{H}_k^\mathsf{geo}((X,e),(A,e')) \to \underline{H}_{k-1}^\mathsf{geo}(A,e')$ defined by
    $\delta_k([\{c_n\}])$ $\coloneqq [\{\partial_k(c_n)\}]$
    is a well-defined group homomorphism and fits into the following long exact sequence:
    $$ \cdots \to \underline{H}_k^\mathsf{geo}(A,e') \xrightarrow{i_*} \underline{H}_k^\mathsf{geo}(X,e) \xrightarrow{j_*} \underline{H}_k^\mathsf{geo}((X,e),(A,e')) \xrightarrow{\delta_k} \underline{H}_{k-1}^\mathsf{geo}(A,e') \to \cdots $$
    where $j_*([\{c_n\}]) \coloneqq [\{c_n\}]$ is induced by the natural projection of cycles.
\end{theorem}

\begin{proof}
    First, we verify that $\delta_k$ is well-defined. Suppose $[\{c_n\}] = [\{x_n\}]$ in the relative group. Then $\{c_n - x_n\} \in \underline{B}_k((X,e),(A,e'))$. By definition, there exist sequences $\{d_n\} \in \underline{C}_{k+1}$ and $\{b_n'\} \in \underline{C}'_k$ such that for $n \gg 0$, we have $c_n - x_n = \partial_{k+1}(d_n) + b_n'$. Applying the boundary operator yields $\partial_k(c_n) - \partial_k(x_n) = \partial_k(\partial_{k+1}(d_n)) + \partial_k(b_n') = \partial_k(b_n')$ for all $n \gg 0$. Since $\{b_n'\} \in \underline{C}'_k$, the sequence $\{\partial_k(c_n) - \partial_k(x_n)\}$ belongs to $\underline{B}_{k-1}(A,e')$. Thus, $[\{\partial_k(c_n)\}] = [\{\partial_k(x_n)\}]$, so $\delta_k$ is well-defined.

    We now prove exactness at each of the three nodes.

    \emph{Exactness at $\underline{H}_k^\mathsf{geo}(X,e)$:} We need to show that $\operatorname{im}(i_*) = \ker(j_*)$. If $\alpha \in \operatorname{im}(i_*)$, then $\alpha$ is represented by a proper cycle $\{c_n\} \in \underline{Z}_k(A,e')$ in $A$. Thus, $\{c_n\}$ satisfies the relative boundary condition $c_n = \partial_{k+1}(0) + c_n$, meaning $\{c_n\} \in \underline{B}_k((X,e),(A,e'))$. Therefore, $j_*(\alpha) = j_*([\{c_n\}]) = 0$. 

    Conversely, let $[\{c_n\}] \in \ker(j_*)$. Then $\{c_n\} \in \underline{B}_k((X,e),(A,e'))$, so there exist $\{d_n\} \in \underline{C}_{k+1}$ and $\{b_n'\} \in \underline{C}'_k$ such that $c_n = \partial_{k+1}(d_n) + b_n'$ for $n \gg 0$. Because $\{c_n\}$ is a proper cycle in $X$, we have $0 = \partial_k(c_n) = \partial_k(b_n')$ for all $n \gg 0$, making $\{b_n'\}$ a proper cycle in $A$. Furthermore, $c_n - b_n' = \partial_{k+1}(d_n)$ for all $n \gg 0$, which means $[\{c_n\}] = i_*([\{b_n'\}])$.

    \emph{Exactness at $\underline{H}_k^\mathsf{geo}((X,e),(A,e'))$:} We show $\operatorname{im}(j_*) = \ker(\delta_k)$. If $\alpha \in \operatorname{im}(j_*)$, then $\alpha$ is represented by a proper cycle $\{c_n\} \in \underline{Z}_k(X,e)$ in $X$, meaning $\partial_k(c_n) = 0$ for all $n \gg 0$. Thus, $\delta_k(\alpha) = \delta_k([\{c_n\}]) = [\{0\}] = 0$.

    Conversely, let $[\{c_n\}] \in \ker(\delta_k)$. Then the sequence $\{\partial_k(c_n)\}$ is a proper boundary in $A$. This implies there exists $\{a_n'\} \in \underline{C}'_k$ such that $\partial_k(c_n) = \partial_k(a_n')$ for all $n \gg 0$. Define a new sequence $z_n \coloneqq c_n - a_n'$. Notice that $\partial_k(z_n) = 0$ for all $n \gg 0$, meaning $\{z_n\} \in \underline{Z}_k(X,e)$. Furthermore, in the relative chain group, $c_n - z_n = a_n' = \partial_{k+1}(0) + a_n'$, meaning $[\{c_n\}] = [\{z_n\}] = j_*([\{z_n\}])$.

    \emph{Exactness at $\underline{H}_{k-1}^\mathsf{geo}(A,e')$:} We need to show that $\operatorname{im}(\delta_k) = \ker(i_*)$. For any $[\{c_n\}] \in \underline{H}_k^\mathsf{geo}((X,e),\allowbreak (A,e'))$, we have $i_*(\delta_k([\{c_n\}])) = i_*([\{\partial_k(c_n)\}])$. Because this class is exactly the boundary of $\{c_n\} \in \underline{C}_k(X,e)$, it is trivial in $\underline{H}_{k-1}^\mathsf{geo}(X,e)$.

    Conversely, let $[\{a_n'\}] \in \ker(i_*)$. This means $\{a_n'\}$ is a proper cycle in $A$ that bounds in $X$. Thus, there exists a sequence $\{c_n\} \in \underline{C}_k(X,e)$ such that $\partial_k(c_n) = a_n'$ for $n \gg 0$. Because its boundary lies entirely in $A$, $\{c_n\} \in \underline{Z}_k((X,e),(A,e'))$. By definition, $\delta_k([\{c_n\}]) = [\{\partial_k(c_n)\}] = [\{a_n'\}]$.
\end{proof}

\section{A Proper Hurewicz Theorem}\label{section:Hur}
This section aims to establish suitable conditions under which Brown's proper homotopy groups are isomorphic to the corresponding proper homology groups defined in the previous section. For simplicity, we work with connected, non-compact smooth manifolds. Recall that by the Cairns--Whitehead triangulation theorem, every such manifold admits a CW-complex structure.

 We now recall the definition of Brown's proper homotopy group \cite{MR356041}. Let $M$ be a connected, non-compact smooth manifold, and let $e$ be an end of $M$ determined by a sequence of connected open subsets $U_0 \supseteq U_1 \supseteq U_2 \supseteq \cdots$ of $M$ such that $\bigcap_n U_n = \varnothing$. Denote $[0,\infty)$ by $\underline{*}$, and define $\underline{\mathbb{S}^k}$ to be the space $\underline{*}$ together with a distinct $k$-sphere attached at each integer point. Adopting the terminology of \cite{MR1361888}, we call $\underline{\mathbb{S}^k}$ the string of $k$-spheres.  Fix a proper ray $r \colon [0,\infty) \to M$ such that $r([n,\infty)) \subseteq U_n$ for every integer $n \geq 0$. We consider $x_n = r(n)$ as the basepoint of $U_n$. Two proper maps between manifolds are said to \emph{have the same germ} if they agree on the complement of some compact set. The equivalence class of a proper map $f$ under this relation is denoted $\underline{f}$ and is called the \emph{germ} of $f$.

A \emph{proper map of pairs} $\alpha\colon \left(\underline{\mathbb{S}^k}, \underline{*}\right)\to (M, \underline{r})$ is a proper map $\alpha\colon \underline{\mathbb{S}^k}\to M$ such that the germ of the restriction $\alpha|_{\underline{*}}$ is $\underline{r}$. If $\beta\colon \left(\underline{\mathbb{S}^k}, \underline{*}\right)\to (M, \underline{r})$ is another proper map of pairs, we say that $\alpha$ and $\beta$ are \emph{germ homotopic rel $\underline{*}$} if there is a proper homotopy $H\colon\underline{\mathbb{S}^k} \times [0,1] \to M$ such that $\underline{H}_0 = \underline{\alpha}$, $\underline{H}_1 = \underline{\beta}$, and the germ of $H|_{\underline{*} \times [0,1]}$ is the germ of the map $\underline{*} \times [0,1] \xrightarrow{p} \underline{*} \xrightarrow{r} M$, where $p$ is the natural projection. The equivalence class of $\alpha$ is denoted $[\alpha]$. The set of all such equivalence classes, denoted $\underline{\pi}_k(M, \underline{r})$, is called the \emph{$k$-th proper homotopy group of $M$ based at $\underline{r}$}.

We now define the group structure. Let $[\alpha],[\beta] \in \underline{\pi}_k(M, \underline{r})$ and choose representatives $\alpha, \beta$ that agree on $\underline{*}$. Define $(\alpha\cdot \beta)|_{\underline{*}} = \alpha|_{\underline{*}}$. If $\mathbb{S}^k_m$ is the $k$-sphere attached at the integer $m$, we define $(\alpha\cdot \beta)|_{\mathbb{S}^k_m}$ using the standard map for the homotopy addition of $\alpha|_{\mathbb{S}^k_m}$ and $\beta|_{\mathbb{S}^k_m}$. The attaching point is treated as the basepoint. The class $[\alpha\cdot \beta]$ depends only on $[\alpha]$ and $[\beta]$, defining the group operation on $\underline{\pi}_k(M, \underline{r})$.

Let $\gamma_n$ be the path in $U_n$ from $x_{n+1}$ to $x_n$ obtained by traversing the ray $r$ backward from $r(n+1)$ to $r(n)$. Let $\underline{\Pi}_k = \{\pi_k(U_n, x_n)\}$ be the inverse system with bonding maps $f_n \coloneqq (\gamma_n)_\sharp \circ (\iota_n)_\sharp \colon \pi_k(U_{n+1}, x_{n+1}) \to \pi_k(U_n, x_n)$, where $(\gamma_n)_\sharp$ is the basepoint change isomorphism and $\iota_n\colon U_{n+1}\hookrightarrow U_n$ is the inclusion. Brown proved that $\wp(\underline{\Pi}_k)$ is isomorphic to $\underline{\pi}_k(M,\underline{r})$ \cite{MR356041}.

Similarly, let $\underline{H}_k = \{H_k(U_n)\}$ be the inverse system whose bonding maps are the homomorphisms $(\iota_n)_* \colon H_k(U_{n+1}) \to H_k(U_n)$ induced by the inclusions $\iota_n\colon U_{n+1}\hookrightarrow U_n$. By Theorem~\ref{computation}, $\underline{H}_k^\mathsf{geo}(M,e)$ is isomorphic to $\underline{H}_k^\mathsf{alg}(M,e)$. For simplicity, we denote this common group by $\underline{H}_k(M,e)$. The main theorem of this section is as follows.

\begin{theorem}\label{thm:hur}
    The abelianization $\operatorname{Ab}(\underline{\pi}_1(M,\underline{r}))$ is isomorphic to $\underline{H}_1(M,e)$. Moreover, for any $k\geq 1$, if $\underline{\pi}_1(M,\underline{r}),\dots, \underline{\pi}_k(M,\underline{r})$ are trivial, then $\underline{\pi}_{k+1}(M,\underline{r})$ is isomorphic to $\underline{H}_{k+1}(M,e)$.
\end{theorem}

\subsection{Isomorphism in Higher Degrees}
We first prove the second statement of Theorem~\ref{thm:hur} using the algebraic version $\underline{H}_k^\mathsf{alg}(M,e)$ of $\underline{H}_k(M,e)$. Suppose that $\underline{\pi}_1(M,\underline{r}),\dots, \underline{\pi}_k(M,\underline{r})$ are trivial. We will explicitly establish the resulting isomorphism $\underline{\pi}_{k+1}(M,\underline{r}) \cong \underline{H}_{k+1}(M,e)$ for the cases $k=2$ and $k=3$; the proof for $k \geq 4$ proceeds similarly. The case $k=2$ is more direct than $k=3$, as it allows us to apply the properties of covering spaces directly. For $k=3$, however, we must first develop a lifting criterion for maps into $2$-connected pointed fibrations.
\begin{theorem}[Degree 2]
    If the proper fundamental group $\underline{\pi}_1(M, \underline{r})$ is trivial, then  $\underline{\pi}_2(M, \underline{r})$ is isomorphic to $\underline{H}_2(M, e)$.
\end{theorem}

\begin{proof}
By hypothesis, the asymptotic limit $\wp(\underline{\Pi}_1) = \underline{\pi}_1(M, \underline{r})$ is trivial. By Proposition~\ref{thm:trivial_wp}, the inverse system $\underline{\Pi}_1$ is pro-trivial.

To establish an isomorphism $\wp(\underline{\Pi}_2) \cong \wp(\underline{H}_2)$, by Theorem~\ref{proinv}, it suffices to show that the systems $\underline{\Pi}_2$ and $\underline{H}_2$ are pro-isomorphic. Because $\underline{\Pi}_1$ is pro-trivial, we can extract a strictly increasing sequence of indices to form a subsequence of end neighborhoods, denoted $V_n = U_{i_n}$, such that for every $n \ge 0$, the inclusion map $\iota_n \colon V_{n+1} \hookrightarrow V_n$ induces the zero map on $\pi_1$.

Recall that the bonding map for $\underline{\Pi}_2$ is defined as $f_n \coloneqq (\gamma_n)_\sharp \circ (\iota_n)_\sharp \colon \pi_2(V_{n+1}, x_{n+1}) \to \pi_2(V_n, x_n)$, where $(\gamma_n)_\sharp$ is the basepoint change isomorphism, and the bonding map for $\underline{H}_2$ is simply $g_n \coloneqq (\iota_n)_* \colon H_2(V_{n+1}) \to H_2(V_n)$.

We need to construct cross-homomorphisms $\mu_n \colon \pi_2(V_n, x_n) \to H_2(V_n)$ and $\lambda_n \colon H_2(V_{n+1}) \to \pi_2(V_n, x_n)$ to form the following pro-isomorphism diagram:
\[
\begin{tikzcd}[column sep=1.5em]
    H_2(V_0) & & H_2(V_1) \arrow[ll, "g_0"'] \arrow[ld, "\lambda_0"'] & & H_2(V_2) \arrow[ll, "g_1"'] \arrow[ld, "\lambda_1"'] & & \cdots \arrow[ll] \arrow[ld] & \\
             & \pi_2(V_0, x_0) \arrow[lu, "\mu_0"] & & \pi_2(V_1, x_1) \arrow[lu, "\mu_1"] \arrow[ll, "f_0"'] & & \pi_2(V_2, x_2) \arrow[lu, "\mu_2"] \arrow[ll, "f_1"'] & & \cdots \arrow[ll]
\end{tikzcd}
\]

First, define $\mu_n$ as the Hurewicz homomorphism evaluated on $V_n$ at basepoint $x_n$:
\[
\mu_n \coloneqq \eta^{(n)} \colon \pi_2(V_n, x_n) \to H_2(V_n)
\]

Next, let $p^{(n)} \colon W_n \to V_n$ be the universal covering map. Choose a basepoint $\widetilde{x}_n \in (p^{(n)})^{-1}(x_n)$ in $W_n$. Let $\widetilde{\gamma}_n$ be the unique lift of the path $\gamma_n$ that ends at $\widetilde{x}_n$, and let its starting point be $\widetilde{x}_{n+1} \in (p^{(n)})^{-1}(x_{n+1})$. Because the inclusion $\iota_n$ trivially maps $\pi_1$, it lifts to a unique map $\widehat{\iota}_n \colon V_{n+1} \to W_n$ such that $p^{(n)} \circ \widehat{\iota}_n = \iota_n$ and $\widehat{\iota}_n(x_{n+1}) = \widetilde{x}_{n+1}$.

Since $W_n$ is $1$-connected, the Hurewicz theorem provides an isomorphism $\widetilde{\eta}^{(n)}_{\widetilde{x}_n} \colon \pi_2(W_n, \widetilde{x}_n) \to H_2(W_n)$. We define $\lambda_n\colon H_2(V_{n+1})\to \pi_2(V_n, x_n)$ by the following composition:
\[
H_2(V_{n+1})\xrightarrow{(\widehat{\iota}_n)_*} H_2(W_n)\xrightarrow{(\widetilde{\eta}^{(n)}_{\widetilde{x}_n})^{-1}} \pi_2(W_n, \widetilde{x}_n) \xrightarrow{(p^{(n)})_\sharp} \pi_2(V_n, x_n)
\]

We now verify the mutually commuting ladder diagram. To show $\mu_n \circ \lambda_n = g_n$, consider the following commutative diagram:
\[
\begin{tikzcd}[column sep=4.5em, row sep=2em]
    \pi_2(W_n, \widetilde{x}_n) \arrow[d, "(p^{(n)})_\sharp"']       & H_2(W_n) \arrow[l, "(\widetilde{\eta}^{(n)}_{\widetilde{x}_n})^{-1}"'] \arrow[d, "(p^{(n)})_*"'] & H_2(V_{n+1}) \arrow[l, "(\widehat{\iota}_n)_*"'] \arrow[ld, "g_n = (\iota_n)_*"] \\
    \pi_2(V_n, x_n) \arrow[r, "\mu_n = \eta^{(n)}"'] & H_2(V_n)                                                                     &                                                                                 
\end{tikzcd}
\]
Evaluating the composition yields $\mu_n \circ \lambda_n = \eta^{(n)} \circ \left( (p^{(n)})_\sharp \circ (\widetilde{\eta}^{(n)}_{\widetilde{x}_n})^{-1} \circ (\widehat{\iota}_n)_* \right)= (p^{(n)})_* \circ \widetilde{\eta}^{(n)}_{\widetilde{x}_n} \circ (\widetilde{\eta}^{(n)}_{\widetilde{x}_n})^{-1} \circ (\widehat{\iota}_n)_* = (p^{(n)})_* \circ (\widehat{\iota}_n)_*=(p^{(n)} \circ \widehat{\iota}_n)_* = (\iota_n)_* = g_n$.

Next, to show $\lambda_n \circ \mu_{n+1} = f_n$, consider the following commutative diagram mapping the basepoint change through the universal cover:
\[\begin{tikzcd}[column sep=4.5em, row sep=2.5em]
\pi_2(V_{n+1}, x_{n+1}) \arrow[r, "(\widehat{\iota}_n)_\sharp"] \arrow[d, "\mu_{n+1} = \eta^{(n+1)}"'] \arrow[rrr, "f_n", bend left=20] & \pi_2(W_n, \widetilde{x}_{n+1}) \arrow[r, "(\widetilde{\gamma}_n)_\sharp"] \arrow[d, "\widetilde{\eta}^{(n)}_{\widetilde{x}_{n+1}}"'] & \pi_2(W_n, \widetilde{x}_n) \arrow[r, "(p^{(n)})_\sharp"] & \pi_2(V_n, x_n) \\
H_2(V_{n+1}) \arrow[r, "(\widehat{\iota}_n)_*"']                                                                                        & H_2(W_n) \arrow[ru, "(\widetilde{\eta}^{(n)}_{\widetilde{x}_n})^{-1}"']                                                           &                                                       &
\end{tikzcd}\]
Evaluating the outer path, we have $\lambda_n \circ \mu_{n+1} = (p^{(n)})_\sharp \circ (\widetilde{\eta}^{(n)}_{\widetilde{x}_n})^{-1} \circ (\widehat{\iota}_n)_* \circ \eta^{(n+1)}$. By the naturality of the Hurewicz homomorphism based at $x_{n+1}$ and its lift $\widetilde{x}_{n+1}$, we have $(\widehat{\iota}_n)_* \circ \eta^{(n+1)} = \widetilde{\eta}^{(n)}_{\widetilde{x}_{n+1}} \circ (\widehat{\iota}_n)_\sharp$. Substituting this gives $\lambda_n \circ \mu_{n+1} = (p^{(n)})_\sharp \circ \left( (\widetilde{\eta}^{(n)}_{\widetilde{x}_n})^{-1} \circ \widetilde{\eta}^{(n)}_{\widetilde{x}_{n+1}} \right) \circ (\widehat{\iota}_n)_\sharp$. Since the Hurewicz homomorphism is invariant under basepoint changes $(\widetilde{\eta}^{(n)}_{\widetilde{x}_n})^{-1} \circ \widetilde{\eta}^{(n)}_{\widetilde{x}_{n+1}} = (\widetilde{\gamma}_n)_\sharp$. Thus $\lambda_n \circ \mu_{n+1} = (p^{(n)})_\sharp \circ (\widetilde{\gamma}_n)_\sharp \circ (\widehat{\iota}_n)_\sharp$. Since $p^{(n)}$ projects $\widetilde{\gamma}_n$ onto $\gamma_n$, we have $(p^{(n)})_\sharp \circ (\widetilde{\gamma}_n)_\sharp = (\gamma_n)_\sharp \circ (p^{(n)})_\sharp$. Therefore $\lambda_n \circ \mu_{n+1} = (\gamma_n)_\sharp \circ (p^{(n)} \circ \widehat{\iota}_n)_\sharp = (\gamma_n)_\sharp \circ (\iota_n)_\sharp = f_n$.

Since both identities hold for all indices $n$, $\underline{\Pi}_2$ is pro-isomorphic to $\underline{H}_2$. By Theorem~\ref{proinv}, $\wp(\underline{\Pi}_2) \cong \wp(\underline{H}_2)$, which concludes the proof that $\underline{\pi}_2(M, \underline{r}) \cong \underline{H}_2(M, e)$.
\end{proof}

We are now ready to address the case $k=3$. 

\begin{theorem}[Degree 3]\label{thm:deg3}
    If the proper homotopy groups $\underline{\pi}_1(M, \underline{r})$ and $\underline{\pi}_2(M, \underline{r})$ are trivial, then $\underline{\pi}_3(M, \underline{r})$ is isomorphic to $\underline{H}_3(M, e)$.
\end{theorem}

The primary tool for proving Theorem~\ref{thm:deg3} is the following lifting criterion for maps into $2$-connected fibrations.

\begin{lemma}\label{lem:lifttool}
    Let $(Y, y_0)$ be a $1$-connected pointed CW-complex. Then there exists a $2$-connected pointed space $(W, w_0)$ and a pointed fibration $p \colon (W, w_0) \to (Y, y_0)$ such that $\pi_n(W, w_0) \cong \pi_n(Y, y_0)$ for all $n \ge 3$. Moreover, if $(X, x_0)$ is another $1$-connected pointed CW-complex, and $f \colon (X, x_0) \to (Y, y_0)$ is a pointed map such that the induced homomorphism $f_\sharp \colon \pi_2(X, x_0) \to \pi_2(Y, y_0)$ is the zero map, then there exists a pointed continuous lift $\widehat{f} \colon (X, x_0) \to (W, w_0)$ such that $p \circ \widehat{f} = f$.
\end{lemma}

\begin{proof}
    Denote $\pi_2(Y, y_0)$ by $G$ and let $(K, k_0)$ be an Eilenberg-MacLane space $K(G,2)$ with basepoint $k_0$. By the Hurewicz theorem, $G$ is isomorphic to $H_2(Y;\mathbb{Z})$. Moreover, by \cite[Theorem 4.57]{MR1867354}, there exists a bijection $T \colon \langle Y, K\rangle \to H^2(Y;G)$ from basepoint-preserving homotopy classes of basepoint-preserving maps onto $H^2(Y;G)$ given by $T([\phi]) = \phi^*(\alpha)$, where $\alpha \in H^2(K;G)$ is the fundamental class. By the Universal Coefficient Theorem, since $Y$ is $1$-connected, we have $H^2(Y;G) \cong \operatorname{Hom}(H_2(Y;\mathbb{Z}), G) \cong \operatorname{Hom}(G, G)$. Let $g \colon (Y, y_0) \to (K, k_0)$ be a pointed map such that $T([g]_*) = g^*(\alpha)$ corresponds to $\operatorname{id}_G$ under this sequence of isomorphisms. 
    
    Let $PK = \{ \gamma \colon I \to K \mid \gamma(0) = k_0 \}$ be the path space with basepoint $c_{k_0}$ (the constant path at $k_0$). Let $p \colon (W,w_0) \to (Y,y_0)$ be the pointed fibration obtained from the following pullback square, where $w_0 = (y_0, c_{k_0})$: 
    $$\begin{tikzcd}
        (W, w_0) \arrow[d, "p"'] \arrow[r, "\theta"] & (PK, c_{k_0}) \arrow[d, "\operatorname{ev}_1"] \\
        (Y, y_0) \arrow[r, "g"']                     & (K, k_0)
    \end{tikzcd}$$

    Because the path space $PK$ is contractible, the evaluation map $\operatorname{ev}_1$ is a Serre fibration whose fiber over $k_0$ is the based loop space $\Omega(K, k_0) \simeq K(G, 1)$. Since the restriction of the projection $\theta \colon W \to PK$ to the fiber $F = \{y_0\} \times \Omega(K, k_0)$ is a homeomorphism onto $\Omega(K, k_0)$, the long exact sequence of homotopy groups shows that $\pi_n(W, w_0) \cong \pi_n(Y, y_0)$ for all $n \ge 3$. Since $g_\sharp\colon \pi_2(Y, y_0)\to \pi_2(K, k_0)$ is an isomorphism, by the naturality of the long exact sequence of a fibration,    
    $$\begin{tikzcd}             
        \pi_2(W, w_0) \arrow[r] \arrow[d, "\theta_\sharp"'] & \pi_2(Y, y_0) \arrow[r, "\partial'"] \arrow[d, "g_\sharp"', "\cong"] & \pi_1(F, w_0) \arrow[d, "(\theta\vert_F)_\sharp"', "\cong"] &             \\             
        0=\pi_2(PK, c_{k_0}) \arrow[r]        & \pi_2(K, k_0) \arrow[r, "\partial"']                    & \pi_1(\Omega(K, k_0), c_{k_0}) \arrow[r]   & \pi_1(PK, c_{k_0})=0         
    \end{tikzcd}$$
    the boundary map $\partial'\colon \pi_2(Y, y_0)\to \pi_1(F, w_0)$ is an isomorphism. Therefore, the tail end of the long exact sequence:
    $$0 = \pi_2(F, w_0) \to \pi_2(W, w_0) \xrightarrow{p_\sharp} \pi_2(Y, y_0) \xrightarrow{\partial'} \pi_1(F, w_0) \to \pi_1(W, w_0) \to \pi_1(Y, y_0) = 0$$
    shows that $\pi_2(W, w_0) \cong \ker(\partial') = 0$ and $\pi_1(W, w_0) \cong \operatorname{coker}(\partial') = 0$. Thus, $W$ is strictly $2$-connected.

    Let $(X, x_0)$ be a $1$-connected pointed CW-complex and let $f\colon (X, x_0)\to (Y, y_0)$ be a pointed map such that $f_\sharp\colon \pi_2(X, x_0)\to \pi_2(Y, y_0)$ is the zero homomorphism. By the defining universal property of the pullback square in $\mathbf{Top}_*$, $f \colon (X, x_0) \to (Y, y_0)$ lifts to a pointed map $\widehat{f} \colon (X, x_0) \to (W, w_0)$ such that $p \circ \widehat{f} = f$ if and only if the composition $g \circ f \colon (X, x_0) \to (K, k_0)$ lifts to a pointed map $f_{PK} \colon (X, x_0) \to (PK, c_{k_0})$ such that $\operatorname{ev}_1 \circ f_{PK} = g \circ f$. Because $PK$ is contractible, a lift $f_{PK} \colon (X, x_0) \to (PK, c_{k_0})$ of $g\circ f$ exists if and only if the composition $g \circ f$ is pointed nullhomotopic. In cohomology, the pointed homotopy class $[g \circ f]\in \langle X,K\rangle$ corresponds to $f^*(g^*(\alpha)) \in H^2(X; G)$. Therefore, the lift $\widehat{f} \colon (X, x_0) \to (W, w_0)$ exists if and only if $f^*(g^*(\alpha))=0$.  

    The naturality of the Hurewicz homomorphism gives the following commuting diagram:
    $$\begin{tikzcd} 
        \pi_2(X, x_0) \arrow[r, "f_\sharp"] \arrow[d, "\eta_X"'] & \pi_2(Y, y_0) \arrow[d, "\eta_Y"] \\ 
        H_2(X) \arrow[r, "f_*"']                                 & H_2(Y)                     
    \end{tikzcd}$$
    By our hypothesis, $f_\sharp=0$. Since $X$ and $Y$ are $1$-connected, $\eta_X$ and $\eta_Y$ are isomorphisms. This forces $f_*$ to be the zero homomorphism. Moreover, because $X$ is $1$-connected, the Universal Coefficient Theorem gives $H^2(X; G) \cong \operatorname{Hom}(H_2(X), G)$. Under this identification, the class $f^*(g^*(\alpha))$ corresponds to the composition of the zero map $f_*\colon H_2(X)\to H_2(Y)$ and $\eta_Y^{-1}\colon H_2(Y)\to\pi_2(Y, y_0)$. Because this composite homomorphism is identically zero, the corresponding cohomology class (i.e., $f^*(g^*(\alpha))$) vanishes. Consequently, the map $g \circ f$ is pointed nullhomotopic, guaranteeing the existence of the pointed lift $\widehat{f} \colon (X, x_0) \to (W, w_0)$ such that $p \circ \widehat{f} = f$.
\end{proof} 

\begin{proof}[Proof of Theorem~\ref{thm:deg3}] 
By Proposition~\ref{thm:trivial_wp}, both $\underline{\Pi}_1$ and $\underline{\Pi}_2$ are pro-trivial. Hence, for any index $k$, there exist integers $a(k), b(k) > k$ such that the inclusion-induced maps $\pi_1(U_{a(k)}) \to \pi_1(U_k)$ and $\pi_2(U_{b(k)}) \to \pi_2(U_k)$ are trivial. Let $M(k) \coloneqq \max(a(k), b(k))$ and define an increasing sequence of indices $i_0 < i_1 < i_2 < \dots$ inductively as follows: let $i_0 \coloneqq 0$ and define $i_{n+1} = M(i_n)$. Let $V_n \coloneqq U_{i_n}$. Then the standard inclusion map $V_{n+1} \hookrightarrow V_n$ induces the trivial zero map on both $\pi_1$ and $\pi_2$.

We extract a further sub-subsequence defined by $Y_n = V_{2n}$. The inclusion map between these new spaces, $J_n \colon Y_{n+1} \hookrightarrow Y_n$, factors as the composition of two consecutive inclusions: $Y_{n+1} = V_{2n+2} \xrightarrow{a_n} V_{2n+1} \xrightarrow{b_n} V_{2n} = Y_n$ where $a_n$ and $b_n$ are the respective inclusion maps.

Let $y_n = r(i_{2n})$ be the basepoint of $Y_n$. Let $\gamma_n$ be the path in $Y_n$ tracing the ray $r\colon [0,\infty)\to M$ backward from $y_{n+1}$ to $y_n$. The bonding maps for $\underline{\Pi}_3$ and $\underline H_3$ are defined as $f_n \coloneqq (\gamma_n)_\sharp \circ (J_n)_\sharp \colon \pi_3(Y_{n+1}, y_{n+1}) \to \pi_3(Y_n, y_n)$ and $g_n \coloneqq (J_n)_* \colon H_3(Y_{n+1}) \to H_3(Y_n)$.

Let $p_n \colon Y_n^{(1)} \to Y_n$ be the universal covering. Choose a basepoint $\widetilde{y}_n \in p_n^{-1}(y_n)$. Lift the path $\gamma_n$ uniquely to a path $\widetilde{\gamma}_n$ in $Y_n^{(1)}$ ending at $\widetilde{y}_n$, and let its starting point be $\widetilde{y}_{n+1} \in p_n^{-1}(y_{n+1})$. Let $q_n\colon (V_{2n+1}^{(1)},*)\to (V_{2n+1},y_{n+1})$ be the universal cover. Consider the liftings $\widetilde a_n\colon (Y_{n+1},y_{n+1})\to (V_{2n+1}^{(1)},*)$ and $\widetilde b_n\colon (V_{2n+1}^{(1)},*)\to (Y_n^{(1)},\widetilde y_{n+1})$ satisfying $q_n\circ \widetilde a_n=a_n$ and $p_n\circ \widetilde b_n=b_n\circ q_n$.

$$\begin{tikzcd} & {(V_{2n+1}^{(1)},*)} \arrow[d, "q_n"] \arrow[r, "\widetilde{b}_n"] & {(Y_n^{(1)},\widetilde y_{n+1})} \arrow[d, "p_n"] \\
{(Y_{n+1},y_{n+1})} \arrow[r, "a_n"'] \arrow[ru, "\widetilde{a}_n"] & {(V_{2n+1},y_{n+1})} \arrow[r, "b_n"']                             & {(Y_n,y_{n+1})}                                  
\end{tikzcd}$$

Define $\widetilde{J}_n \colon Y_{n+1} \to Y_n^{(1)}$ by $\widetilde J_n=\widetilde b_n\circ \widetilde a_n$. Then $p_n \circ \widetilde{J}_n = J_n$ and $\widetilde{J}_n(y_{n+1}) = \widetilde{y}_{n+1}$. Because $(\widetilde{b}_n)_\sharp = 0$ on $\pi_2$, the composition $\widetilde{J}_n$ induces the zero map on $\pi_2$. By applying Lemma~\ref{lem:lifttool} to the $1$-connected pointed space $(Y_n^{(1)}, \widetilde{y}_{n+1})$, we obtain a $2$-connected pointed space $(W_n, \widehat{y}_{n+1})$ and a pointed fibration $r_n \colon (W_n, \widehat{y}_{n+1}) \to (Y_n^{(1)}, \widetilde{y}_{n+1})$. Because $(\widetilde{J}_n)_\sharp = 0$ on $\pi_2$, Lemma~\ref{lem:lifttool} provides a pointed lift $L_n \colon (Y_{n+1}, y_{n+1}) \to (W_n, \widehat{y}_{n+1})$ such that $r_n \circ L_n = \widetilde{J}_n$. 

$$\begin{tikzcd} & {(W_n,\widehat y_{n+1})} \arrow[d, "r_n"] \\
{(Y_{n+1},y_{n+1})} \arrow[r, "\widetilde J_n"'] \arrow[ru, "L_n"] & {(Y_n^{(1)},\widetilde y_{n+1})}         
\end{tikzcd}$$

Next, by the Pointed Homotopy Lifting Property of $r_n$, we lift the path $\widetilde{\gamma}_n$ (which originates at $\widetilde{y}_{n+1}$) starting from the basepoint $\widehat{y}_{n+1} \in W_n$. This yields a path $\widehat{\gamma}_n$ in $W_n$ terminating at some point $\widehat{y}_n \in r_n^{-1}(\widetilde{y}_n)$. Let $c_n \coloneqq p_n \circ r_n$. Then $c_n \circ L_n = J_n$. Moreover, because $W_n$ is $2$-connected, the Hurewicz theorem provides an isomorphism $\widetilde{\eta}^{(n)}_{\widehat{y}_n} \colon \pi_3(W_n, \widehat{y}_n) \to H_3(W_n)$. 

To conclude the proof, we need to construct the cross-homomorphisms as follows:
\[
\begin{tikzcd}[column sep=1.5em]
    H_3(Y_0) & & H_3(Y_1) \arrow[ll, "g_0"'] \arrow[ld, "\lambda_0"'] & & H_3(Y_2) \arrow[ll, "g_1"'] \arrow[ld, "\lambda_1"'] & & \cdots \arrow[ll] \arrow[ld] & \\
             & \pi_3(Y_0, y_0) \arrow[lu, "\mu_0"] & & \pi_3(Y_1, y_1) \arrow[lu, "\mu_1"] \arrow[ll, "f_0"'] & & \pi_3(Y_2, y_2) \arrow[lu, "\mu_2"] \arrow[ll, "f_1"'] & & \cdots \arrow[ll]
\end{tikzcd}
\]

First, define $\mu_n$ as the classical Hurewicz homomorphism evaluated on the pointed space $(Y_n,y_n)$:
$$\mu_n \coloneqq \eta^{(n)}_{y_n} \colon \pi_3(Y_n, y_n) \to H_3(Y_n)$$
Next, define $\lambda_n\colon H_3(Y_{n+1})\to \pi_3(Y_n, y_n)$ by the following composition:
\[
H_3(Y_{n+1})\xrightarrow{(L_n)_*} H_3(W_n)\xrightarrow{(\widetilde{\eta}^{(n)}_{\widehat{y}_n})^{-1}} \pi_3(W_n, \widehat{y}_n) \xrightarrow{(c_n)_\sharp} \pi_3(Y_n, y_n)
\]

To show that $\mu_n \circ \lambda_n = g_n$, we apply the naturality of the Hurewicz homomorphism:
$$\begin{tikzcd}
\pi_3(W_n, \widehat{y}_n) \arrow[r, "\widetilde\eta^{(n)}_{\widehat y_n}"] \arrow[d, "(c_n)_\sharp"'] & H_3(W_n) \arrow[d, "(c_n)_*"] \\
\pi_3(Y_n, y_n) \arrow[r, "\eta^{(n)}_{y_n}"']                                     & H_3(Y_n)                     
\end{tikzcd}$$

Then $\mu_n \circ \lambda_n = \eta^{(n)}_{y_n} \circ (c_n)_\sharp \circ (\widetilde{\eta}^{(n)}_{\widehat{y}_n})^{-1} \circ (L_n)_* = (c_n)_* \circ \widetilde{\eta}^{(n)}_{\widehat{y}_n} \circ (\widetilde{\eta}^{(n)}_{\widehat{y}_n})^{-1} \circ (L_n)_* = (c_n)_* \circ (L_n)_* = (c_n \circ L_n)_* = (J_n)_* = g_n $

To show that $\lambda_n \circ \mu_{n+1} = f_n$, we expand the composition:
$\lambda_n \circ \mu_{n+1} = (c_n)_\sharp \circ (\widetilde{\eta}^{(n)}_{\widehat{y}_n})^{-1} \circ (L_n)_* \circ \eta^{(n+1)}$. By the naturality of the Hurewicz homomorphism evaluated at basepoint $y_{n+1}$ and its lift $\widehat{y}_{n+1}$, we have $(L_n)_* \circ \eta^{(n+1)} = \widetilde{\eta}^{(n)}_{\widehat{y}_{n+1}} \circ (L_n)_\sharp$. 

\[\begin{tikzcd}[column sep=4.5em, row sep=2.5em]
\pi_3(Y_{n+1}, y_{n+1}) \arrow[r, "(L_n)_\sharp"] \arrow[d, "\eta^{(n+1)}_{y_{n+1}}"'] \arrow[rrr, "f_n", bend left=20] & \pi_3(W_n, \widehat y_{n+1}) \arrow[r, "(\widehat \gamma_n)_\sharp"] \arrow[d, "\widetilde{\eta}^{(n)}_{\widehat y_{n+1}}"'] & \pi_3(W_n, \widehat y_n) \arrow[r, "(c_n)_\sharp"] & \pi_3(Y_n, y_n) \\
H_3(Y_{n+1}) \arrow[r, "(L_n)_*"']                                                                                        & H_3(W_n) \arrow[ru, "(\widetilde{\eta}^{(n)}_{\widehat y_n})^{-1}"']                                                           &                                                       &
\end{tikzcd}\]

Substituting this gives: $\lambda_n \circ \mu_{n+1} = (c_n)_\sharp \circ \left( (\widetilde{\eta}^{(n)}_{\widehat{y}_n})^{-1} \circ \widetilde{\eta}^{(n)}_{\widehat{y}_{n+1}} \right) \circ (L_n)_\sharp$. 
By the Hurewicz homomorphism, $(\widetilde{\eta}^{(n)}_{\widehat{y}_n})^{-1} \circ \widetilde{\eta}^{(n)}_{\widehat{y}_{n+1}} = (\widehat{\gamma}_n)_\sharp$. Thus $\lambda_n \circ \mu_{n+1} = (c_n)_\sharp \circ (\widehat{\gamma}_n)_\sharp \circ (L_n)_\sharp $. 

Because $c_n$ projects the lifted path $\widehat{\gamma}_n$ onto $\gamma_n$, we have $(c_n)_\sharp \circ (\widehat{\gamma}_n)_\sharp = (\gamma_n)_\sharp \circ (c_n)_\sharp$. Therefore $\lambda_n \circ \mu_{n+1} = (\gamma_n)_\sharp \circ (c_n \circ L_n)_\sharp = (\gamma_n)_\sharp \circ (J_n)_\sharp = f_n$.

Since both identities hold for all indices $n$, $\underline{\Pi}_3$ is pro-isomorphic to $\underline{H}_3$. By Theorem~\ref{proinv}, $\wp(\underline{\Pi}_3) \cong \wp(\underline{H}_3)$, concluding the proof that $\underline{\pi}_3(M, \underline{r}) \cong \underline{H}_3(M, e)$.
\end{proof}

\subsection{Abelianization Isomorphism in Degree One}
Finally, we prove the $k=1$ case using the geometric version $\underline{H}_1^\mathsf{geo}(M,e)$ of $\underline H_1(M,e)$. 

\begin{theorem}\label{thm:Hur1}
    The abelianization $\operatorname{Ab}(\underline{\pi}_1(M,\underline{r}))$ is isomorphic to $\underline{H}_1(M,e)$.
\end{theorem}

To prove Theorem~\ref{thm:Hur1}, we adapt several lemmas from \cite[Chapter IV, \S3]{MR1224675}, which are used to prove the classical one-dimensional Hurewicz theorem, by incorporating necessary additional information.

Let $V$ be a topological space. For any path $f \colon [0, 1] \to V$, let $c_f \in C_1(V)$ denote the corresponding singular $1$-chain. Let $f * g$ denote the concatenation of paths $f$ and $g$ (where $f$ is traversed first).  Let $x_0 \in V$ be a basepoint, and let $e_{x_0}$ denote the constant path at $x_0$. 
\begin{lemma}\textup{\cite[p. 173]{MR1224675}}\label{lem:H1}
    Let $f$ and $g$ be paths in $V$ such that $f(1) = g(0)$. Then there exists a singular $2$-chain $\sigma \in C_2(V)$ such that $\partial_2\sigma=c_{f * g} - c_f - c_g$. Furthermore, $\operatorname{supp}(\sigma) = \operatorname{im}(f) \cup \operatorname{im}(g)$.
\end{lemma}
\begin{lemma}\textup{\cite[p. 173]{MR1224675}}\label{lem:H2}
     If $f$ is a path in $V$, then there exists a singular $2$-chain $\sigma \in C_2(V)$ with $\operatorname{supp}(\sigma)=\operatorname{im}(f)$ such that $\partial_2\sigma=c_f + c_{f^{-1}}$. Moreover, any constant path is a boundary via a $2$-chain supported at that single point. 
\end{lemma}
\begin{lemma}\textup{\cite[p. 173]{MR1224675}}\label{lem:H3}
      Let $H\colon I\times I\to V$ be a homotopy rel $\partial I$ between paths $f$ and $g$. Then there exists a $2$-chain $\sigma\in C_2(V)$ with $\operatorname{supp}(\sigma)\subseteq\operatorname{im}(H)$ such that $\partial_2\sigma=c_f - c_g$.
\end{lemma}
\begin{lemma}\textup{\cite[p. 174]{MR1224675}}\label{lem:H4}
    Let $\sigma \colon \Delta_2 \to V$ be a $2$-simplex. Let $y_i = \sigma(e_i)$ for $i=0,1,2$, and let $f = \sigma^{(2)}$, $g = \sigma^{(0)}$, and $h = (\sigma^{(1)})^{-1}$ be its face paths. Choose connecting paths $\lambda_{y_i}$ from $x_0$ to $y_i$. For any path $p$ between vertices, denote the conjugated loop at $x_0$ by $\widehat{p} \coloneqq \lambda_{p(0)} * p * \lambda_{p(1)}^{-1}$. Then the concatenated loop $\widehat{f} * \widehat{g} * \widehat{h}$ is null-homotopic in $V$ relative to $x_0$, and there exists a null-homotopy $H$ rel $\{x_0\}$ from $\widehat{f} * \widehat{g} * \widehat{h}$ to $e_{x_0}$ such that $\operatorname{im}(H)\subseteq \operatorname{im}(\sigma) \cup \bigcup_{i=0}^2 \operatorname{im}(\lambda_{y_i})$.
\end{lemma}
\begin{lemma}\textup{\cite[p. 175]{MR1224675}}\label{lem:H5}
    Let $\sigma$ be a $1$-simplex in $V$ with vertices $y_0 = \sigma(e_0)$ and $y_1 = \sigma(e_1)$. Let $\lambda_{y_0}, \lambda_{y_1}$ be paths from a basepoint $x_0$ to these vertices, and define the loop $\widehat{\sigma} = \lambda_{y_0} * \sigma * \lambda_{y_1}^{-1}$. There exists a singular $2$-chain $\tau_\sigma \in C_2(V)$ such that $\partial_2 \tau_\sigma = c_{\widehat{\sigma}} - c_\sigma - c_{\lambda_{y_0}} + c_{\lambda_{y_1}}$ and $\operatorname{supp}(\tau_\sigma) \subseteq \operatorname{im}(\sigma) \cup \operatorname{im}(\lambda_{y_0}) \cup \operatorname{im}(\lambda_{y_1})$. Consequently, if $c = \sum m_i \sigma_i$ is a $1$-cycle and we linearly extend this construction to define $c_{\widehat{c}} \coloneqq \sum m_i c_{\widehat{\sigma}_i}$ and $\tau_c \coloneqq \sum m_i \tau_{\sigma_i}$, then $\partial_2 \tau_c = c_{\widehat{c}} - c$.
\end{lemma}
The proofs of the next two lemmas are straightforward and are therefore omitted. For instance, Lemma~\ref{lem:H6} follows directly from Lemma~\ref{lem:H4}.

\begin{lemma}\label{lem:H6}
    Let $b = \sum_{j} n_j \omega_j \in C_2(V)$ be a $2$-chain, and let $E(b)$ be the set of all vertices of the simplices in $b$. For each $y \in E(b)$, choose a connecting path $\lambda_y$ from $x_0$ to $y$. Let the boundary be given by the indexed sum $\partial_2 b = \sum_{i=1}^k m_i \sigma_i$. Extending the loop assignment from Lemma \ref{lem:H4}, we form the concatenated loop $\ell_{\partial_2 b} \coloneqq \prod_{i=1}^k \left( \lambda_{\sigma_i(e_0)} * \sigma_i * \lambda_{\sigma_i(e_1)}^{-1} \right)^{m_i}$. 
    
    Then $\ell_{\partial_2 b}$ represents the trivial element in $\operatorname{Ab}(\pi_1(V, x_0))$, and there exists a null-homotopy $H$ rel $\{x_0\}$ from a representative of this class to $e_{x_0}$ such that $\operatorname{im}(H) \subseteq \operatorname{supp}(b) \cup \bigcup_{y \in E(b)} \operatorname{im}(\lambda_y)$.
\end{lemma}

\begin{lemma}\label{lem:H7}
    Let $\alpha \colon [0, 1] \to V$ be a loop based at $x_0$. Let $c_\alpha \in Z_1(V)$ be the corresponding singular $1$-cycle. If we choose the constant path $e_{x_0}$ as the connecting path for the basepoint vertex $x_0$, the assigned loop from Lemma \ref{lem:H5} is $\widehat{\alpha} = e_{x_0} * \alpha * e_{x_0}^{-1}$. Then $\widehat{\alpha}$ is homotopic to $\alpha$ rel $\{x_0\}$ via a homotopy $H$ such that $\operatorname{im}(H) = \operatorname{im}(\alpha)$.
\end{lemma}

We are now ready to define the proper Hurewicz map and establish its bijectivity. Recall that from this point onward, we will use the geometric version of the proper homology groups.

Let $[\alpha] \in \underline{\pi}_1(M, \underline{r})$ be represented by a proper map $\alpha \colon \underline{\mathbb{S}^1} \to M$. Without loss of generality, we may assume that $\alpha|_{\underline{*}} = r$. Let $\pi \colon I \to \mathbb{S}^1$ be defined by $\pi(t)=e^{2\pi i t}$ for $t\in I$.

For each $k \ge 0$, let $\alpha_k$ be the restriction of $\alpha$ to $\mathbb{S}^1_k$. We can view the composition $\alpha_k \circ \pi \colon I \to M$ as a path in $M$ based at $r(k)$. We define a singular $1$-chain $c_k^\alpha \in C_1(M)$ by $c_k^\alpha \coloneqq c_{\alpha_k \circ \pi}$. Since $\alpha_k \circ \pi$ maps the endpoints of $I$ to the same basepoint $r(k)$, $\partial_1(c_k^\alpha) = 0$, meaning $c_k^\alpha \in Z_1(M)$.

Because $\alpha$ is a proper map, $\{\operatorname{supp}(c_k^\alpha)\} = \{\alpha(\mathbb{S}^1_k)\}$ converges to the end $e$. Thus, $\{c_k^\alpha\}_{k \geq 0} \in \underline{Z}_1(M, e)$. We define the proper Hurewicz map as:
\[
  \underline{\mathrm{Hu}} \colon \underline{\pi}_1(M, \underline{r}) \to \underline{H}_1(M, e), \quad [\alpha] \mapsto [\{c_k^\alpha\}]_{\underline{H}_1}
\]
\begin{lemma}\label{lem:properH1}
    If $\alpha$ and $\beta$ are germ homotopic relative to $\underline{*}$, then $[\{c_k^\alpha\}]_{\underline{H}_1} = [\{c_k^\beta\}]_{\underline{H}_1}$. Thus, $\underline{\mathrm{Hu}}$ is well-defined.
\end{lemma}
\begin{proof}
    Let $H \colon \underline{\mathbb{S}^1} \times [0, 1] \to M$ be a germ homotopy between $\alpha$ and $\beta$. Thus, for sufficiently large $k$, the homotopy on the basepoints $H(k, -)$ is identically the constant path at $r(k)$. For each $k \gg 0$, the restriction of $H$ to $\mathbb{S}^1_k \times [0, 1]$ precomposed with $\pi \times \mathrm{id}_{[0,1]}$ yields a homotopy $H_k \colon I \times I \to M$ relative to $\partial I$ between the paths $\alpha_k \circ \pi$ and $\beta_k \circ \pi$. By Lemma \ref{lem:H3}, there exists a singular $2$-chain $b_k \in C_2(M)$ with $\operatorname{supp}(b_k) \subseteq \operatorname{im}(H_k)$ such that $\partial_2(b_k) = c_k^\alpha - c_k^\beta$ for all $k\gg 0$.
    
    Because $H$ is a proper map, $\{\operatorname{im}(H_k)\} = \{H(\mathbb{S}^1_k \times [0, 1])\}$ converges to the end $e$. Since $\operatorname{supp}(b_k) \subseteq \operatorname{im}(H_k)$, the sequence $\{\operatorname{supp}(b_k)\}$ also converges to $e$. Therefore, $\{b_k\} \in \underline{C}_2(M, e)$. Since $\partial_2(b_k) = c_k^\alpha - c_k^\beta$ holds for all $k \gg 0$, we have $\{c_k^\alpha\} - \{c_k^\beta\} \in \underline{B}_1(M, e)$. This shows the homology class is independent of the chosen representative, proving that $\underline{\mathrm{Hu}}$ is well-defined.
\end{proof}

\begin{lemma}\label{lem:properH2}
    $\underline{\mathrm{Hu}}([\alpha \cdot \beta]) = \underline{\mathrm{Hu}}([\alpha]) + \underline{\mathrm{Hu}}([\beta])$.
\end{lemma}

\begin{proof}
    Let $[\alpha], [\beta] \in \underline{\pi}_1(M, \underline{r})$ be represented by proper maps of pairs $\alpha, \beta$ that agree exactly on $\underline{*}$. The group operation $\alpha \cdot \beta$ is defined by standard path concatenation on each attached sphere $\mathbb{S}^1_k$. Let $c_k^{\alpha \cdot \beta}$ be the corresponding singular $1$-cycle for the concatenated map on $\mathbb{S}^1_k$. Because $(\alpha_k \circ \pi)(1) = r(k) = (\beta_k \circ \pi)(0)$, by Lemma \ref{lem:H1}, there exists a singular $2$-chain $\sigma_k \in C_2(M)$ such that:
    $\partial_2 \sigma_k = c_k^{\alpha \cdot \beta} - c_k^\alpha - c_k^\beta$ for all  $k \ge 0,$ and $\operatorname{supp}(\sigma_k) = \operatorname{im}(\alpha_k \circ \pi) \cup \operatorname{im}(\beta_k \circ \pi) = \alpha(\mathbb{S}^1_k) \cup \beta(\mathbb{S}^1_k).$
    
    Since the sequences of sets $\{\alpha(\mathbb{S}^1_k)\}$ and $\{\beta(\mathbb{S}^1_k)\}$ both converge to the end $e$, the sequence $\{\operatorname{supp}(\sigma_k)\}$ converges to $e$, meaning $\{\sigma_k\} \in \underline{C}_2(M, e)$. Moreover, because $\partial_2(-\sigma_k) = c_k^\alpha + c_k^\beta - c_k^{\alpha \cdot \beta}$ holds for all $k \ge 0$, we conclude that $\{c_k^\alpha\} + \{c_k^\beta\} - \{c_k^{\alpha \cdot \beta}\} \in \underline{B}_1(M, e).$
    
    Passing to quotients in the proper homology group, this gives $[\{c_k^{\alpha \cdot \beta}\}]_{\underline{H}_1} = [\{c_k^\alpha\}]_{\underline{H}_1} + [\{c_k^\beta\}]_{\underline{H}_1}$, proving $\underline{\mathrm{Hu}}$ is a homomorphism.
\end{proof}

Since $\underline{H}_1(M, e)$ is an abelian group, $\underline{\mathrm{Hu}}$ naturally induces a group homomorphism from the abelianization of $\underline{\pi}_1(M, \underline{r})$. We will denote this induced map by the same symbol:
\begin{equation}\label{eq1}
\underline{\mathrm{Hu}} \colon \operatorname{Ab}(\underline{\pi}_1(M, \underline{r})) \to \underline{H}_1(M, e), \quad [\alpha]_{\operatorname{Ab}(\underline \pi_1)} \mapsto [\{c_k^\alpha\}]_{\underline{H}_1}.
\end{equation}
The remainder of this section is devoted to constructing the inverse of \eqref{eq1}. We first define a group homomorphism $\underline{K} \colon \underline{C}_1(M, e) \to \operatorname{Ab}(\underline{\pi}_1(M, \underline{r}))$. Let $\{c_k\} \in \underline{C}_1(M, e)$. Because $\{\operatorname{supp}(c_k)\}$ converges to the end $e$, there exists an integer $N \ge 0$ and a sequence of integers $l(k) \to \infty$ such that $\operatorname{supp}(c_k) \subseteq U_{l(k)}$ for all $k \ge N$. Then for each $k \ge N$, both the vertices of $c_k$ and the basepoint $r(k)$ lie within the path-connected set $U_{m(k)}$, where $m(k) = \min(k, l(k))$. 

For each $k \ge N$, and for every vertex $y \in \operatorname{supp}(c_k)$, choose a connecting path $\lambda_k^y \colon [0, 1] \to U_{m(k)}$ from the basepoint $r(k)$ to $y$. Write the $1$-chain $c_k$ as a finite formal sum of ordered singular simplices: $c_k = \sum_{i=1}^{t_k} m_{i,k} \sigma_{i,k}$. Let $\alpha \colon \underline{\mathbb{S}^1} \to M$ be the proper map given by $\alpha\vert_{\underline{*}} = r$. Identifying $I / \partial I$ with $\mathbb{S}^1_k$ via the quotient map $\pi$, we define the loop $\alpha\vert_{\mathbb{S}^1_k}$ in $U_{m(k)}$ based at $r(k)$ for $k \ge N$ by:
\begin{equation}\label{eq:K}
    \alpha\vert_{\mathbb{S}^1_k} \coloneqq \prod_{i=1}^{t_k} \left( \lambda_k^{\sigma_{i,k}(e_0)} * \sigma_{i,k} * (\lambda_k^{\sigma_{i,k}(e_1)})^{-1} \right)^{m_{i,k}} 
\end{equation}
For $k < N$, we simply define $\alpha\vert_{\mathbb{S}^1_k}$ to be the constant loop at $r(k)$. Because the germ of $\alpha$ depends only on the behavior for sufficiently large $k$, this construction uniquely determines a germ class $[\alpha]_{\underline{\pi}_1}$ in $\operatorname{Ab}(\underline{\pi}_1(M, \underline{r}))$, i.e., $[\alpha]_{\underline{\pi}_1}$ is independent of $N$. We define $\underline{K}(\{c_k\}) \coloneqq [\alpha]_{\underline{\pi}_1}$.

\begin{lemma}\label{lem:properH3}
    If $\{c_k\} \in \underline{B}_1(M, e)$, then $\underline{K}(\{c_k\}) = 0$ in $\operatorname{Ab}(\underline{\pi}_1(M, \underline{r}))$. Consequently, $\underline{K}$ induces a well-defined group homomorphism $\underline{K} \colon \underline{H}_1(M, e) \to \operatorname{Ab}(\underline{\pi}_1(M, \underline{r}))$.
\end{lemma}

\begin{proof}
    Let $\{c_k\} \in \underline{B}_1(M, e)$. By definition, there exists a proper $2$-chain $\{b_k\} \in \underline{C}_2(M, e)$ such that $c_k = \partial_2(b_k)$ for all $k \gg 0$.
    
    Since the sequence $\{\operatorname{supp}(b_k)\}$ converges to the end $e$, there exists an integer $N_1 \ge 0$ and a sequence of integers $l'(k)$ diverging to $\infty$ such that $\operatorname{supp}(b_k) \subseteq U_{l'(k)}$ for all $k \ge N_1$. Choose an integer $N \ge N_1$ large enough such that $c_k = \partial_2(b_k)$ for all $k \ge N$. Thus, for all $k \ge N$, we have $\operatorname{supp}(c_k) \subseteq \operatorname{supp}(b_k) \subseteq U_{l'(k)}$. We can therefore use this integer $N$ and the sequence $l'(k)$ as the bounds in the construction of the representative map $\alpha$ for $\underline{K}(\{c_k\})$.
    
    For each $k \ge N$, we evaluate the chain-level assignment of $\alpha\vert_{\mathbb{S}^1_k}$ on the boundary $c_k = \partial_2(b_k)$. Using the connecting paths $\lambda_k^y$ defined in the construction of $\underline{K}$ (which are contained in $U_{m(k)}$ where $m(k) = \min(k, l'(k))$), we apply Lemma~\ref{lem:H6} to the $2$-chain $b_k$ within the space $V = U_{m(k)}$ with basepoint $r(k)$.

    By Lemma~\ref{lem:H6}, the concatenated loop $\ell_{\partial_2 b_k}$ assigned to the boundary $\partial_2(b_k)$ represents the trivial element in $\operatorname{Ab}(\pi_1(U_{m(k)}, r(k)))$. Furthermore, Lemma~\ref{lem:H6} guarantees the existence of a null-homotopy $H_k \colon \mathbb{S}^1_k \times [0, 1] \to M$ relative to $r(k)$ contracting this loop such that $\operatorname{im}(H_k) \subseteq \operatorname{supp}(b_k) \cup \bigcup_{y \in E(b_k)} \operatorname{im}(\lambda_k^y)$. Because $\operatorname{supp}(b_k) \subseteq U_{l'(k)}$ and $\operatorname{im}(\lambda_k^y) \subseteq U_{m(k)}$, the image of $H_k$ is contained within $U_{l'(k)} \cup U_{m(k)} = U_{m(k)}$. 
    
    For $k < N$, the map $\alpha\vert_{\mathbb{S}^1_k}$ is defined as the constant loop at $r(k)$. We simply define $H_k \colon \mathbb{S}^1_k \times [0, 1] \to M$ for $k < N$ to be the constant homotopy at $r(k)$.

    We now assemble these components into a single map $H \colon \underline{\mathbb{S}^1} \times [0, 1] \to M$ as follows: On $\underline{*} \times [0, 1]$, $H$ is stationary, i.e., $H(x,t)=r(x)$ for all $(x,t)\in \underline{*} \times [0,1]$. On each $\mathbb{S}^1_k \times [0, 1]$, $H$ equals the null-homotopy $H_k$ for $k \ge 0$.
    
    Because $k \to \infty$ and $l'(k) \to \infty$, the index $m(k) \to \infty$. Therefore, the sequence $\{\operatorname{im}(H_k)\}$ converges to the end $e$, proving that $H$ is a valid germ homotopy rel $\underline{*}$. This homotopy connects the proper map of pairs $\alpha$ representing $\underline{K}(\{c_k\})$ to the trivial proper map of pairs (which constantly maps $\mathbb{S}^1_k$ to $r(k)$ for every $k \ge 0$). Thus, $\underline{K}(\{c_k\}) = 0 \in \operatorname{Ab}(\underline{\pi}_1(M, \underline{r}))$, and the map passes uniquely to the quotient.
\end{proof}

\begin{lemma} \label{lem:properH4}
    The homomorphism $\underline K$ is a left inverse of \eqref{eq1}.
\end{lemma}

\begin{proof}
    Let $[\alpha] \in \underline{\pi}_1(M, \underline{r})$ be an equivalence class represented by a proper map of pairs $\alpha \colon (\underline{\mathbb{S}^1}, \underline{*})\allowbreak \to (M, \underline{r})$. By choosing a representative that replaces $\alpha_k$ with the constant loop at $r(k)$ for finitely many $k$, and ensuring $\alpha\vert_{\underline{*}} = r$ exactly, we may assume without loss of generality that the entire image satisfies $\operatorname{im}(\alpha) \subseteq U_0$. Thus, there exists a sequence of integers $m(k) \to \infty$ such that $\operatorname{im}(\alpha_k) \subseteq U_{m(k)}$ for all $k \ge 0$, where $\alpha_k$ is the restriction of $\alpha$ to the attached circle $\mathbb{S}^1_k$.
    
    Applying the proper Hurewicz map gives $\underline{\mathrm{Hu}}([\alpha]) = [\{c_k^\alpha\}]_{\underline{H}_1}$, where for each $k \ge 0$, the $1$-cycle $c_k^\alpha \in Z_1(M)$ is the singular $1$-chain $c_k^\alpha = c_{\alpha_k \circ \pi}$. Because $\alpha_k$ maps the basepoint of $\mathbb{S}^1_k$ to $r(k)$, both endpoints of the $1$-simplex $\alpha_k \circ \pi$ lie exactly at $r(k)$.
    
    Recall that constructing $\underline{K}$ requires choosing connecting paths $\lambda_k^y$ from the basepoint $r(k)$ to the vertices of the $1$-chains. Since the vertices are already at $r(k)$, we choose $\lambda_k^{r(k)}$ to be the constant path at $r(k)$, denoted $e_{r(k)}$. This constant path trivially satisfies the condition of being contained within $U_{m(k)}$.

    Let $\widehat{\alpha} \colon \underline{\mathbb{S}^1} \to M$ be the proper map given by $\widehat{\alpha}\vert_{\underline{*}} = r$ and $\widehat{\alpha}_k \coloneqq \widehat{\alpha}\vert_{\mathbb{S}^1_k} = e_{r(k)} * (\alpha_k \circ \pi) * e_{r(k)}^{-1}$. By the definition of $\underline{K}$, the element $\underline{K} \circ \underline{\mathrm{Hu}}([\alpha]_{\operatorname{Ab}(\underline{\pi}_1)})$ is represented by $\widehat{\alpha}$. Therefore, to conclude the proof it is enough to show that $\widehat{\alpha}$ is germ homotopic to $\alpha$. 
    
    By Lemma~\ref{lem:H7}, there exists a homotopy $H_k \colon \mathbb{S}^1_k \times [0, 1] \to M$ between $\widehat{\alpha}_k$ and $\alpha_k$ relative to the basepoint $r(k)$ such that $\operatorname{im}(H_k) = \operatorname{im}(\alpha_k \circ \pi) = \alpha(\mathbb{S}^1_k)$. Since the sequence $\{\alpha(\mathbb{S}^1_k)\}$ converges to $e$, the proper homotopy $H \colon \underline{\mathbb{S}^1} \times [0,1] \to M$ defined by $H(x,t) \coloneqq r(x)$ for all $(x,t) \in \underline{*} \times [0,1]$ and $H\vert_{\mathbb{S}^1_k \times [0,1]} \coloneqq H_k$, is a germ homotopy between $\widehat{\alpha}$ and $\alpha$.
    
    This concludes the proof that  $\underline{K} \circ \underline{\mathrm{Hu}} = \operatorname{id}_{\operatorname{Ab}(\underline{\pi}_1)}$.
\end{proof}

\begin{lemma}\label{lem:properH5}
    The homomorphism $\underline K$ is a right inverse of \eqref{eq1}.
\end{lemma}

\begin{proof}
    Let $[\{c_k\}] \in \underline{H}_1(M, e)$. By definition, we can choose a representative $\{c_k\} \in \underline{Z}_1(M, e)$ such that for all $k \ge 0$, $\partial_1(c_k) = 0$ and $\operatorname{supp}(c_k) \subseteq U_{l(k)}$ for some sequence of integers $l(k) \to \infty$. We write each $1$-chain as a finite formal sum $c_k = \sum_{i=1}^{t_k} m_{i,k} \sigma_{i,k}$. Let $E(c_k)$ be the set of all vertices of the simplices in $c_k$.
    
    Applying $\underline{K}$ to $[\{c_k\}]$ yields an element $[\alpha]_{\operatorname{Ab}(\underline{\pi}_1)}$ where $\alpha\vert_{\mathbb{S}^1_k}$ is a loop in $U_{m(k)}$ based at $r(k)$ given by:
    \[
    \alpha\vert_{\mathbb{S}^1_k} \coloneqq \prod_{i=1}^{t_k}\left( \lambda_k^{\sigma_{i,k}(e_0)} * \sigma_{i,k} * (\lambda_k^{\sigma_{i,k}(e_1)})^{-1} \right)^{m_{i,k}}.
    \]
    Here, $m(k) = \min(k, l(k))$ and $\lambda_k^y \colon [0, 1] \to U_{m(k)}$ is a path from the basepoint $r(k)$ to $y$.

     Applying $\underline{\mathrm{Hu}}$ to $\underline K([\{c_k\}])$ gives the proper homology class of the sequence of singular $1$-chains $\{c_k^\alpha\}$, where $c_k^\alpha$ is the $1$-chain of the concatenated loop $\alpha\vert_{\mathbb{S}^1_k}$.

     On the other hand, let $c_{\widehat{c_k}} \coloneqq \sum_{i=1}^{t_k} m_{i,k} c_{\widehat{\sigma}_{i,k}}$ be the linearly extended $1$-chain defined in Lemma \ref{lem:H5}. By repeated application of Lemma \ref{lem:H1} for path concatenations, the $1$-chain of the concatenated loop $c_k^\alpha$ is homologous to the formal sum $c_{\widehat{c_k}}$ via a $2$-chain $\eta_k \in C_2(U_{m(k)})$, i.e., $\partial_2 \eta_k = c_k^\alpha - c_{\widehat{c_k}}$.
    
    Furthermore, by applying Lemma \ref{lem:H5} within the space $U_{m(k)}$, there exists a singular $2$-chain $\tau_{c_k} \in C_2(U_{m(k)})$ such that $\partial_2 \tau_{c_k} = c_{\widehat{c_k}} - c_k$. Thus, $\operatorname{supp}(\tau_{c_k}) \subseteq U_{m(k)}$. Therefore, the combined $2$-chain $b_k = \eta_k + \tau_{c_k}$ is supported in $U_{m(k)}$ and satisfies $\partial_2(b_k) = c_k^\alpha - c_k$. Since $m(k) \to \infty$ as $k \to \infty$, the sequence $\{\operatorname{supp}(b_k)\}$ converges to the end $e$, meaning $\{b_k\} \in \underline{C}_2(M, e)$. Since $\partial_2(b_k) = c_k^\alpha - c_k$ holds for all $k \ge 0$, this means that $\{c_k^\alpha\} - \{c_k\} \in \underline{B}_1(M, e)$. Therefore, $\underline{\mathrm{Hu}} \circ \underline{K}([\{c_k\}]_{\underline{H}_1}) = [\{c_k^\alpha\}]_{\underline{H}_1} = [\{c_k\}]_{\underline{H}_1}$. 
    
    This concludes the proof that $\underline{\mathrm{Hu}} \circ \underline{K} = \mathrm{id}_{\underline{H}_1}$.  
\end{proof}

\begin{proof}[Proof of Theorem~\ref{thm:Hur1}]
    By Lemma~\ref{lem:properH1} and ~\ref{lem:properH2}, we have a well-defined group homomorphism $\underline{\mathrm{Hu}} \colon \operatorname{Ab}(\underline{\pi}_1(M, \underline{r})) \to \underline{H}_1(M, e)$. By Lemmas~\ref{lem:properH3},~\ref{lem:properH4}, and~\ref{lem:properH5}, $\underline K$ is the inverse of $\underline{\operatorname{Hu}}.$ 
\end{proof}
\section{Computations and Applications}\label{section:computation}
\subsection{Proper Fundamental Groups of Exotic Contractible \texorpdfstring{$3$}{3}-Manifolds}\label{section:exotic}
In dimensions $n=1,2$, it is known that every contractible open $n$-manifold is homeomorphic to $\mathbb R^n$. However, this is not the case for $n \ge 3$. In 1932, J. H. C. Whitehead attempted to prove a non-compact version of the Poincaré Conjecture---namely, that every contractible open $3$-manifold is homeomorphic to $\mathbb R^3$. He later corrected this error in \cite{zbMATH02532945} by constructing the famous Whitehead manifold, a contractible open $3$-manifold not homeomorphic to $\mathbb R^3$. Generalizing Whitehead's approach, McMillan~\cite{MR137105} later constructed uncountably many contractible open subsets of $\mathbb R^3$ that are pairwise non-homeomorphic.

To distinguish $\mathbb R^3$ from this uncountable collection, Brown--Tucker used the proper fundamental group. They proved that an open, irreducible, contractible $3$-manifold with one end has trivial $\underline \pi_1$ if and only if it is homeomorphic to $\mathbb R^3$~\cite[Corollary 3.3]{MR334225}. Their proof relies on the Open Collar Theorem: If $M$ is a connected, orientable, irreducible, one-ended $3$-manifold such that $M$ has no spherical boundary components, $\underline \pi_1(M)$ vanishes, and $M$ is not homeomorphic to $\mathbb R^3$, then $\partial M$ is non-empty and connected, and $M$ is homeomorphic to $\partial M \times [0,\infty)$~\cite[Theorem 3.2]{MR334225} \cite[p. 26]{MR2621405}.

Until now, the only known fact regarding the proper fundamental groups of exotic, one-ended, open, contractible $3$-manifolds was that they are non-trivial. In this section, we provide additional information about these proper fundamental groups; specifically, we prove that they are uncountable, perfect groups. Moreover, we show that the ``one-ended'' condition, which Brown--Tucker assumed throughout their work on contractible open manifolds, is actually redundant.

\begin{theorem}\label{thm:contra}
    Let $\mathcal W$ be an open, contractible $3$-manifold. Then $\mathcal W$ has exactly one end and $\underline H_r(\mathcal W)$ is isomorphic to $\underline H_r(\mathbb R^3)$ for every non-negative integer $r$. In particular, $\underline H_r(\mathcal W)$ is isomorphic to $\left.\prod_{m\geq 0} \mathbb Z\right/\bigoplus_{m\geq 0} \mathbb Z$ for $r \in \{0, 2\}$, and is the trivial group otherwise. Moreover, if $\mathcal W$ is not homeomorphic to $\mathbb R^3$, then $\underline \pi_1(\mathcal W)$ is uncountable and perfect.
\end{theorem}
We need a few lemmas to prove Theorem~\ref{thm:contra}. First, we show that open contractible manifolds of dimension $n\geq2$ are one-ended.
\begin{lemma}\label{lem:oneend}
    Let $M$ be a connected, non-compact $n$-manifold such that $H_k(M;\mathbb{Z})$ is finitely generated for all $k \geq 0$. If $\partial M = \varnothing$, then the space of ends of $M$ is a finite set of cardinality at most $1 + \dim_{\mathbb{Z}_2} H_{n-1}(M;\mathbb{Z}_2)$.
\end{lemma}

\begin{remark}
    The hypothesis $\partial M = \varnothing$ is necessary for Lemma~\ref{lem:oneend}. To see this, let $\mathcal{E} \subseteq \mathbb{S}^1$ be an infinite closed subset of the Cantor set. The space $\{ z \in \mathbb{C} : |z| \leq 1 \} \setminus \mathcal{E}$ is a contractible $2$-manifold with non-empty boundary, yet it has infinitely many ends.
\end{remark}

\begin{proof}[Proof of Lemma~\ref{lem:oneend}]
    Assume $\partial M = \varnothing$. By \cite[\S1.4]{MR275436} or \cite[\S 2]{MR120637}, there is an exact sequence 
    $$\cdots \to H^0_c(M;\mathbb{Z}_2) \xrightarrow{\varphi_0} H^0(M;\mathbb{Z}_2) \xrightarrow{\psi_0} H^0_e(M;\mathbb{Z}_2) \xrightarrow{\delta_0} H^1_c(M;\mathbb{Z}_2) \xrightarrow{\varphi_1} H^1(M;\mathbb{Z}_2) \to \cdots$$ 
    where $H^*$, $H^*_e$, and $H^*_c$ denote singular, end, and compactly supported cohomology, respectively. By Poincaré duality \cite[Exercise~35, Chapter~3.3]{MR1867354}, we have $H^1_c(M;\mathbb{Z}_2) \cong H_{n-1}(M;\mathbb{Z}_2)$. Since the integral homology of $M$ is finitely generated, the Universal Coefficient Theorem implies that $H_{n-1}(M;\mathbb{Z}_2)$ is finite-dimensional. Hence $H^1_c(M;\mathbb{Z}_2)$ is a finite-dimensional vector space over $\mathbb{Z}_2$.

    Because $M$ is connected, $H^0(M;\mathbb{Z}_2) \cong \mathbb{Z}_2$. Furthermore, since $M$ is non-compact, Poincaré duality gives $H^0_c(M;\mathbb{Z}_2) \cong H_n(M;\mathbb{Z}_2) = 0$, and hence $\operatorname{im}(\varphi_0) = 0$. By exactness, $\ker \psi_0 = 0$, and consequently $\dim_{\mathbb{Z}_2} \operatorname{im}(\psi_0) = \dim_{\mathbb{Z}_2} H^0(M;\mathbb{Z}_2) = 1$. 
    
    Moreover, the quotient vector space $H^0_e(M;\mathbb{Z}_2) / \operatorname{im}(\psi_0) = H^0_e(M;\mathbb{Z}_2) / \ker \delta_0$ is isomorphic via $\delta_0$ to a subspace of the finite-dimensional vector space $H^1_c(M;\mathbb{Z}_2)$. Therefore, 
    \begin{align*}
    \dim_{\mathbb{Z}_2} H^0_e(M;\mathbb{Z}_2)
    &= \dim_{\mathbb{Z}_2} \operatorname{im}(\psi_0)
       + \dim_{\mathbb{Z}_2} \frac{H^0_e(M;\mathbb{Z}_2)}{\operatorname{im}(\psi_0)} \\
    &= 1 + \dim_{\mathbb{Z}_2} \frac{H^0_e(M;\mathbb{Z}_2)}{\operatorname{im}(\psi_0)} \\
    &\leq 1 + \dim_{\mathbb{Z}_2} H^1_c(M;\mathbb{Z}_2) \\
    &= 1 + \dim_{\mathbb{Z}_2} H_{n-1}(M;\mathbb{Z}_2),
    \end{align*}
    which is finite. By \cite[Remark, p.~244]{MR275436} or \cite[Theorem 1.13]{MR120637}, $H^0(\mathsf{Ends}(M);\mathbb{Z}_2)$ is isomorphic to $H^0_e(M;\mathbb{Z}_2)$. If $\alpha$ is the number of path components of $\mathsf{Ends}(M)$, then $H^0(\mathsf{Ends}(M);\mathbb{Z}_2) \cong \mathbb{Z}_2^{\alpha}$. Since the space $\mathsf{Ends}(M)$ is totally disconnected, $\alpha$ equals the number of elements of $\mathsf{Ends}(M)$. 
    
    Thus, the number of ends of $M$ is at most $1 + \dim_{\mathbb{Z}_2} H_{n-1}(M;\mathbb{Z}_2)$.
\end{proof}

The following lemma is useful for computing the proper homology group of an end shaped like an infinite cylinder over a closed, codimension-one submanifold. For instance, the interior of a compact, connected manifold with boundary always has ends of this form.
\begin{lemma}\label{example}
    Let $M$ be a connected non-compact manifold, and let $N$ be a connected, compact, codimension one, separating submanifold of $M$ such that the closure of one of the two sides of $M\setminus N$ is homeomorphic to $N\times [0,\infty)$. Denote the end corresponding to the side $N\times [0,\infty)$ by $e$. Then $\underline H_k^\mathsf{alg}(M,e)$ is isomorphic to $\left.\prod_{n\geq 0}H_k(N)\right/ \bigoplus_{n \ge 0} H_k(N)$.
\end{lemma}

\begin{proof}
    Let $C$ be the closure of the side of $M \setminus N$ corresponding to the end $e$, and let $\Phi \colon N \times [0, \infty) \xrightarrow{\cong} C$ be the homeomorphism given by the hypothesis, such that $\Phi(N \times \{0\}) = N$. For each integer $i \ge 0$, define the open set $U^{(i)} \coloneqq \Phi(N \times (i, \infty))$. Since $M$ admits an efficient exhaustion and $N$ is compact, we can choose an efficient exhaustion $K_0 \subseteq K_1 \subseteq K_2 \subseteq \cdots$ of $M$ such that $K_i \cap C = \Phi(N \times [0, i])$. By construction, $U^{(i)}$ is a connected component of $M \setminus K_i$. Thus, the end $e$ is represented by the sequence $(U^{(0)}, U^{(1)}, U^{(2)}, \ldots)$. Therefore, the inverse system $\underline{G} \coloneqq \{H_k(U^{(n)}), (p^m_n)_k\}$, where $(p^m_n)_k \colon H_k(U^{(m)}) \to H_k(U^{(n)})$ is induced by the inclusion $U^{(m)} \hookrightarrow U^{(n)}$, computes the proper homology group $\underline{H}_k(M,e)$.

    We claim that $\underline G$ is pro-isomorphic to the  system $\underline{H_{k}(N)}\coloneqq H_{k}(N) \xleftarrow{\operatorname{id}_{H_k(N)}} H_{k}(N) \xleftarrow{\operatorname{id}_{H_k(N)}} \cdots$, i.e.,  there exist homomorphisms $\lambda_i$ and $\mu_i$ such that the following ladder diagram commutes:
    \[
    \begin{tikzcd}[column sep=.5em]
        H_k(U^{(0)}) & & H_k(U^{(1)}) \arrow[ll, "(p^{1}_{0})_k"'] \arrow[ld, "\lambda_0"'] & & H_k(U^{(2)}) \arrow[ll, "(p^{2}_{1})_k"'] \arrow[ld, "\lambda_1"'] & & H_k(U^{(3)}) \arrow[ll, "(p^{3}_{2})_k"'] \arrow[ld, "\lambda_2"'] & & \cdots \arrow[ll] \arrow[ld] & \\
                & H_{k}(N) \arrow[lu, "\mu_0"] & & H_{k}(N) \arrow[lu, "\mu_1"] \arrow[ll, "\operatorname{id}_{H_k(N)}"] & & H_{k}(N) \arrow[lu, "\mu_2"] \arrow[ll, "\operatorname{id}_{H_k(N)}"] & & H_{k}(N) \arrow[lu, "\mu_3"] \arrow[ll, "\operatorname{id}_{H_k(N)}"] & & \cdots \arrow[ll]
    \end{tikzcd}
    \]
    
    To prove this claim, for each $i \ge 0$, define the inclusion map $\widetilde{\mu}_i \colon N \to U^{(i)}$ by 
    $\widetilde{\mu}_i(x) \coloneqq \Phi(x, i+1),$
    and define the projection map $\widetilde{\lambda}_i \colon U^{(i+1)} \to N$ by 
    $\widetilde{\lambda}_i(\Phi(x, t)) \coloneqq x.$ Let $\mu_i \coloneqq (\widetilde{\mu}_i)_*$ and $\lambda_i \coloneqq (\widetilde{\lambda}_i)_*$ be their respective induced homomorphisms on homology. We now verify two space-level compositions to establish commutativity:
    
    First, for any $i \ge 0$, the composition $\widetilde{\lambda}_i \circ \widetilde{\mu}_{i+1} \colon N \to N$ maps $x \mapsto \Phi(x, i+2) \mapsto x$. Thus, $\widetilde{\lambda}_i \circ \widetilde{\mu}_{i+1} = \operatorname{id}_N$ exactly, which implies 
    $\lambda_i \circ \mu_{i+1} = \operatorname{id}_{H_k(N)}.$
    
    Second, the composition $\widetilde{\mu}_i \circ \widetilde{\lambda}_i \colon U^{(i+1)} \to U^{(i)}$ maps $\Phi(x, t) \mapsto \Phi(x, i+1)$. We claim this is homotopic to the standard inclusion $p^{i+1}_i \colon U^{(i+1)} \hookrightarrow U^{(i)}$. Define the straight-line homotopy $H \colon U^{(i+1)} \times [0, 1] \to U^{(i)}$ by
    $H(\Phi(x, t), s) \coloneqq \Phi(x, (1-s)t + s(i+1)).$
    Because $\Phi(x, t) \in U^{(i+1)}$, we know $t > i+1$. For any $s \in [0, 1]$, the coordinate $(1-s)t + s(i+1)$ is a convex combination of $t$ and $i+1$, both of which are strictly greater than $i$. Thus, the image of $H$ lies entirely within $U^{(i)}$. This shows $\widetilde{\mu}_i \circ \widetilde{\lambda}_i \simeq p^{i+1}_i$, which induces 
    $\mu_i \circ \lambda_i = (p^{i+1}_i)_k.$

    Because the triangles commute, the inverse system $\underline{G}$ is pro-isomorphic to the constant inverse system $\underline{H_{k}(N)}$. The result then follows immediately from Theorem~\ref{constant}.
\end{proof}

\begin{corollary}\label{cor:Rn}
    If $n\geq 2$, the group $\underline H_k^\mathsf{alg}(\mathbb R^n,\infty)$ is isomorphic to $\left.\prod_{m\geq 0} \mathbb Z\right/\bigoplus_{m\geq 0} \mathbb Z$ for $k \in \{0, n-1\}$, and is the trivial group otherwise, where $\infty$ denotes the unique end of $\mathbb R^n$. Furthermore, for either end $e$ of $\mathbb R$, the group $\underline H_k^\mathsf{alg}(\mathbb R,e)$ is isomorphic to $\left.\prod_{m\geq 0} \mathbb Z\right/\bigoplus_{m\geq 0} \mathbb Z$ for $k=0$, and is the trivial group otherwise.
\end{corollary}
By Theorem~\ref{eff}, every connected, non-compact manifold admits an efficient exhaustion. However, for the remainder of this section, we require each term of the exhaustion to be a compact, codimension-zero submanifold.

Let $M$ be a connected, non-compact smooth $n$-manifold without boundary. A \emph{regular exhaustion} of $M$ is a sequence of subsets $K_0\subseteq K_1\subseteq K_2\subset\cdots$ such that $M=\bigcup_{i\geq 0}K_i$ and, for every $i\geq 0$, the following conditions hold: $K_i$ is a connected, compact, embedded $n$-dimensional submanifold with boundary; $K_i\subseteq \operatorname{int}(K_{i+1})$; and the closure of each component of $M\setminus K_i$ is non-compact. By \cite[Compact Submanifold Exhaustion Proposition]{friedlbook}, every such manifold $M$ admits a regular exhaustion.

The proof of the next lemma is well-known. We include it for the reader's convenience.
\begin{lemma}\label{lem:inclusionzero}
    Let $M$ be a connected, non-compact smooth $n$-manifold without boundary, and let $\{K_i\}$ be a regular exhaustion of $M$. If $H_1(M)=0$, then for every $i$, there exists $j>i$ such that the inclusion-induced map $H_1(K_i)\to H_1(K_j)$ is the zero map.
\end{lemma}
\begin{proof}
    Suppose $H_1(M)=0$. Let $i$ be a non-negative integer. Since $K_i$ is a compact manifold with boundary, $H_1(K_i)$ is finitely generated \cite[p. 538, Corollary E.5]{MR1224675}. Let $\{g_1, g_2, \dots, g_\ell\}$ be a finite set of generators for $H_1(K_i)$. For each generator $g_m \in H_1(K_i)$, choose a representing singular $1$-cycle $z_m$. Because $H_1(M) = 0$, there exists a singular $2$-chain $c_m$ in $M$ such that $\partial c_m = z_m$. Because a singular chain is a finite linear combination of singular simplices, its support $\operatorname{supp}(c_m)$ is compact. Since $\{\operatorname{int}(K_r)\}_{r\geq 0}$ forms a nested open cover of $M$, there exists an integer $i_m$ such that $\operatorname{supp}(c_m)\subseteq K_{i_m}$. Let $j\coloneqq \max\left(\{i+1\} \cup \{i_m \mid 1 \leq m \leq \ell\}\right)$. Then $j > i$, and the inclusion-induced map $H_1(K_i)\to H_1(K_j)$ is the zero map.
\end{proof}

The following lemma shows that the first homology of a neighborhood system of the unique end of an open, contractible $3$-manifold is pro-trivial.
\begin{lemma}\label{lem:firstcontr}
    Let $M$ be a connected, oriented, non-compact smooth $3$-manifold without boundary, and let $\{K_i\}$ be a regular exhaustion of $M$. Suppose that $H_1(M) = 0$ and $H_2(M) = 0$. Then for every $i$, there exists $j>i$ such that the inclusion-induced map $H_1(M\setminus K_j)\to H_1(M\setminus K_i)$ is the zero map.
\end{lemma}
\begin{proof}
    For $r \ge 0$, the exactness of the long exact sequence in homology for the pair $(M, M\setminus K_r)$,
    \[
        H_2(M) \to H_2(M, M\setminus K_r) \xrightarrow{\partial} H_1(M\setminus K_r) \to H_1(M),
    \]
    together with the assumptions $H_2(M) = 0$ and $H_1(M) = 0$, implies that the connecting homomorphism $\partial\colon H_2(M, M\setminus K_r) \to H_1(M\setminus K_r)$ is an isomorphism. 

    By Poincaré--Alexander--Lefschetz Duality \cite[p.~351, Theorem 8.3]{MR1224675}, capping with the orientation class $\vartheta$ of $M$ induces an isomorphism $\widecheck{H}^1(K_r) \to H_2(M, M\setminus K_r)$. Since $K_r$ is a manifold, we have an isomorphism $\rho\colon \widecheck H^1(K_r)\to H^1(K_r)$ \cite[p. 539, Corollary E.6]{MR1224675}. Furthermore, because $K_r$ is connected, $H_0(K_r) \cong \mathbb{Z}$ is free abelian, so $\operatorname{Ext}(H_0(K_r), \mathbb{Z}) = 0$. The Universal Coefficient Theorem for cohomology then yields an isomorphism $\operatorname{ev}\colon H^1(K_r) \to \operatorname{Hom}(H_1(K_r), \mathbb{Z})$.

    For any $j > i$, consider the following ladder diagram:
    \[
        \begin{tikzcd}[column sep = 2.5em]
            H_1(M\setminus K_i) 
                & {H_2(M,M\setminus K_i)} \arrow[l, "\partial"', "\cong"] 
                & \widecheck H^1(K_i) \arrow[l, "\frown \vartheta"', "\cong"] \arrow[r, "\rho", "\cong"'] 
                & H^1(K_i) \arrow[r, "\operatorname{ev}", "\cong"'] 
                & {\operatorname{Hom}(H_1(K_i),\mathbb Z)} \\
            H_1(M\setminus K_j) \arrow[u] 
                & {H_2(M,M\setminus K_j)} \arrow[l, "\partial", "\cong"'] \arrow[u] 
                & \widecheck H^1(K_j) \arrow[l, "\frown \vartheta", "\cong"'] \arrow[u] \arrow[r, "\rho"', "\cong"] 
                & H^1(K_j) \arrow[u] \arrow[r, "\operatorname{ev}"', "\cong"] 
                & {\operatorname{Hom}(H_1(K_j),\mathbb Z)} \arrow[u]
        \end{tikzcd}
    \]
    where all horizontal maps are isomorphisms. The first vertical map is induced by the inclusion $M \setminus K_j \hookrightarrow M \setminus K_i$, the second vertical map is induced by the inclusion of pairs $(M, M \setminus K_j) \hookrightarrow (M, M \setminus K_i)$, and the remaining vertical maps are induced by the inclusion $K_i \hookrightarrow K_j$. 

    The first square commutes by naturality of the connecting homomorphism, the second square commutes by naturality of the cap product with the orientation class $\vartheta$ \cite[p.~349, Lemma 8.1]{MR1224675}, the third square commutes because $\rho$ is a natural transformation \cite[p. 285]{MR415602}, and the fourth square commutes by naturality of the evaluation homomorphism.

    By Lemma~\ref{lem:inclusionzero}, there exists $j>i$ such that the inclusion-induced map $H_1(K_i)\to H_1(K_j)$ is the zero map. Applying the functor $\operatorname{Hom}(-,\mathbb Z)$, the induced map $\operatorname{Hom}(H_1(K_j),\mathbb Z) \to \operatorname{Hom}(H_1(K_i),\mathbb Z)$ is the zero map. By the commutativity of the diagram and the fact that all horizontal maps are isomorphisms, it follows that the inclusion-induced map $H_1(M\setminus K_j)\to H_1(M\setminus K_i)$ is also the zero map.
\end{proof}
The following lemma can be proved using an argument similar to that of Lemma~\ref{lem:firstcontr}.
\begin{lemma}\label{lem:secondcontr}
    Let $M$ be a connected, oriented, non-compact smooth $3$-manifold without boundary, and let $\{K_i\}$ be a regular exhaustion of $M$. Suppose that $H_2(M) = 0$. Then $H_2(M\setminus K_i)\cong \mathbb Z$ and for every $j > i$, the inclusion-induced map $H_2(M\setminus K_j)\to H_2(M\setminus K_i)$ is an isomorphism.
\end{lemma}

The next lemma shows that the irreducibility hypothesis assumed by Brown and Tucker for all contractible open 3-manifolds is in fact redundant. Its proof is well known, but we include it here for the reader's convenience.

\begin{lemma}\label{lem:irr}
    Let $M$ be a connected, open $3$-manifold such that both $\pi_1(M)$ and $\pi_2(M)$ vanish. Then $M$ is irreducible.
\end{lemma}

\begin{proof}
    Let $S \subseteq M$ be an arbitrary smoothly embedded $2$-sphere. We first establish that $S$ separates $M$. Since $H^1(\mathbb S^2; \mathbb{Z}/2\mathbb{Z}) = 0$, the normal bundle of $S$ in $M$ is trivial, meaning $S$ admits a tubular neighborhood $V \cong \mathbb S^2 \times (-1, 1)$. Setting $U \coloneqq M \setminus S$, we apply the Mayer--Vietoris sequence to the open cover $\{U, V\}$. Since $M$ is connected and simply connected, $\widetilde{H}_0(M) = \widetilde{H}_1(M) = 0$. Furthermore, the intersection $U \cap V \cong \mathbb S^2 \times (-1, 0) \sqcup \mathbb S^2 \times (0, 1)$ consists of two connected components, giving $\widetilde{H}_0(U \cap V) \cong \mathbb{Z}$. The sequence then reduces to $0 \to \mathbb{Z} \to \widetilde{H}_0(U) \oplus 0 \to 0,$ yielding $\widetilde{H}_0(U) \cong \mathbb{Z}$. Consequently, $U$ has exactly two path-components.

    Let $A$ and $B$ denote the closures of these two components, so that $M = A \cup B$ and $A \cap B = S$. Because $\pi_1(M) = 0$, the manifold $M$ is orientable; because it is also non-compact, $H_3(M) = 0$. Additionally, the Hurewicz theorem implies $H_2(M) = 0$. The Mayer--Vietoris sequence for $M = A \cup B$ thus provides the following: $0 \to H_2(S) \to H_2(A) \oplus H_2(B) \to 0.$ Since $H_2(S) \cong \mathbb{Z}$, this yields $H_2(A) \oplus H_2(B) \cong \mathbb{Z}$. 

    We claim that at least one of these components must be compact. Suppose, for a contradiction, that both $A$ and $B$ are non-compact. By Lefschetz duality for non-compact manifolds \cite[p. 260, Exercise 35]{MR1867354}, we have $H_3(A,S) \cong H_c^0(A) = 0$. Similarly, $H_3(B,S) = 0$. The long exact sequence of the pair $(A, S)$ then forces the inclusion-induced map $H_2(S) \to H_2(A)$ to be injective. The same holds for $H_2(S) \to H_2(B)$. This implies that $H_2(A) \oplus H_2(B)$ contains a subgroup isomorphic to $\mathbb{Z} \oplus \mathbb{Z}$, directly contradicting the established isomorphism $H_2(A) \oplus H_2(B) \cong \mathbb{Z}$. Thus, exactly one component---say, $B$---is compact.

    To conclude, we must show that this compact region $B$ is a standard $3$-ball. Note that $B$ is a compact $3$-manifold with boundary $\partial B = S$. Applying the Seifert-van Kampen theorem to the decomposition $M = A \cup B$, the triviality of both $\pi_1(M)$ and $\pi_1(S)$ forces $\pi_1(B) = 0$. By attaching a standard $3$-ball $\mathbb D^3$ to $B$ along their common spherical boundary, we obtain a closed, simply connected $3$-manifold $\widehat{B} = B \cup_S \mathbb D^3$. By Perelman's resolution of the Poincaré conjecture, $\widehat{B}$ is homeomorphic to $\mathbb S^3$. Thus, $B$ is homeomorphic to $\mathbb D^3$. Therefore, every smoothly embedded $2$-sphere in $M$ bounds a $3$-ball, proving that $M$ is irreducible.
\end{proof}

Before proving Theorem~\ref{thm:contra}, we need one final lemma. The proof is straightforward, since every isomorphism $\mathbb{Z}\to \mathbb{Z}$ is determined by whether $1$ is mapped to $1$ or $-1$.
\begin{lemma}\label{lem:proZ}
    Let $\underline{G} = \{G_n, p^m_n\}$ and $\underline{H} = \{H_n, q^m_n\}$ be inverse systems of groups such that $G_n = H_n = \mathbb{Z}$ for all $n \geq 0$. If the bonding maps satisfy $p^m_n = \mathrm{id}$ and $q^m_n$ is an isomorphism for all $m \geq n$, then $\underline{G}$ and $\underline{H}$ are pro-isomorphic.
\end{lemma}

\begin{proof}[Proof of Theorem~\ref{thm:contra}]
    Fix a smooth structure on $\mathcal W$. Let $\{K_i\}$ be a regular exhaustion of $\mathcal W$. By Lemma~\ref{lem:oneend}, $\mathcal W$ has exactly one end. Thus, $\mathcal W\setminus K_i$ is connected for every $i\geq 0$. For every integer $r\geq 0$, consider the inverse system $\underline H_r\coloneqq \{H_r(\mathcal W\setminus K_i)\}$, where for each $j>i$, the bonding map $(p_r)^{j}_i\colon H_r(\mathcal W\setminus K_j)\to H_r(\mathcal W\setminus K_i)$ is induced by inclusion.

    Since $\mathcal W\setminus K_i$ is path-connected, $H_0(\mathcal W\setminus K_i)\cong \mathbb Z$ and $(p_0)^{j}_i$ is an isomorphism for $j>i$ \cite[Theorem 4.14]{MR957919}. By Lemma~\ref{lem:proZ}, Theorem~\ref{proinv}, and Theorem~\ref{constant}, $\underline H_0(\mathcal W)\cong\wp(\underline H_0)\cong \wp(\underline{\mathbb Z})\cong \left.\prod_{n \ge 0} \mathbb Z \right/ \bigoplus_{n \ge 0} \mathbb Z$.
    
    Next, by Lemma~\ref{lem:firstcontr}, Theorem~\ref{proinv}, and Theorem~\ref{constant}, $\underline H_1(\mathcal W)\cong\wp(\underline H_1)\cong \wp(\underline{1})\cong 1$. Similarly, by Lemma~\ref{lem:secondcontr}, Lemma~\ref{lem:proZ}, and Theorem~\ref{constant}, $\underline H_2(\mathcal W)\cong \wp(\underline H_2)\cong \left.\prod_{n \ge 0} \mathbb Z \right/ \bigoplus_{n \ge 0} \mathbb Z$. Moreover, since each $\mathcal W\setminus K_i$ is non-compact, $H_3(\mathcal W\setminus K_i)\cong 1$, and thus, $\underline H_3(\mathcal W)\cong\wp(\underline H_3)\cong \wp(\underline{1})\cong 1$. Therefore, by Corollary~\ref{cor:Rn}, $\underline H_r(\mathcal W)\cong \underline H_r(\mathbb R^3)$ for every $r\geq 0$. 

    Finally, assume that $\mathcal W$ is not homeomorphic to $\mathbb R^3$. We show that $\underline \pi_1(\mathcal W)$ is uncountable and perfect. First, recall that $\underline \pi_1(\mathcal W)$ can be computed using $\wp$ \cite{MR356041}. Because $\mathcal W$ is not homeomorphic to $\mathbb R^3$, by \cite[Corollary 3.3]{MR334225} and Lemma~\ref{lem:irr}, $\underline \pi_1(\mathcal W)$ is non-trivial. Thus, by Theorem~\ref{dichotomy}, $\underline \pi_1(\mathcal W)$ is uncountable. Moreover, since $\underline H_1(\mathcal W)\cong 1$, by Theorem~\ref{thm:hur}, $\underline \pi_1(\mathcal W)$ is perfect.
\end{proof}

\subsection{Proper Homology Groups of Brown's String of \texorpdfstring{$k$}{k}-Spheres}\label{section:stringofspheres}
In the setting of the Brown--Grossman proper fundamental group, the analogue of a pointed circle is the string of circles, denoted $\underline{\mathbb S^1}$. It is defined as the base ray $\underline{*} = [0,\infty)$ with a distinct circle $\mathbb S^1$ attached at each integer $n\geq 0$. In particular, $\underline{\mathbb S^1}$ has exactly one end.

A natural approach to computing the first proper homology $\underline{H}_1(\underline{\mathbb S^1})$ would be to mimic the standard calculation of $H_1(\mathbb S^1)$ via a proper analogue of the Mayer--Vietoris sequence. However, the classical Mayer--Vietoris argument for $H_1(\mathbb S^1)$ relies on covering the circle with two contractible open arcs whose intersection is homotopy equivalent to the disconnected space $\mathbb S^0=\{-1,+1\}$, and then decomposing the homology of this intersection into a direct sum over its path components. Because our proper homology groups are defined exclusively for nice spaces---which are, in particular, connected and non-compact---this strategy encounters an obstruction: it is impossible to cover $\underline{\mathbb S^1}$ by two connected, contractible open subsets whose intersection is also connected. Note that restricting proper homology groups to such spaces aligns with Brown's original strategy; he defined proper homotopy groups only for connected, locally finite simplicial complexes \cite[p. 41]{MR356041}.

Consequently, rather than developing a proper analogue of the Mayer--Vietoris sequence, we calculate the proper homology groups of $\underline{\mathbb S^1}$ directly using the $\wp$-functor. This methodology and the resulting calculations apply equally well to the string of $k$-spheres, $\underline{\mathbb S^k}$. In particular, unlike the classical case, where the homology of $\mathbb S^k$ is computed via the Mayer--Vietoris sequence by induction on $k$, the computation of the $q$-th proper homology of $\underline{\mathbb S^k}$ here does not depend on the $(q-1)$-th proper homology of $\underline{\mathbb S^{k-1}}$.

\begin{theorem}\label{thm:properhomologystring}
    The $q$-th proper homology group of $\underline{\mathbb{S}^k}$ is isomorphic to $\prod_{n \ge 0} \mathbb{Z} \big/ \bigoplus_{n \ge 0} \mathbb{Z}$ for $q=0,k$ and is the trivial group otherwise.
\end{theorem} 

The computation of the $q$-th proper homology of $\underline{\mathbb{S}^k}$ for $q < k$ is straightforward. However, the case $q=k$ is less direct. We first identify an intermediate group $\mathcal{Q}$ isomorphic to $\underline{H}_k(\underline{\mathbb{S}^k})$. We then show that $\mathcal{Q}$ is algebraically compact and shares the same algebraic invariants as the group $\prod_{n \ge 0} \mathbb{Z} \big/ \bigoplus_{n \ge 0} \mathbb{Z}$, which implies they are isomorphic. More precisely, we rely on the classification theory of algebraically compact groups, which states that two such groups are isomorphic if and only if they have isomorphic maximal divisible subgroups, equal $p$-th torsion free number  for every prime $p$, and equal $n$-th Ulm invariants with respect to $p$ for every prime $p$ and integer $n \ge 0$ \cite[Theorem 9.1]{MR357652}. The necessary background on the classification of algebraically compact groups will be provided later. First, we define the intermediate group $\mathcal{Q}$.

\begin{definition}\label{def:RCFM}
    Let $\mathbb{N}_0 \coloneqq \{0, 1, 2, \dots\}$ and let $\operatorname{Mat}_{\mathbb{N}_0 \times \mathbb{N}_0}(\mathbb{Z})$ denote the group of all infinite matrices over $\mathbb{Z}$ indexed by $\mathbb{N}_0 \times \mathbb{N}_0$ under coordinate-wise addition. 

Define the subgroup of row-and-column-finite matrices as
\[
\operatorname{RCFM} \coloneqq \left\{ (m_{i,j}) \in \operatorname{Mat}_{\mathbb{N}_0 \times \mathbb{N}_0}(\mathbb{Z}) \;\middle|\; 
\begin{array}{l} 
\forall i \in \mathbb{N}_0, m_{i,j} \neq 0 \text{ for only finitely many } j, \text{ and} \\ 
\forall j \in \mathbb{N}_0, m_{i,j} \neq 0 \text{ for only finitely many } i 
\end{array} 
\right\}
\]
and define the subgroup of finitely supported matrices as
\[
\operatorname{FSM} \coloneqq \left\{ (m_{i,j}) \in \operatorname{Mat}_{\mathbb{N}_0 \times \mathbb{N}_0}(\mathbb{Z}) \;\middle|\; m_{i,j} \neq 0 \text{ for only finitely many pairs } (i,j) \right\}.
\]
For an $M\in \operatorname{RCFM}$, we denote the quotient class in $\mathcal Q\coloneqq\operatorname{RCFM}/\operatorname{FSM}$ of $M$ by $[M]$.
\end{definition}

We now consider an inverse system $\underline{G}$ that arises in the computation of the $k$-th proper homology of $\underline{\mathbb{S}^k}$. In the following lemma, we show that $\wp(\underline{G})$ is isomorphic to $\mathcal{Q}$.

\begin{lemma}\label{lem:rcfm_isomorphism}
Let $\underline{G} = \{G_n, p_n^m\}_{m \geq n \geq 0}$ be the inverse system of groups where $G_n = \bigoplus_{i=n}^\infty \mathbb{Z}$, and the bonding maps $p_n^{n+1} \colon G_{n+1} \to G_n$ are given by $(x_{n+1}, x_{n+2}, \dots) \mapsto (0, x_{n+1}, x_{n+2}, \dots)$. 
Then there is a group isomorphism from $\wp(\underline G)$ onto the quotient group \( \mathcal Q \).
\end{lemma}
\begin{proof}
    Let $D \coloneqq \bigoplus_{i=0}^\infty \mathbb{Z}$. We identify each group $G_n$ with a subgroup of $D$ via the embedding 
\[
\iota_n \colon G_n \hookrightarrow D, \quad (x_n, x_{n+1}, \dots) \longmapsto (\underbrace{0, \dots, 0}_{n \text{ times}}, x_n, x_{n+1}, \dots).
\]
Under this identification, the sequence of subgroups is descending: $D = G_0 \supseteq G_1 \supseteq G_2 \supseteq \cdots$, and for any $m \ge n \ge 0$, the bonding map $p_n^m \colon G_m \to G_n$ is the inclusion map.

Recall that $\mathscr{S}_{\underline{G}}$ consists of sequences $s=\{g_{k(n)}\}_{n \geq 0}$ such that $g_{k(n)} \in G_{k(n)}$ and $k(n) \to \infty$ as $n \to \infty$. Because all bonding maps are inclusions, two sequences $\{g_{k(n)}\}$ and $\{g'_{l(n)}\}$ in $\mathscr{S}_{\underline{G}}$ are  equivalent if and only if they are eventually equal in $D$; that is, $g_{k(n)} = g'_{l(n)}$ for all sufficiently large $n$. Thus, $\wp(\underline{G}) = \mathscr{S}_{\underline{G}} / {\sim}$ where $\sim$ is the relation of eventual equality. The group operation on $\wp(\underline{G})$ is induced by coordinate-wise addition in $D$.

We construct a homomorphism $\Psi \colon \mathscr{S}_{\underline{G}} \to \operatorname{RCFM}$. Let $s = \{g_{k(n)}\}_{n \geq 0} \in \mathscr{S}_{\underline{G}}$. We define $\Psi(s)$ to be the infinite matrix $M = (m_{n,j})$ whose $n$-th row is the sequence $g_{k(n)} \in D$. 

We must verify that $M \in \operatorname{RCFM}$. 
First, since $g_{k(n)} \in D = \bigoplus_{i=0}^\infty \mathbb{Z}$, it has only finitely many non-zero entries by definition of the direct sum. Thus, every row of $M$ contains finitely many non-zero entries, establishing row-finiteness.

Second, we establish column-finiteness. The condition that the $n$-th row belongs to $G_{k(n)} = \bigoplus_{i=k(n)}^\infty \mathbb{Z}$ means that $m_{n,j} = 0$ for all column indices $j < k(n)$. 
Because the sequence of indices $k(n) \to \infty$ as $n \to \infty$, for any fixed column index $j \ge 0$, there exists an integer $N$ such that $k(n) > j$ for all $n \geq N$, and hence, $m_{n,j} = 0$ for all rows $n \geq N$. This proves that every column of $M$ has only finitely many non-zero entries. Hence, $M \in \operatorname{RCFM}$, and $\Psi$ is well-defined map. 

It is clear that $\Psi$ is a homomorphism since the operations in both $\mathscr{S}_{\underline{G}}$ and $\operatorname{RCFM}$ are coordinate-wise addition.

Next, we show that $\Psi$ is surjective. Let $M = (m_{n,j}) \in \operatorname{RCFM}$. Let $r_n \in D$ be the $n$-th row of $M$. Because $M$ is column-finite, no single column can contain infinitely many non-zero entries. Therefore, the column index of the very first non-zero entry in row $r_n$ cannot be bounded by any fixed integer $J$ as $n \to \infty$ (otherwise, at least one of the columns $0, 1, \dots, J$ would contain infinitely many non-zero entries, violating column-finiteness). Thus, this index must tend to infinity. Let $k(n)$ be the index of this first non-zero entry (if $r_n = 0$, set $k(n) = n$). Then $k(n) \to \infty$. Furthermore, because $m_{n,j} = 0$ for all $j < k(n)$, we have $r_n \in \bigoplus_{i=k(n)}^\infty \mathbb{Z} = G_{k(n)}$. Thus, $s = \{r_n\}_{n \geq 0} \in \mathscr{S}_{\underline{G}}$ and $\Psi(s) = M$, establishing surjectivity.

Finally, we determine the kernel of the composition $\pi \circ \Psi \colon \mathscr{S}_{\underline{G}} \to \mathcal Q$, where $\pi$ is the canonical projection. A sequence $s = \{g_{k(n)}\} \in \mathscr{S}_{\underline{G}}$ maps to the identity in $\mathcal Q$ if and only if the corresponding matrix $M = \Psi(s)$ lies in $\operatorname{FSM}$. A matrix $M$ is in $\operatorname{FSM}$ if and only if it has finitely many non-zero entries in total. Since each row is finitely supported, this is equivalent to $M$ having only finitely many non-zero rows. This occurs if and only if $g_{k(n)} = 0$ for all sufficiently large $n$, which is precisely the condition that $s \sim \{0\}$ in $\mathscr{S}_{\underline{G}}$.

Because $\mathscr{S}_{\underline{G}}$ is a group under coordinate-wise addition and $\sim$ is compatible with this addition, $s \sim s' \iff s - s' \sim \{0\}$. Thus, the equivalence classes of $\sim$ coincide with the cosets of $\ker(\pi \circ \Psi)$. By the First Isomorphism Theorem, $\Psi$ descends to an isomorphism from $\wp(\underline G)$ onto $\mathcal Q$.
\end{proof}

Now, we show that $\mathcal Q$ and $\left.\prod_{m\geq 0} \mathbb Z\right/\bigoplus_{m\geq 0} \mathbb Z$ are isomorphic. Our proof will be based on the classification theorem of algebraically compact groups \cite[Theorem 9.1]{MR357652}. We first recall some standard concepts regarding abelian groups following our reference \cite[\S1 and \S2]{MR534230}.

Let $A$ be an abelian group with identity element $0$. For any integer $n \geq 1$, let $nA \coloneqq \{na \mid a \in A\}$ and $A[n] \coloneqq \{x \in A \mid nx = 0\}$. The \emph{$\mathbb Z$-adic topology} on $A$ is formed by declaring the collection of subgroups $\{nA\}_{n=1}^{\infty}$ to be a base of neighborhoods around $0$. Observe that $\{n! A\}_{n=1}^{\infty}$ is also a base of neighborhoods around $0$. It can be shown that the $\mathbb{Z}$-adic topology on $A$ is Hausdorff if and only if the \emph{first Ulm subgroup} $U(A) \coloneqq \bigcap_{n\in\mathbb{N}} nA=\bigcap_{n\in\mathbb{N}} n!A$ is trivial \cite[p. 343]{MR534230}.

Let $\{x_k\}$ be a sequence in $A$. We say $\{x_k\}$ \emph{converges to a limit} $L \in A$ if for every integer $n \geq 1$, there exists an index $N$ such that $x_k - L \in n! A$ for all $k \ge N$. Any two limits of a convergent sequence differ by an element in the first Ulm subgroup $U(A)$. We say $\{x_k\}$ is \emph{Cauchy} if for every integer $n \geq 1$, there exists an index $N$ such that $x_i - x_j \in n! A$ for all $i> j \ge N$. The group $A$ is called \emph{$\mathbb Z$-adically complete} if every Cauchy sequence converges to a limit within $A$.

Recall that the abelian group $A$ is said to be \emph{torsion-free} if for all $x \in A$ and all integers $n \ge 1$, the equation $nx = 0$ implies $x = 0$. It will be called \emph{divisible} if for every element $x \in A$ and every integer $n\geq 1$, there exists an element $y \in A$ such that $ny = x$. It is easy to show that a torsion-free divisible abelian group admits a vector space structure over $\mathbb{Q}$. The \emph{divisible rank} of a torsion-free divisible abelian group is defined as the dimension of that group viewed as a vector space over $\mathbb{Q}$. 

For the abelian group $A$, there exists a unique divisible subgroup, denoted $A_d$, which is \emph{maximal} in the sense that $A_d$ contains every other divisible subgroup of $A$ (in fact, $A_d$ is the subgroup of $A$ generated by the union of all the divisible subgroups of $A$). Observe that $A_d\subseteq U(A)$ since $A_d$ is a divisible group.

A subgroup $B$ of the abelian group $A$ is called \emph{pure} in $A$ if $nB = nA \cap B$ for all integers $n\geq 1$. The group $A$ is called \emph{algebraically compact} if $A$ is a direct summand of $X$ whenever $A$ is a pure subgroup of $X$.

For a prime $p$, let $A_0(p)$ denote the dimension of $A/(pA + A_t)$ as a vector space over the field $\mathbb{F}_p$. We call $A_0(p)$ the \emph{$p$-th torsion-free number of $A$}. Here, $A_t$ denotes the torsion subgroup of $A$ (i.e., the subgroup of all elements of finite order), and $pA$, as usual, denotes the subgroup $\{pa \mid a \in A\}$.

For a prime $p$ and an integer $n\geq 0$, let $A_p(n)$ denote the dimension of $(p^nA)[p] / (p^{n+1}A)[p]$ as a vector space over $\mathbb{F}_p$. We call $A_p(n)$ the \emph{$n$-th Ulm invariant of $A$ with respect to $p$}.

Now, we are ready to prove several lemmas.

\begin{lemma}\label{lem:algcom}
    If $A$ is torsion-free and $\mathbb Z$-adically complete, then $A$ is algebraically compact.
\end{lemma}
\begin{proof}
    By \cite[Satz 2.2]{MR534230}, it is enough to show that $U(A)=A_d$. Since $A_d\subseteq U(A)$ always holds and $A_d$ is the maximal divisible subgroup, it suffices to show that $U(A)$ is divisible. Pick any $x\in U(A)$ and fix a positive integer $m$. Since $x\in mA$, there exists $y\in A$ such that $my=x$. We must now prove that this specific $y$ is an element of $U(A)$. To belong to $U(A)$, $y$ must be an element of $kA$ for every positive integer $k$. This holds because for every positive integer $k$, we have $x \in U(A)\subseteq (mk)A$, meaning there exists $z\in A$ such that $(mk)z=x$. Equating the two expressions for $x$ yields $m(y-kz)=0$, which implies $y=kz$ since $A$ is torsion-free. Therefore, $y \in kA$, completing the proof.
\end{proof}

\begin{lemma}\label{lem:Qcompact}
    The quotient group $\mathcal{Q}= \operatorname{RCFM}/\operatorname{FSM}$ is torsion-free and $\mathbb Z$-adically complete. Thus, $\mathcal Q$ is algebraically compact.
\end{lemma}
\begin{proof}
    It is straightforward to show that $\mathcal{Q}$ is torsion-free. Thus, by Lemma~\ref{lem:algcom}, it is enough to prove that $\mathcal{Q}$ is $\mathbb{Z}$-adically complete.

    Let $\{[Y_n]\}_{n=1}^\infty$ be a Cauchy sequence in $\mathcal Q$. Thus, for every integer $k > 0$, there exists an index $N_k$ such that $[Y_i] - [Y_j] \in k!\mathcal{Q}$ for all $i, j \ge N_k$. Let $n_0=0$, and define $n_1<n_2<n_3<\cdots$ recursively by $n_k = \max(N_k, n_{k-1} + 1)$. Let $M_k\coloneqq Y_{n_k}$. Then for every $k\geq 1$, we have $[M_{k+1}]-[M_k]= [Y_{n_{k+1}}] - [Y_{n_k}]\in k!\mathcal Q$, which implies $M_{k+1}-M_k= k!B_k+ F_k$ for some $B_k\in \operatorname{RCFM}$ and $F_k\in \operatorname{FSM}$.

    Define $M_k'$ inductively as follows: $M'_1 \coloneqq M_1$. If $M_k'$ is defined, then let $M'_{k+1} \coloneqq M'_k + k! B_k$. Then inductively we can show that $[M'_{k}] = [M_{k}]=[Y_{n_k}]$ for all $k\geq 1$.

    Define $M''_k \coloneqq M'_k - M'_1$ for all $k\geq 1$. Thus, $M''_1 = 0$ and $M''_{k+1} - M''_k = k! B_k$. Hence, $M''_k = \sum_{m=1}^{k-1} (M''_{m+1} - M''_m) = \sum_{m=1}^{k-1} m! B_m$.\medskip

   Define $X \in \operatorname{Mat}_{\mathbb N_0 \times \mathbb N_0}(\mathbb Z)$ as follows: $X_{i,j} \coloneqq \sum_{m=1}^{\min(i, j)} m! (B_m)_{i,j}$. We first show that $X \in \operatorname{RCFM}$. Fix row $i$. For all columns $j \ge i$, the entry $X_{i,j}$ is equal to the $(i,j)$-th entry of $\sum_{m=1}^{i} m! B_m \in \operatorname{RCFM}$. Hence, row $i$ has only finitely many non-zero entries. By a symmetric argument, the columns also have only finitely many non-zero entries. Thus $X \in \operatorname{RCFM}$.

We claim that $\{[M_k'']\}$ converges to $[X]$. We analyze the difference $X - M_k''$ entry-by-entry:$$(X - M_k'')_{i,j} = \sum_{m=1}^{\min(i, j)} m! (B_m)_{i,j} - \sum_{m=1}^{k-1} m! (B_m)_{i,j}$$

\emph{First Region:} Where $\min(i, j) \ge k$. The first $k-1$ terms cancel, leaving a sum starting at $m=k$. Every term is a multiple of $k!$. Factoring out $k!$ yields a matrix $Z^{(k)}$. Thus, $Z^{(k)}_{i,j}$ is $0$ if $\min(i,j)<k$, and is $\frac{1}{k!}\sum_{m=k}^{\min(i,j)}m! (B_m)_{i,j}$ if $\min(i, j) \ge k$. Because any given row $i \ge k$ involves only a finite linear combination of matrices $B_k, \dots, B_i \in \operatorname{RCFM}$, row $i$ of $Z^{(k)}$ has finitely many non-zero entries. By symmetry, $Z^{(k)} \in \operatorname{RCFM}$. \medskip

\emph{Second Region:} Where $\min(i, j) < k$. The remaining terms are $- \sum_{m=\min(i, j)+1}^{k-1} m! (B_m)_{i,j}$. Let $F^{(k)}$ contain exactly these entries. Thus, $F^{(k)}_{i,j}$ is $- \sum_{m=\min(i, j)+1}^{k-1} m! (B_m)_{i,j}$ if $\min(i, j) < k$, and is $0$ if $\min(i, j) \geq k$. Because this region is restricted to the first $k-1$ rows and columns, and relies exclusively on a finite sum of matrices $B_1, \dots, B_{k-1}$ (each belonging to $\operatorname{RCFM}$), $F^{(k)}$ contains a  finite number of non-zero entries overall. Hence, $F^{(k)} \in \operatorname{FSM}$.\medskip

Therefore, $X - M_k'' = k! Z^{(k)} + F^{(k)}$. In $\mathcal{Q}$, this becomes:$$[X]+[M_1]-[M_k] = [X] - ([M_k]-[M_1']) = [X]-[M_k''] \in k!\mathcal{Q}.$$Thus, the subsequence $\{[M_k]\}$ of $\{[Y_n]\}$ converges to $[X]+[Y_{n_1}]$. Because a Cauchy sequence with a convergent subsequence must converge to the same limit, $\{[Y_n]\}$ also converges to $[X]+[Y_{n_1}]$, proving $\mathcal{Q}$ is $\mathbb Z$-adically complete.
\end{proof}

\begin{lemma}\label{lem:divQ}
    The $\mathbb Q$-vector space $\mathcal Q_d$ is isomorphic to the direct sum of $\mathfrak c$ copies of $\mathbb Q$, where $\mathfrak c$ denotes the cardinality of the continuum.
\end{lemma}
\begin{proof}
    Because the cardinality of $\mathcal Q$ is at most $\aleph_0^{\aleph_0} = \mathfrak{c}$, it is enough to find $\mathfrak c$ linearly independent elements in $\mathcal Q_d$. Let $\mathcal{A}$ be a family of almost disjoint infinite subsets of the natural numbers $\mathbb{N}$ such that the cardinality of $\mathcal A$ is $\mathfrak c$. Here, two infinite subsets $A$ and $B$ of $\mathbb N$ are said to be \emph{almost disjoint} if their intersection is finite. 

    For each set $A \in \mathcal{A}$, construct a diagonal matrix $M_A$ whose $(k,k)$-th entry is defined as $k!$ if $k \in A$, and $0$ if $k \notin A$. Because $M_A$ is diagonal, $M_A \in \operatorname{RCFM}$.

    For any set $A \in \mathcal{A}$ and any integer $n > 0$, construct a corresponding matrix $X_{A,n}$ whose $(k,k)$-th entry is $k!/n$ for $k \in A$ with $k \ge n$, and $0$ otherwise. Then, $X_{A,n}\in\operatorname{RCFM}$ and $M_A - n X_{A,n}\in \operatorname{FSM}$. Thus, $n[X_{A,n}] = [M_A]$, proving that every $[M_A]$ is an element of $\mathcal{Q}_d$.

    We now check the $\mathbb{Q}$-linear independence of the family $\{[M_A]\}_{A \in \mathcal{A}}$. Assume a finite linear combination equals zero in $\mathcal{Q}$:
\[ c_1[M_{A_1}] + c_2[M_{A_2}] + \dots + c_m[M_{A_m}] = [0] \]
where $c_i \in \mathbb{Q}$ and $A_i$ are distinct sets in $\mathcal{A}$. By clearing denominators, we may assume without loss of generality that $c_i \in \mathbb{Z}$. Then $S \coloneqq \sum_{i=1}^m c_i M_{A_i}\in\operatorname{FSM}$. Choose a specific index $j \in \{1, 2, \dots, m\}$. Since $S$ has only finitely many non-zero entries, for all sufficiently large $k\in A_j\setminus \cup_{i\neq j}A_i$, the $(k,k)$-th entry of $S$ vanishes, i.e., $c_j k! = 0$, forcing $c_j = 0$. By symmetry, all coefficients $c_i = 0$. This confirms that the uncountably infinite family $\{[M_A]\}_{A \in \mathcal{A}}$ of divisible elements is $\mathbb Q$-linearly independent.
\end{proof}

\begin{lemma}\label{lem:torsionfree}
   Let $p$ be a prime. Then, the $p$-th torsion-free number $\mathcal{Q}_0(p)$ of $\mathcal{Q}$ is equal to $\mathfrak{c}$.
\end{lemma}
\begin{proof}
    Define the groups of row-and-column-finite matrices $\operatorname{RCFM}(\mathbb{F}_p)$ over $\mathbb{F}_p$ and finitely supported matrices $\operatorname{FSM}(\mathbb{F}_p)$ over $\mathbb{F}_p$ following Definition~\ref{def:RCFM}. Let $\pi\colon \operatorname{RCFM} \to \operatorname{RCFM}(\mathbb{F}_p)$ be the element-wise modulo $p$ reduction map. It is straightforward to check that $\pi$ is a well-defined surjective group homomorphism.
    
    Let $q\colon \operatorname{RCFM}(\mathbb{F}_p) \to \operatorname{RCFM}(\mathbb{F}_p)/\operatorname{FSM}(\mathbb{F}_p)$ be the canonical projection map. Define $\varphi \coloneqq q \circ \pi$. Then $\varphi$ is also a surjective group homomorphism. We now determine the kernel of $\varphi$. By definition, a matrix $M \in \ker(\varphi)$ if and only if the reduced matrix $\pi(M)$ belongs to $\operatorname{FSM}(\mathbb{F}_p)$. We claim that $p\operatorname{RCFM} + \operatorname{FSM} = \ker(\varphi)$.
    
    To prove this claim, let $M = (m_{i,j}) \in \ker(\varphi)$. Then the index set $S_M \coloneqq \{(i,j) \mid m_{i,j} \not\equiv 0 \pmod p\}$ is finite. Define a matrix $F$ such that $F_{i,j} = m_{i,j}$ for all $(i,j) \in S_M$, and $F_{i,j} = 0$ otherwise. Then, $F \in \operatorname{FSM}$. Thus, every entry of $M - F \in \operatorname{RCFM}$ is an integer multiple of $p$. Factoring out $p$, we write $M - F = pK$ for some $K \in \operatorname{RCFM}$. Thus, $M = pK + F$. Therefore, $\ker(\varphi) \subseteq p\operatorname{RCFM} + \operatorname{FSM}$. 
    
    Conversely, if $M = pK + F$ for some $K \in \operatorname{RCFM}$ and $F \in \operatorname{FSM}$, then $\pi(M) = \pi(pK) + \pi(F) = \pi(F) \in \operatorname{FSM}(\mathbb{F}_p)$, and hence $M \in \ker(\varphi)$. 
    
    Combining these inclusions yields $p\operatorname{RCFM} + \operatorname{FSM} = \ker(\varphi)$. By the First Isomorphism Theorem,
    \[
        \frac{\operatorname{RCFM}}{\ker(\varphi)} \cong \operatorname{Im}(\varphi) \implies \frac{\operatorname{RCFM}}{p\operatorname{RCFM} + \operatorname{FSM}} \cong \frac{\operatorname{RCFM}(\mathbb{F}_p)}{\operatorname{FSM}(\mathbb{F}_p)}.
    \]
    
    The group $\mathcal{Q}$ is torsion-free, meaning $\mathcal{Q}_t = 0$. Thus, $\mathcal{Q}_0(p) =  \dim_{\mathbb{F}_p} (\mathcal{Q} / p\mathcal{Q})$. Observe that
    \[
        \mathcal{Q}/p\mathcal{Q} = \frac{\operatorname{RCFM} / \operatorname{FSM}}{p(\operatorname{RCFM} / \operatorname{FSM})} = \frac{\operatorname{RCFM} / \operatorname{FSM}}{(p\operatorname{RCFM}+\operatorname{FSM})/\operatorname{FSM}} \cong \frac{\operatorname{RCFM}}{p\operatorname{RCFM} + \operatorname{FSM}} \cong \frac{\operatorname{RCFM}(\mathbb{F}_p)}{\operatorname{FSM}(\mathbb{F}_p)}.
    \]
    
    The second equality follows from the fact that for an abelian group $A$, if $B$ is a subgroup of $A$, and $n$ is a non-negative integer, then $n(A/B) = \{na + B \mid a \in A\} = (nA + B)/B$.
    
    The proof of Lemma~\ref{lem:divQ} can be mimicked to show that $\dim_{\mathbb{F}_p}(\operatorname{RCFM}(\mathbb{F}_p)/\operatorname{FSM}(\mathbb{F}_p)) = \mathfrak{c}$. Thus, $\mathcal{Q}_0(p) = \mathfrak{c}$. 
\end{proof}
\begin{theorem}\label{thm:QBS}
    The group $\mathcal{Q}$ is isomorphic to $\left.\prod_{m\geq 0} \mathbb{Z}\right/\bigoplus_{m\geq 0} \mathbb{Z}$.
\end{theorem}
\begin{proof}
    Let $G \coloneqq \left.\prod_{m\geq 0} \mathbb{Z}\right/\bigoplus_{m\geq 0} \mathbb{Z}$. Then $G$ is torsion-free. By \cite{MR108529}, $G$ is algebraically compact. An argument similar to that given in the proof of Lemma~\ref{lem:divQ} shows that $G_d$ is isomorphic to the direct sum of $\mathfrak{c}$ copies of $\mathbb{Q}$. Moreover, an argument similar to that given in the proof of Lemma~\ref{lem:torsionfree} shows that for every prime $p$, there exists an isomorphism $G/pG \cong \left.\prod_{m\geq 0} \mathbb{F}_p\right/\bigoplus_{m\geq 0} \mathbb{F}_p$, and thus, the $p$-th torsion-free number $G_0(p)$ of $G$ is equal to $\mathfrak{c}$. Since every subgroup of a torsion-free group is also torsion-free, for every integer $n \geq 0$ and every prime $p$, the $n$-th Ulm invariant of $G$ with respect to $p$ is trivial; i.e., $G_p(n) = \dim_{\mathbb{F}_p}\left((p^nG)[p] / (p^{n+1}G)[p]\right) = 0$. 

    By Lemma~\ref{lem:Qcompact}, $\mathcal{Q}$ is algebraically compact and torsion-free. Lemma~\ref{lem:divQ} shows that $\mathcal{Q}_d$ is isomorphic to the direct sum of $\mathfrak{c}$ copies of $\mathbb{Q}$. Moreover, by Lemma~\ref{lem:torsionfree}, for every prime $p$, we have $\mathcal{Q}_0(p) = \mathfrak{c}$. By reasoning similar to that given in the previous paragraph, for every integer $n \geq 0$ and every prime $p$, the $n$-th Ulm invariant of $\mathcal{Q}$ is trivial; i.e., $\mathcal{Q}_p(n) = \dim_{\mathbb{F}_p}\left((p^n\mathcal{Q})[p] / (p^{n+1}\mathcal{Q})[p]\right) = 0$.

    Therefore, by \cite[Theorem 4.3 and Theorem 9.1]{MR357652}, $G$ is isomorphic to $\mathcal{Q}$.
\end{proof}

\begin{proof}[Proof of Theorem~\ref{thm:properhomologystring}]
    Let $\mathbb S^k$ denote the $k$-sphere, pointed at $t_0$, and define $X \coloneqq \underline{\mathbb S^k}$. We can write $X$ as $\underline{*} \times \{t_0\} \cup \bigcup_{i=0}^\infty \{i\} \times \mathbb S^k$. For each non-negative integer $n$, define $C_n \coloneqq [0, n] \times \{t_0\} \cup \bigcup_{i=0}^n \{i\} \times \mathbb S^k$. Then $\{C_n\}$ is an efficient exhaustion of $X$. Observe that $X \setminus C_n$ is path-connected and strongly deformation retracts onto the infinite wedge of $k$-spheres $\bigvee_{i=n+1}^\infty \mathbb S^k$. Thus, $H_r(X \setminus C_n)$ equals $\mathbb Z$ if $r=0$, equals $\bigoplus_{i=n+1}^\infty \mathbb Z$ if $r=k$, and is zero if $r \notin \{0,k\}$.
    
    Let $n \geq 0$. The inclusion-induced map $H_0(X \setminus C_{n+1}) \to H_0(X \setminus C_n)$ is the identity map on $\mathbb Z$, and the inclusion-induced map $H_k(X \setminus C_{n+1}) \to H_k(X \setminus C_n)$ is the inclusion $\bigoplus_{i=n+2}^\infty \mathbb{Z} \hookrightarrow \bigoplus_{i=n+1}^\infty \mathbb{Z}$ given by $(x_{n+2}, x_{n+3}, \ldots) \mapsto (0, x_{n+2}, x_{n+3}, \ldots)$. Finally, if $r \notin \{0,k\}$, the inclusion-induced map $H_r(X \setminus C_{n+1}) \to H_r(X \setminus C_n)$ is the zero map.
    
    Thus, by Theorem~\ref{constant}, $H_0^\mathsf{alg}(X) = \wp(\underline{\mathbb Z}) \cong \prod_{m \ge 0} \mathbb{Z} \big/ \bigoplus_{m \ge 0} \mathbb{Z}$. By Lemma~\ref{lem:rcfm_isomorphism} and Theorem~\ref{thm:QBS}, $H_k^\mathsf{alg}(X) \cong \mathcal Q \cong \prod_{m \ge 0} \mathbb{Z} \big/ \bigoplus_{m \ge 0} \mathbb{Z}$. Finally, $H_r^\mathsf{alg}(X)$ is trivial for $r \notin \{0,k\}$.
\end{proof}

\section*{Acknowledgments}
\addcontentsline{toc}{section}{Acknowledgments}
The first author acknowledges support from the Institute Postdoctoral Fellowship at IIT Bombay.

\section*{Declaration of Generative AI}
\addcontentsline{toc}{section}{Declaration of Generative AI}
During the preparation of this manuscript, the authors used Google Gemini 3.1 Pro to refine the text, verify calculations, and find relevant literature. The authors independently verified all theorem statements and proofs, and explicitly declare that no mathematical ideas or proof structures were generated by AI. The authors take full responsibility for the final content.

\bibliography{bibliography}
\bibliographystyle{alphaurl}
\addcontentsline{toc}{section}{References}
\vspace{2em} 

\noindent
\textsc{Sumanta Das} \\
\small Department of Mathematics, Indian Institute of Technology Bombay, India \\
\textit{Email address:} \href{mailto:sumantadas@alum.iisc.ac.in}{\texttt{sumantadas@alum.iisc.ac.in}}

\vspace{1.5em} 

\noindent
\textsc{Rekha Santhanam} \\
\small Department of Mathematics, Indian Institute of Technology Bombay, India \\
\textit{Email address:} \href{mailto:reksan@math.iitb.ac.in}{\texttt{reksan@math.iitb.ac.in}}
\end{document}